\documentclass{daj}
\usepackage{amssymb, amsmath,latexsym, amsthm}
\usepackage[font=small,labelfont=bf]{caption}
\usepackage{subcaption}
\usepackage{sidecap}
\usepackage{bbm}
\usepackage{graphicx}
\usepackage{wrapfig}
\usepackage{mathrsfs,verbatim,pdfsync,color}

\newcommand{\nobracket}{}
\newcommand{\nocomma}{}

\newenvironment{itemizedot}{\begin{itemize} }{\end{itemize}}

\newcommand{\assign}{:=}

\newcommand{\tmem}[1]{{\em #1\/}}
\newcommand{\tmop}[1]{\ensuremath{\operatorname{#1}}}

\newtheoremstyle{dotless}{}{}{\itshape}{}{\bfseries}{}{ }{}
\theoremstyle{dotless}
\newtheorem{proposition}{Proposition}[section]
\newtheorem{corollary}{Corollary}[section]
\newtheorem{theorem}{Theorem}[section]
\newtheorem{lemma}{Lemma}[section]

\newcommand{\DDD}{{\mathcal{D}}}

\newcommand{\BBB}{{\mathcal{B}}}

\newcommand{\DD}{{\mathbb{D}}}
\newcommand{\RR}{{\mathbb{R}}}
\newcommand{\CC}{{\mathbb{C}}}

\newcommand{\VV}{{\mathbb{V}}}

\renewcommand{\Cap}{{\text{Cap}}}

\theoremstyle{definition}
\DeclareMathOperator{\vvar}{Var}
\DeclareMathOperator{\supp}{supp}

\DeclareMathOperator{\capp}{Cap}

\dajAUTHORdetails{%
  title = {Bi-parameter Potential Theory and Carleson Measures for the Dirichlet Space on the Bidisc}, 
  author = {Nicola Arcozzi, Pavel Mozolyako, Karl-Mikael Perfekt, and Giulia Sarfatti},
  plaintextauthor = {Nicola Arcozzi, Pavel Mozolyako, Karl-Mikael Perfekt, Giulia Sarfatti},
  runningtitle = {Bi-parameter Potential theory and Carleson Measures}, 
  keywords = {Dirichlet space on the bidisc, trace inequality, Carleson measures, strong capacitary inequality, bitree},
}   

\dajEDITORdetails{%
   year={2023},
   number={22},
   received={16 June 2022},   
   published={29 December 2023},  
   doi={10.19086/da.91187},       
}   

\begin{document}

\begin{frontmatter}[classification=text]


\author[nicola]{Nicola Arcozzi\thanks{Partially supported by the grants INDAM-GNAMPA 2017 "Operatori e disuguaglianze integrali in spazi con simmetrie" and PRIN 2018 "Variet\`{a} reali e complesse: geometria, topologia e analisi armonica"}}
\author[pavel]{Pavel Mozolyako\thanks{Partially supported by project 22-11-00071 by the Russian Science Foundation (results of Sections 3.1 and 3.3), by the Ministry of Science and Higher Education of Russian Federation, agreement 075-15-2021-602}}
\author[kalle]{Karl-Mikael Perfekt\thanks{Partially supported by grant 334466 of the Research Council of Norway, “Fourier Methods and Multiplicative Analysis”.}}
\author[giulia]{Giulia Sarfatti\thanks{Partially supported by by INDAM-GNSAGA, by the 2014 SIR grant "Analytic Aspects in Complex and Hypercomplex Geometry", by Finanziamento Premiale FOE 2014 "Splines for accUrate NumeRics: adaptIve models for Simulation Environments", and by PRIN 2022 "Interactions between Geometric Structures and Function Theories" of the Italian Ministry of Education (MIUR).}}

\begin{abstract}
We characterize the Carleson measures for the Dirichlet space on the bidisc, hence also its multiplier space. Following Maz'ya and Stegenga, 
the characterization is given in terms of a capacitary condition. We develop the foundations of a bi-parameter potential theory on the bidisc
and prove a Strong Capacitary Inequality. In order to do so, we have to overcome the obstacle that the Maximum Principle fails in the bi-parameter theory.
\end{abstract}
\end{frontmatter}

\section{Introduction}

\noindent{\bf Notation.} We denote by $\DD$ the unit disc $\DD=\{z\in\CC:\ |z|<1\}$ in the complex plane and by $\partial\DD$ its boundary. We write 
$A\lesssim B$ ($A \gtrsim B$) if there is a constant independent on the variables on which $A$ and $B$ depend (which might be numbers, variables, sets...) such that 
$A\le CB$ ($CA\ge B$ respectively), and $A\approx B$, if $A\lesssim B$ and $A \gtrsim B$.

\vskip1cm
In $1979$, Alice Chang \cite{chang1979}, extending a foundational result of Carleson \cite{carleson1962} in one variable, 
characterized the Carleson measures for the bi-harmonic Hardy space of the bidisc, that is, those measures $\mu$ on $\DD^2$ such that the identity operator
boundedly maps $h^2(\DD)\otimes h^2(\DD)$ into $L^2(\mu)$. 
In \cite{carleson1974} Carleson had previously shown that the bi-parameter theory presents very peculiar features.
Results in bi-parameter potential theory and harmonic analysis are scattered in the literature. See however Chapter 3 in the monograph \cite{MS2013}, and the references mentioned at p. 124, and the lectures notes \cite{Gundy1980}.
At the same time, Stegenga \cite{stegenga1980} characterized the Carleson measures for the holomorphic Dirichlet space on the unit disc.
Following standard use in complex function theory, we say that a measure $\mu$ is  a {\it Carleson measure for the 
Hilbert function space $H$} if $H$ continuously embeds into $L^2(\mu)$. 

Carleson measures proved to be a central notion in the analysis of holomorphic spaces, as they intervene in the characterization of multipliers, 
interpolating sequences, and Hankel-type forms, in Corona-type problems, in the characterization of exceptional sets at the boundary, and more.
In this article we characterize the Carleson measures for the Dirichlet space on the bidisc, and we obtain as a consequence a characterization 
of its multiplier space.

As the Dirichlet space is defined by a Sobolev norm, it is not surprising that Stegenga's 
characterization is given in terms of a potential theoretic object, set capacity, and that the proof relies on deep results from Potential Theory, such as the Strong Capacitary Inequality. The main effort in this article is developing a 
bi-parameter potential theory which is rich enough to state and prove the characterization theorem. There are obstructions to doing
so, which we will illustrate below.

Other approaches to similar problems have been suggested in the past. The closest result is Eric Sawyer's characterization of the 
weighted inequalities for the bi-parameter Hardy operator \cite{sawyer1985}. Sawyer's extremely clever combinatorial-geometric argument does not seem
to work in our context, or at least we were not able to make it work. The difficulty lies in the fact that Sawyer deals with 
the product of two segments, while we work with the product of two hyperbolic discs. For similar reasons, we were not able 
to extend the \textit{good-lambda} argument in \cite{ars2002} to the bi-parameter case.
The simple approach via maximal functions in \cite{ars2007}
could work, if knowledge concerning weighted maximal bi-parameter functions was more developed. 
The difficulty is that, contrary to the linear case, bi-parameter
maximal functions do not always satisfy weighted $L^2$ inequalities. 
We refer to  \cite{Cairoli1971} and \cite{Gundy1980} for early, detailed accounts of two-parameter processes and their associated maximal functions, and \cite{marcin} for seminal results on bi-parameter maximal functions.
In the one parameter case, Carleson measures for the Dirichlet space can also be characterized using a Bellman function argument \cite{AHMV18}. At the moment, 
however, the Bellman function technique, having at its heart stochastic optimization for martingales, does not work in the two time-parameter martingale
theory underlying bi-parameter Potential Theory. The scheme of dualizing the embedding to translate the problem from bi-parameter--holomorphic to bi-parameter--dyadic has been borrowed from  \cite{ars2007, ars2008}.

We now state our results more precisely. Let $\DDD(\DD)$ be the Dirichlet space on the unit disc $\DD=\{z\in\CC:\ |z|<1\}$; that is,
the space of the functions $f(z)=\sum_{n=0}^\infty\widehat{f}(n)z^n$, analytic in $\DD$, 
such that the norm
\begin{equation}\label{introdef}
\|f\|_\DDD^2:=\sum_{n=0}^\infty(n+1)\left|\widehat{f}(n)\right|^2=\|f\|_{H^2(\DD)}^2+\frac{1}{\pi}\int_\DD\left|f^\prime(z)\right|^2dA(z)<\infty. 
\end{equation}
The Dirichlet space on the bidisc $\DD^2$ can be temporarily defined as $\DDD(\DD^2):=\DDD(\DD)\otimes\DDD(\DD)$. The main aim of this article is proving the following.
\begin{theorem}\label{intromain}
 Let $\mu\ge0$ be a Borel measure on $\overline\DD^2$, the closure of the bidisc. Then the following are equivalent.
 \begin{itemize}
  \item[(I)] There is a constant $C_1>0$ such that 
  \begin{equation}\label{introimbed}
   \sup_{0\le r<1}\int_{r\DD^2}|f|^2d\mu\le C_1\|f\|_{\DDD(\DD^2)}^2, \quad f \in \DDD(\DD^2);
  \end{equation}
  \item[(II)] There is a constant $C_2>0$ such that for all $n\geq 1$ and for all choices of arcs $J^1_1,\dots,J^1_n$ and $J^2_1,\dots,J^2_n$ on $\partial\mathbb{D}$, we have that
  \begin{equation}\label{introcap}
   \mu\left(\bigcup_{k=1}^n S\left( J^1_k\times J^2_k\right) \right)\le C_2 \capp_{(\frac{1}{2}, \frac{1}{2})}\left(\bigcup_{k=1}^n  J^1_k\times J^2_k \right).
  \end{equation}
 \end{itemize}
Moreover, the constants $C_1$ and $C_2$ are comparable independently of $\mu$.
\end{theorem}
Here $S\left( J^1\times J^2\right)=S\left( J^1\right)\times S\left(J^2\right)$ 
is the Carleson box in $\overline{\DD}^2$ based on  $J^1\times J^2$ and $\capp_{(\frac{1}{2}, \frac{1}{2})}$ is the canonical extension of $\frac{1}{2}$-Bessel capacity from
linear potential theory to bi-parameter potential theory, which will be defined later in the article.
It can be estimated from above and below by the capacity which is
naturally associated with the reproducing kernel of 
$\DDD(\DD^2)$, and several other versions of capacity. In the one parameter case, Stegenga (see \cite[Theorem 4.2]{stegenga1980}) exhibited examples of measures $\lambda$ satisfying the appropriate version of (\ref{introcap}) for single arcs, but not for arbitrary unions of arcs. It is easy to see that the measure $\mu
=\lambda\otimes m$, where $m$ is arclength, provides an example where (\ref{introcap}) holds for rectangles, yet fails for arbitrary unions of them.

We can, for the moment informally, view (\ref{introimbed}) 
as the boundedness of the imbedding $Id:\DDD(\DD^2)\hookrightarrow L^2(\mu)$.
A measure satisfying (\ref{introimbed}) is a \it Carleson measure \rm for $\DDD(\DD^2)$ and we define $$[\mu]_{CM}:=\|Id\|_{\BBB(\DDD(\DD^2),L^2(\mu))}^2$$
as its {\it Carleson measure norm}. Actually, the result has a stronger version, which we will prove, where in the left hand side of (I),
$f$ is replaced by its {\it radial variation}, whose main contribution is given by
\begin{equation}\label{radialvariation}
 \vvar_{12}f(\zeta,\xi)=\int_0^{1}\int_0^{1}\left|\partial_{zw}f(r\zeta,s\xi )\right|drds,\ (\zeta,\xi)\in\overline\DD^2
\end{equation}
The full definition appears in Section \ref{SS:2.5}. In particular, this allows us to recover, on the bidisc, Beurling's result \cite{Beurling40} on exceptional sets 
for the radial variation of functions in the Dirichlet space.

A function $b$ holomorphic in $\DD^2$ is a \it multiplier \rm of 
$\DDD(\DD^2)$ if multiplication times $b$, $M_b:f\mapsto bf$, is bounded on $\DDD(\DD^2)$. The operator norm of $M_b$ is, by definition, 
the multiplier norm of $b$. From Theorem \ref{intromain} we deduce a characterization of multipliers.
\begin{theorem}\label{intromultiplier}
Let $b$ be holomorphic in $\DD^2$ and define the measure $d\mu_b:=|\partial_{zw}b(z,w)|^2dA(z)dA(w)$.
There exist positive constants $C_1,C_2$ such that:
\begin{equation}\label{introstegenga}
\left\|M_b\right\|_{\BBB(\DDD(\DD^2))}^2\approx [\mu_b]_{CM}+\|b\|_{H^\infty}^2+\sup_{w \in \DD} [|\partial_z b(\cdot, w)|^2dA(\cdot)]_{CM(\mathcal{D}(\mathbb{D}))}
+\sup_{z \in \DD} [|\partial_w b(z, \cdot)|^2dA(\cdot)]_{CM(\mathcal{D}(\mathbb{D}))},
\end{equation}
where $[\cdot]_{CM(\mathcal{D}(\mathbb{D}))}$ denotes the Carleson measure norm with respect to the Dirichlet space in the disc.
\end{theorem}
Up to this point the results exactly match those obtained by Stegenga in 1980. The proof, however, 
has to overcome a series of obstructions. The most prominent one is that the potentials of positive 
measures do not satisfy a maximum principle: the supremum of the potential $\VV^\mu$ of a measure $\mu$ can be much larger than its 
supremum on the support of $\mu$. Essentially, this is due to the fact that the product of distance functions on metric spaces $X$ and $Y$
is usually very far from itself being a distance on $X\times Y$. To oversimplify a large body of knowledge, a far reaching Potential Theory can be developed if the
reciprocal of the defining kernel behaves like a distance, satisfying the triangle inequality up to a constant factor.
The product of two copies of such kernels, however, fails to satisfy this property.

In order to isolate the essential difficulties, it is convenient to transfer the holomorphic problem to a dyadic one.
In the sequel, $T$ will denote a set of vertices labelling dyadic arcs in $\partial\DD$, which might be seen as an oriented dyadic tree with respect 
to the relation $I\subset J$: we identify a dyadic arc with a vertex of $T$, and an oriented edge of $T$ can be thought of as an inclusion relation $J\subset I$ of a a dyadic arc in the dyadic arc of twice its length. 
To each arc $I$ we associate 
the usual Carleson box $S(I)$ and the Whitney square $Q(I)$ consisting of its upper half, $Q(I)=S(I)\setminus\cup_{J\subsetneqq I}S(J)$.

It is convenient to set up a different notation for objects related to $\DD$ and objects related to $T$: we will assign to a dyadic arc $I$ a label
$\alpha_I$ in $T$ and, vice versa, each $\alpha$ in $T$ will be the label of some dyadic arc $I_\alpha = I(\alpha)$. The \it root \rm of $T$ is $o$, where
$I_o=\partial\DD$.
Also, $T$ has a natural boundary $\partial T$,
endowed with a metric which makes $\overline{T} = T \cup \partial T$ into a compact space. We can think of an element $\zeta$ in $\partial T$ as the label for an infinite, decreasing 
sequence of dyadic arcs: 
$P(\zeta):=\{I_n\}_{n=0}^\infty$, with  $I_n\supset I_{n+1}$, and $\left|I_n\right|/2\pi=2^{-n}$. We write $\alpha\le\beta$ if $I_\alpha\subseteq I_\beta$.
This convention is different from the one used in \cite{ars2002}.
To each $\alpha$ in $\overline{T}$ we associate the subset
$P(\alpha)=\{\beta:\ \alpha\le\beta\le o\}$. To each $\alpha$ we also associate the region $S(\alpha)=\{\zeta\in\overline{T}:\ \alpha\in P(\zeta)\}$.

The geometric objects defined in the disc have natural counterparts in the bidisc: 
we define $S(J^1\times J^2)=S(J^1)\times S(J^2)$, $Q(J^1\times J^2)=Q(J^1)\times Q(J^2)$, and so on. 
The bitree $T^2=T\times T$ labels dyadic rectangles. If $\alpha=(\alpha_z,\alpha_w)\in T\times T$, we associate to it 
$J(\alpha)=J(\alpha_z)\times J(\alpha_w)$.
A basic fact is that $T^2$, identified with the set of the dyadic rectangles $J^1\times J^2$, 
does not have a tree structure with respect to inclusion.
We can move positive Borel measures from  $\overline{\DD}^2$ to  $\overline{T}^2$, 
in a way which will be made precise later: to each Borel measure $\mu$ on $\overline{\DD}^2$ we associate a unique Borel measure
$\Lambda^\ast\mu$ on $\overline{T}^2$ and, vice-versa, to each Borel measure $\nu$ on $(\partial T)^2$ we can associate a unique Borel measure
$\Lambda_\ast\nu$ on $(\partial\mathbb{D})^2$ (here we restrict ourselves to measures supported on $(\partial T)^2$, since we do not really need to transplant the measure on the rest of the bitree to the bidisc). Essentially, $\Lambda_\ast\nu$ associates in natural way a measure on the \it distinguished boundary \rm $(\partial\mathbb{D})^2$ of $\DD^2$ to the restriction of $\nu$ to the  
\it distinguished boundary \rm $(\partial T)^2$ of $T^2$. 
Conversely, $\Lambda^\ast\mu$ concentrates the measure of $Q(I_\alpha)\subseteq\DD^2$ into $\alpha\in T^2$, it essentially preserves the measures 
on the distinguished boundaries, and acts in a mixed way on the remaining parts of the boundaries.  
The precise definitions of $\Lambda_\ast$ and $\Lambda^\ast$ are slightly technical and will be given later.

We define a natural \it bi-parameter Hardy operator \rm $\mathbb{I}$
acting on functions $\varphi:\overline{T}^2\to\RR$,
\begin{equation}\label{bihardy}
 \mathbb{I}\varphi(\zeta):=\sum_{\alpha\in P(\zeta)}\varphi(\alpha),
\end{equation}
provided the sum makes sense. This is certainly the case if $\varphi\ge0$, which is what we will assume throughout the article. We define the operator $ \mathbb{I}$ on real valued functions in order to have for its adjoint the usual definition.
The operator $\mathbb{I}$ is analogous to the bi-parameter Hardy operator studied by Eric Sawyer in \cite{sawyer1985}, 
and it is the bi-parameter version of the operator $I$ introduced in \cite{ars2002}.
Dually, we have the operator $\mathbb{I}^\ast$ acting on (a priori, signed) Borel measures on $\overline{T}^2$, $\mathbb{I}^\ast\mu(\alpha):=\mu(S(\alpha))$. 
These simple operators encode all relevant information.
\begin{theorem}\label{dyadization}
 Let $\mu$ be a positive, Borel measure on $\overline{\DD}^2$. Then
 \begin{equation}\label{dyadizationnorm}
  C_1[\mu]_{CM}\le   \left\|\mathbb{I}\right\|_{\BBB(\ell^2(T^2),L^2(\Lambda^\ast\mu))}^2\le  C_2[\mu]_{CM},
 \end{equation}
where $C_1$ and $C_2$ are universal constants. Moreover, we can replace $[\mu]_{CM}$ by the larger quantity
$$
\sup_{\left\|f\right\|_{\DDD(\DD^2)}=1}\left\|\vvar f\right\|_{L^2(\overline{\mathbb{D}}^2, d\mu)}^2,
$$
where the {\bf radial variation} of $f$, $\vvar f$, is defined in \eqref{e:575}, and has \eqref{radialvariation} as its leading summand. 
\end{theorem}
With this theorem at hand, the characterization of the Carleson measures for $\DDD(\DD^2)$ can be reduced to that of estimating the quantity 
$\left\|\mathbb{I}\right\|_{\BBB(\ell^2(T^2),L^2(\nu))}^2$, and it suffices to consider the case of nonnegative functions. In the linear case, 
in \cite{ars2002} this was done in terms of a Kerman--Sawyer \cite{kermansawyer1986} type testing condition. We will instead follow the capacitary 
path introduced by Maz'ya \cite{mazya1972}, then Adams \cite{adams1976}, in proving sharp trace inequalities, 
which was transplanted by Stegenga \cite{stegenga1980} to the holomorphic world.


Following a general scheme \cite[Sections 2.3-2.5]{ah1996}, the operators $\mathbb{I}$ and $\mathbb{I}^\ast$ can interpreted in terms of a
potential theory on $(\overline{T})^2\times T^2$. More precisely, we consider the potential kernel 
$k(\zeta,\alpha)=\chi(\alpha\in P(\zeta))=\chi(\zeta\in S(\alpha))$. Following \cite{ah1996}, if $E\subseteq \overline{T}^2$, then we define \textit{discrete bi-logarithmic capacity} by
\begin{equation}\label{capacitydef}
 \capp(E)=\inf\left\{\|\varphi\|_{\ell^2(T^2)}^2:\ \mathbb{I}\varphi(\zeta)\ge1 \text{ if }\zeta\in E\right\}.
\end{equation}
The \it trace inequality \rm we wish to prove is
\begin{theorem}\label{traceinequality}
There are positive constants $C_1,C_2$ such that,  if $\nu$ is a Borel measure on $\overline{T}^2$, then
 \begin{equation}\label{eqtrace}
  C_1\sup_{n\ge1; \alpha_1,\dots,\alpha_n\in T^2}\frac{\nu\left(\cup_{j=1}^n S(\alpha_j)\right)}{\capp\left(\cup_{j=1}^n \partial S(\alpha_j)\right)}
  \le \left\|\mathbb{I}\right\|_{\BBB(\ell^2(T^2),L^2(\nu))}^2
  \le C_2\sup_{n\ge1; \alpha_1,\dots,\alpha_n\in T^2}\frac{\nu\left(\cup_{j=1}^n S(\alpha_j)\right)}{\capp\left(\cup_{j=1}^n \partial S(\alpha_j)\right)}.
 \end{equation}
\end{theorem}
Theorem \ref{traceinequality} follows by a standard argument from a Strong Capacitary Inequality of Adams type \cite{adams1976}.
\begin{theorem}\label{capacitarystrong}
 There is a constant $C>0$ such that, whenever $\varphi:T^2\to[0,+\infty)$,
 \begin{equation}\label{capstrong}
   \int_0^\infty\capp(\{\zeta\in(\partial T)^2:\ \mathbb{I}\varphi(\zeta)>\lambda\})d\lambda^2\le C\|\varphi\|_{\ell^2(T^2)}^2.
 \end{equation}
\end{theorem}
The Strong Capacitary Inequality is standard when $\capp$ is the capacity associated to a radially decreasing kernel. This is the case, 
with sufficient approximation, in the linear case of a tree $T$, where $\Cap$ is associated with a Bessel-like kernel. 
See for example \cite{ah1996} for the general theory, \cite{kaltonverbitsky1999} for the relation with semilinear equations, and \cite{arsw2014} for 
case of trees and metric spaces. The literature on Bessel-like kernels is vast and we just mention a few titles.
However, our capacity is associated with the tensor product of two Bessel-like kernels, which is itself very different from a Bessel kernel, in
the same way that the tensor product of two distance functions is typically not a distance function. In particular, the Maximum Principle
for potentials fails completely.
\begin{proposition}\label{countermaxprinc} 
 For any $\lambda \geq 1$ there exists a measure $\mu\ge0$ on $(\partial{T})^2$ such that $\VV^\mu\le1$ on $\supp\mu$, but $\VV^\mu(\zeta_0)\geq\lambda$ at some 
 $\zeta_0\in (\partial{T})^2$. Moreover, we can take $\mu=\mu_E$ to be the equilibrium measure of some set $E\subseteq \partial{T}^2$.
\end{proposition}
We go around this difficulty by proving an estimate showing that the set where the potential is large has small capacity. 
This is the main novelty of this article.
\begin{theorem}\label{almostmaximum}
 There is a positive constant $C$ such that, if $\mu$ is the equilibrium measure of a Borel subset of $(\partial T)^2$ and $\lambda\ge1$, one has:
 \begin{equation}\label{smallcapacity}
  \capp(\{\zeta\in(\partial T)^2:\ \VV^\mu(\zeta)>\lambda\})\le C\frac{\|\mathbb{I}^\ast\mu\|_{\ell^2(T^2)}^2}{\lambda^3}.
 \end{equation}
\end{theorem}
Theorem \ref{almostmaximum} is reduced to a new problem in linear Potential Theory on the tree $T$, which is solved.
We have considered the Dirichlet space on the bidisc only. Some parts of our argument easily extend to more general environments; for instance, Theorem \ref{intromultiplier} extends to polydiscs. The dyadization scheme in Theorem \ref{dyadization}
can be similarly extended to polytrees with any number of factors. We believe that the results can be extended using different 
powers in the definition of the Dirichlet norm: $1<p<\infty$ is a natural choice. Weights could be taken into consideration. 
As this article enters unexplored territory, we have preferred to consider its most basic object: the unweighted 
Dirichlet space ($p=2$) on the bidisc. Our results can also be used to prove trace inequalities in other contexts, using different dyadization schemes. We will return to this in other works. We have made an effort to provide all details of all proofs. 
We will point out, however, which parts of our arguments are, in our opinion, standard, and which are new.
\subsubsection*{Layout of the article} The paper is organized as follows. In Section~\ref{S:2} we prove Theorems \ref{intromain} and \ref{intromultiplier} -- modulo the Strong Capacitary Inequality -- as well as several other statements mentioned above. We introduce the discrete model of the bidisc, move the problem there and solve the discrete version. The approach is adapted from one-dimensional techniques, and mostly follows \cite{ars2008} and \cite{arsw2014}. We only present the general line of the argument, postponing technical details to the Appendix. Section \ref{S:5} contains substantially new results. There we prove the Strong Capacitary Inequality on the bitree, which is a crucial part of our method. Section \ref{conclusions} contains some concluding reflections. In the Appendix, Section \ref{S:A}, we collect auxiliary results used or mentioned before, as well as some counterexamples.



\section{Proof of Theorem 1: discretization and 'soft' argument}\label{S:2}
We start with introducing some notation and describe the properties of $\mathcal{D}(\mathbb{D}^2)$ that we use later.\\

Given a holomorphic function $f(z_1,z_2)=\sum_{m,n\geq0}a_{mn}z_1^mz_2^n$ on $\mathbb{D}^2$ we let
\begin{equation}\notag
\|f\|^2_{\mathcal{D}(\mathbb{D}^2)} = \sum_{m,n\geq 0}|a_{mn}|^2(m+1)(n+1).
\end{equation}
This norm can also be written as follows,
\begin{equation}\notag
\begin{split}
&\|f\|^2_{\mathcal{D}(\mathbb{D}^2)} = \frac{1}{\pi^2}\int_{\mathbb{D}^2}|\partial_{z_1,z_2}f(z_1,z_2)|^2\,dA(z_1)\,dA(z_2)+\sup_{0\leq r_2<1}\frac{1}{2\pi^2}\int_{\partial\mathbb{D}}\int_{\mathbb{D}}|\partial_{z_1}f(z_1,r_2e^{it})|^2\,dA(z_1)\,dt +\\
&\sup_{0\leq r_1<1}\frac{1}{2\pi^2}\int_{\mathbb{D}}\int_{\partial\mathbb{D}}|\partial_{z_2}f(r_1e^{it},z_2)|^2\,ds\,dA(z_2) + \sup_{0\leq r_1,r_2<1}\frac{1}{4\pi^2}\int_{\partial\mathbb{D}}\int_{\partial\mathbb{D}}|f(r_1e^{is},r_2e^{it})|^2\,ds\,dt = \\
&\|f\|^2_* + \;\textup{other terms},
\end{split}
\end{equation}
where $\|f\|_*$ is a seminorm which is invariant under biholomorphisms of the bidisc. In what follows however we use an equivalent norm, arising from the representation $\mathcal{D}(\mathbb{D}^2) = \mathcal{D}(\mathbb{D})\otimes\mathcal{D}(\mathbb{D})$. For $f\in \textrm{Hol}(\mathbb{D})$ let
\begin{equation}\label{e:21}
\|f\|^2_{\mathcal{D}} := \frac{1}{\pi}\int_{\mathbb{D}}|f'(z)|^2\,dA(z) + C_0|f(0)|^2,
\end{equation}
where $C_0>0$ is a constant to be chosen shortly.
It is classical fact that the Dirichlet space on the unit disc is a Reproducing Kernel Hilbert Space (RKHS) \cite{AM02}, and, consequently,  $\mathcal{D}(\mathbb{D}^2)$ is one as well. The reproducing kernel $K_z,\; z\in\mathbb{D}^2$, generated by $\|\cdot\|_{\mathcal{D}}$, is
\begin{equation}\label{e:22}
K_z(w) = \left(C_1 + \log\frac{1}{1-\bar{z}_1w_1}\right)\left(C_1 + \log\frac{1}{1-\bar{z}_2w_2}\right),\quad z,w\in\mathbb{D}^2
\end{equation}
a product of reproducing kernels for $\mathcal{D}(\mathbb{D})$ in respective variables. Here $C_1 = 1/C_0$. 
Hence $K_z$ enjoys the following important property
\begin{equation}\label{e:23}
\Re K_z(w) \approx |K_z(w)|,\quad z,w\in\mathbb{D}^2,
\end{equation}
if we take $C_1$ to be large enough (see Lemma \ref{l:A.0.1}).

This Section is organized as follows. First we use duality arguments and the RKHS property of $\mathcal{D}(\mathbb{D}^2)$ to replace the Carleson Condition \eqref{introimbed} with something more tangible, first doing so for measures supported strictly inside the bidisc. Then, in Section \ref{SS:2.3}, we construct the discrete approximation of the bidisc -- the bitree -- and we introduce the basics of (logarithmic) Potential Theory there. Next we move all the objects from the bidisc to the bitree, obtaining an equivalent discrete characterization of Carleson measures for $\mathcal{D}(\mathbb{D}^2)$, see Theorem \ref{dyadization}. Adopting the approach by Maz'ya we reduce the problem to verifying a certain property of the bilogarithmic potential --- the Strong Capacitary Inequality. Assuming this inequality holds, see Theorem \ref{capacitarystrong} and its proof in Section \ref{S:5}, we prove Theorem \ref{intromain} for measures inside the bidisc in Section \ref{SS:2.6}. In Section \ref{SS:2.5} we extend this result to measures supported on $\overline{\mathbb{D}}^2$ and replace the function $f$ in \eqref{introimbed} with its radial variation. Finally, in Section \ref{SS:2.7} we describe the multipliers of $\mathcal{D}(\mathbb{D}^2)$.

\subsection{Duality approach}\label{SS:2.2}
Let $\mu \ge0$ be a finite Borel measure on $\overline{\mathbb{D}}^2$, and assume for a time being that $\mu(\partial\mathbb{D}^2) = 0$. We consider the general case later in Section \ref{SS:2.5}.

To modify \eqref{introimbed} we first observe that the embedding $Id:\mathcal{D}(\mathbb{D}^2)\rightarrow L^2(\mathbb{D}^2,d\mu)$ is bounded if and only if the adjoint operator 
\[
\Theta := Id^*: L^2(\mathbb{D}^2,d\mu)\rightarrow \mathcal{D}(\mathbb{D}^2)
\]
is bounded as well. To proceed we need to know the action of $\Theta$, which is provided by the RKHS property of $\mathcal{D}(\mathbb{D}^2)$. Indeed, given a function $g\in L^2(\mathbb{D}^2,\,d\mu)$ we have 
\[
(\Theta g)(z) = \langle\Theta g, K_z\rangle_{\mathcal{D}(\mathbb{D}^2)} = \langle g, K_z\rangle_{L^2(\mathbb{D}^2,\,d\mu)} =
\int_{\mathbb{D}^2}g(w)\overline{K_z(w)}\,d\mu(w).
\]
Then boundedness of $\Theta$ means that for any $g\in L^2(\mathbb{D}^2,\,d\mu)$ one has
\begin{equation}\label{e:24}
\begin{split}
&\|g\|^2_{L^2(\mathbb{D}^2,\,d\mu)} \gtrsim \|\Theta g\|^2_{\mathcal{D}(\mathbb{D}^2)} = \langle\Theta g,\Theta g\rangle_{\mathcal{D}(\mathbb{D}^2)} = \langle g, \Theta g\rangle_{L^2(\mathbb{D}^2,\,d\mu)} =\\
&\int_{\mathbb{D}^2}\int_{\mathbb{D}^2}g(z)\overline{g(w)}K_z(w)\,d\mu(z)\,d\mu(w).
\end{split}
\end{equation}
Taking $g$ to be real and non-negative, we see that \eqref{e:24} becomes
\begin{equation}\label{e:25}
\|g\|^2_{L^2(\mathbb{D}^2,\,d\mu)} \gtrsim \int_{\mathbb{D}^2}\int_{\mathbb{D}^2}g(z)g(w)\Re K_z(w)\,d\mu(z)\,d\mu(w),
\end{equation}
since, clearly, $\int_{\mathbb{D}^2}\int_{\mathbb{D}^2}g(z)g(w)\Im K_z(w)\,d\mu(z)\,d\mu(w) =0$. On the other hand, if $g$ is an arbitrary function in $L^2(\mathbb{D}^2,\,d\mu)$, then \eqref{e:25} applied to $|g|$ gives
\[
\|g\|^2_{L^2(\mathbb{D}^2,\,d\mu)} \gtrsim \int_{\mathbb{D}^2}\int_{\mathbb{D}^2}|g(z)||g(w)|\Re K_z(w)\,d\mu(z)\,d\mu(w) \approx \int_{\mathbb{D}^2}\int_{\mathbb{D}^2}|g(z)||g(w)||K_z(w)|\,d\mu(z)\,d\mu(w),
\]
and we are back at \eqref{e:24}. To summarize, the measure $\mu$ on $\mathbb{D}^2$ is Carleson for $\mathcal{D}(\mathbb{D}^2)$ if and only if \eqref{e:25} holds for any non-negative function in $L^2(\mathbb{D}^2,\,d\mu)$. Observe that unlike \eqref{introimbed}, condition \eqref{e:25} does not mention the Sobolev norm of the Dirichlet space, nor its analytic structure, which makes it a much more viable candidate for discretization.

\subsection{The bitree}\label{SS:2.3}
Let $T$ be a rooted directed uniform (each vertex has the same amount of children) infinite binary tree -- in what follows we call such an object \textit{a dyadic tree}. There is a natural order relation provided by the arborescent structure: given two points $\alpha,\beta$ in the vertex set $\mathcal{V}(T)$ we say that $\alpha \leq \beta$, if $\beta$ lies on the unique geodesic connecting $\alpha$ and the root $o$, we also write $\beta \geq\alpha$, if $\alpha\leq\beta$. Here and when this is less cumbersome, we identify a geodesic with a sequence of vertices in the obvious way. Consider an infinite directed sequence $\omega$ starting at the root, i.e. $\omega = \{\omega^j\}_{j=0}^{\infty}$ with $\omega^0 = o$, $\omega^{j+1} \leq\omega^{j}$ and $\omega^j,\omega^{j+1}$ are connected by an edge. The set of all such sequences is (as usual) called the \textit{boundary} of $T$ and denoted by $\partial T$, we also let $\overline{T} = T\cup \partial T$. For any $\omega = \{\omega^j\}_{j=0}^{\infty}\in \partial T$ we always say that $\omega \leq \omega^j,\; j\geq0$.

 If $\alpha,\beta \in \mathcal{V}(T)\cup\partial T$, then there there exists a unique point $\gamma \in \mathcal{V}(T)\cup\partial T$ that is \textit{the least common ancestor} of $\alpha$ and $\beta$, we denote it by $\alpha\wedge\beta$. Namely, we have that $\gamma\geq\alpha,\;\gamma\geq\beta$, and if there is another point $\tilde{\gamma}$ satisfying these relations, then $\tilde{\gamma}\geq\gamma$. In other words, $\gamma$ is the first intersection point of the geodesics connecting $\alpha$ and $\beta$ to the root. The total amount of common ancestors of $\alpha$ and $\beta$ is denoted by $d_T(\alpha\wedge\beta)$. Note that $d_T(\alpha\wedge\beta) = \textrm{dist}_T(\alpha\wedge\beta,o)+1$, where $\textrm{dist}_T$ is the usual graph distance on $T$. $d_T$ may be infinite; for instance, $d_T(\omega\wedge\omega) = \infty$ whenever $\omega\in\partial T$. The \textit{predecessor set} (with respect to the geometry of $\overline{T}$) of a point $\alpha\in\mathcal{V}(T)\cup\partial T$ is
\[
\mathcal{P}(\alpha) = \{\beta\in \mathcal{V}(T)\cup\partial T:\; \beta\geq\alpha\}.
\]
In particular, every point is its own predecessor. The \textit{successor set} is
\[
\mathcal{S}(\beta) := \{\alpha\in\mathcal{V}(T)\cup\partial T:\; \beta\in\mathcal{P}(\alpha)\},\quad \beta\in \overline{T}.
\]
Clearly $d_T(\alpha\wedge\beta) = |\mathcal{P}(\alpha\wedge\beta)|$.
\\

A dyadic tree is a well known and often used way of discretizing the unit disc, so it is reasonable to assume that the discrete analogue of the bidisc $\mathbb{D}\times\mathbb{D}$ should look like a Cartesian product of two dyadic trees. With this observation in mind we define the graph $T^2$ which we call \textit{bitree} as follows. Given two dyadic trees $T_x$ and $T_y$ we let
\[
\mathcal{V}(T^2) := \mathcal{V}(T_x)\times \mathcal{V}(T_y),
\]
to be the vertex set of $T^2$. In other words a vertex $\alpha\in \mathcal{V}(T^2)$ is a pair $(\alpha_x,\alpha_y)$ of vertices of two (identical) coordinate trees. Given two vertices $\alpha,\beta\in \mathcal{V}(T^2)$ we connect them by an edge whenever $\alpha_x=\beta_x$ and $\alpha_y$ and $\beta_y$ are connected by an edge in $T_y$, or, vice versa, $\alpha_y=\beta_y$ and $\alpha_x,\beta_x$ are neighbours. The order relation (and hence the direction of $T^2$) are induced from the coordinate trees, we say that $\alpha \leq \beta$ if and only if $\alpha_x\leq\beta_x$ \textit{and} $\alpha_y\leq\beta_y$. As in one-dimensional case we define the boundary $\partial T^2 := T_x\times\partial T_y\bigcup \partial T_x\times T_y\bigcup \partial T_x\times\partial T_y$. The last part $\partial T_x\times\partial T_y$ we call the \textit{distinguished boundary} of the bitree (similar to the bidisc setting) and denote by $(\partial T)^2$. We also let $\overline{T}^2 = T^2\bigcup\partial T^2$. As before, we define predecessor and successor sets of a vertex $\alpha = (\alpha_x,\alpha_y)$ using the same notation
\[
\mathcal{P}(\alpha) = \mathcal{P}(\alpha_x)\times\mathcal{P}(\alpha_y),\; \mathcal{S}(\alpha) = \mathcal{S}(\alpha_x)\times\mathcal{S}(\alpha_y).
\]
Sometimes, to avoid confusion, we specify the dimension by writing $\mathcal{S}_{T}(\alpha)$ for a point $\alpha$ in the tree $T$, and $\mathcal{S}_{T^2}(\alpha)$ for a point $\alpha$ in the bitree. We use the same convention for predecessor sets.

Similar to one-dimensional setting we denote the number, possibly infinite, of common ancestors of $\alpha$ and $\beta$ by $d_{T^2}(\alpha\wedge\beta)$, where $\alpha\wedge\beta = (\alpha_x\wedge\beta_x,\alpha_y\wedge\beta_y)$ is the (unique) least common ancestor of $\alpha$ and $\beta$. We have 
\[d_{T^2}(\alpha\wedge\beta) = d_T(\alpha_x\wedge\beta_x)\cdot d_T(\alpha_y\wedge\beta_y) = |\mathcal{P}(\alpha_x\wedge\beta_x)||\mathcal{P}(\alpha_y\wedge\beta_y)| = |\mathcal{P}(\alpha\wedge\beta)|.
\]

In what follows we do not really need to consider the edges of $T$, since we are studying the unweighted Dirichlet space. From now on we do not distinguish the graph and its vertex set, in that we write $\alpha \in \overline{T}^2$ ($\overline{T}$) instead of $\alpha\in \mathcal{V}(T^2)\bigcup\partial T^2$ ($\mathcal{V}(T)\bigcup\partial T$). We also write $d_T(\alpha_x)$ and $d_{T^2}(\beta)$ instead of $d_T(\alpha_x\wedge\alpha_x)$ and $d_{T^2}(\beta\wedge\beta)$.\\

A natural way to interpret the dyadic tree is to identify its vertices with the approximating intervals for the classical Cantor set on the unit interval. Namely, consider the ternary Cantor set $E = \bigcap_{j=0}^{\infty}E_j$, where $E_0 = [0,1]$, and $E_k$ consists of $2^k$ closed intervals of length $3^{-k}$. Then each point of $T$ corresponds to a unique interval in $E_j$ (or, more precisely, to its midpoint), and, similarly, $\partial T$ maps to $E$. In the same vein the points of $T^2$ correspond to ternary rectangles (Cartesian products of centerpoints of intervals in $E_j$ and $E_k$). In particular, $(\partial T)^2$ can be identified with $E^2$. Note that this means that $\overline{T}$ can be embedded into $\mathbb{R}^2$, and, consequently, $\overline{T}^2$ into $\mathbb{R}^4$. We will use this embedding to define a Potential Theory on the bitree. We also observe that $T^2$ no longer has unique geodesics; it is not acyclic like $T$. However, $T^2$ still \textit{does not have directed cycles}.

\subsection{Potential theory on the bitree}\label{SS:2.35}
We start by defining a metric on $\overline{T}^2$: given $\alpha,\beta\in \overline{T}^2$ we let
\begin{equation}\label{e:31}
\delta(\alpha,\beta) := 2^{-d_T(\alpha_x\wedge\beta_x)} + 2^{-d_T(\alpha_y\wedge\beta_y)} - \frac12\left(2^{-d_T(\alpha_x)} + 2^{-d_T(\beta_x)} + 2^{-d_T(\alpha_y)} + 2^{-d_T(\beta_y)}\right).
\end{equation}
This metric makes $\overline{T}^2$ into a compact space. The properties of a metric are easily verified for the function $\delta_x:\overline{T}\times \overline{T}\to[0,\infty)$, $\delta_x(\alpha_x,\beta_x)=
2^{-d_T(\alpha_x\wedge\beta_x)} - \frac12\left(2^{-d_T(\alpha_x)} + 2^{-d_T(\beta_x)}\right)$, and $\delta$ is the sum of two copies of such distance, one on each factor of $\overline{T}\times\overline{T}$.

Denote by $\mathbb{M}$ the (open) bitree $T^2$ equipped with the counting measure $\nu_c$ (so that $\nu_c(\{\alpha\}) = 1,\; \alpha\in T^2$). We define a kernel $G:\mathbb{R}^4\times\mathbb{M}\rightarrow \mathbb{R}_+$ to be $G(\alpha,\beta) := \chi_{\mathcal{S}_{\beta}}(\alpha)$, where $\alpha\in \overline{T}^2\subset\mathbb{R}^4$ (here we consider $\overline{T}^2$ to be a compact subset of $\mathbb{R}^4$), $\beta\in T^2$ and $\mathcal{S}_{\beta} := \{\gamma\in\overline{T}^2:\; \gamma\leq\beta\}$ is the $\overline{T}^2$-successor set of $\beta$. It is easy to verify that $G$ is lower semicontinuous on $\overline{T}\subset\mathbb{R}^4$ in first variable, and measurable on $\mathbb{M}$ in second variable. Extending kernels, functions, and measures from $\overline{T}^2$ to $\mathbb{R}^4$, by letting them be zero outside $\overline{T}^2$, we are squarely in the context of Adams and Hedberg \cite[Chapter 2]{ah1996}. We thus have a well-defined potential theory on the bitree. We refer to \cite[Chapter 2]{ah1996} for the general theory, while recalling some of its main features below. \par
Given a non-negative Borel measure $\mu$ on $\overline{T}^2$ (which by extension is Borel on $\mathbb{R}^4$) and a non-negative $\nu_c$-measurable function $f$ on $\mathbb{M}$ we let
\begin{subequations}\label{e:32}
\begin{eqnarray}
\label{e:32.1}& (\mathbb{I}f)(\alpha) := \int_{\mathbb{M}}G(\alpha,\beta)f(\beta)\,d\nu_c(\beta) = \sum_{\gamma\geq\alpha}f(\gamma),\\
\label{e:32.2}& (\mathbb{I}^*\mu)(\beta) := \int_{\overline{T}^2}G(\alpha,\beta)\,d\mu(\alpha) = \int_{\mathcal{S}(\beta)}\,d\mu(\alpha).
\end{eqnarray}
\end{subequations}
Observe that a measure supported on $T^2$ and a non-negative function there are pretty much the same objects --- a collection of masses assigned to the points of the bitree. The Potential Theory generated by these two operators leads us to the notions of bilogarithmic potential 
\begin{equation}\label{e:33}
\mathbb{V}^{\mu} := (\mathbb{I}\mathbb{I}^*)(\mu)
\end{equation}
and capacity
\begin{equation}\label{e:334}
\capp E := \inf\left\{\int f^2\,d\nu_c:\; f\geq0,\,(\mathbb{I}f)(\alpha)\geq1,\; \alpha\in E\right\},
\end{equation}
for Borel set $E\subset \overline{T}^2$.
Given two Borel measures $\mu,\nu\geq0$ on $\overline{T}^2$ we define their mutual energy to be
\begin{equation}\label{e:337}
\mathcal{E}[\mu,\nu] := \int_{\overline{T}^2}\mathbb{V}^{\mu}\,d\nu = \int_{\overline{T}^2}\mathbb{V}^{\nu}\,d\mu = \sum_{\alpha\in T^2}(\mathbb{I}^*\mu)(\alpha)(\mathbb{I}^*\nu)(\alpha),
\end{equation}
the last two equalities following from Tonelli's theorem. When $\mu=\nu$ we write $\mathcal{E}[\mu]$ instead, and we call it \textit{the energy of $\mu$}. Given a Borel set $E\subset \overline{T}$ there exists a uniquely defined \textit{equilibrium measure} $\mu_E\geq0$ that generates the minimizer in \eqref{e:334}, so that
\[
\capp E = \int_{T^2}(\mathbb{I}^*\mu_E)^2\,d\nu_C = \mathcal{E}[\mu_E] = \mu_E(E),
\]
see \cite{ah1996}. If $E$ is a compact set, then one also has $\supp\mu_E\subset E$.\\

\subsection{From the bidisc to the bitree: measures supported inside}\label{SS:2.4}
Now that the bitree has been defined, we can move all the objects from the bidisc here. In doing that we first assume that the measures in question are supported on $\mathbb{D}^2$ (so they have no mass on the boundary).\\
\begin{figure}[h]
\centering
\includegraphics[width=0.7\textwidth]{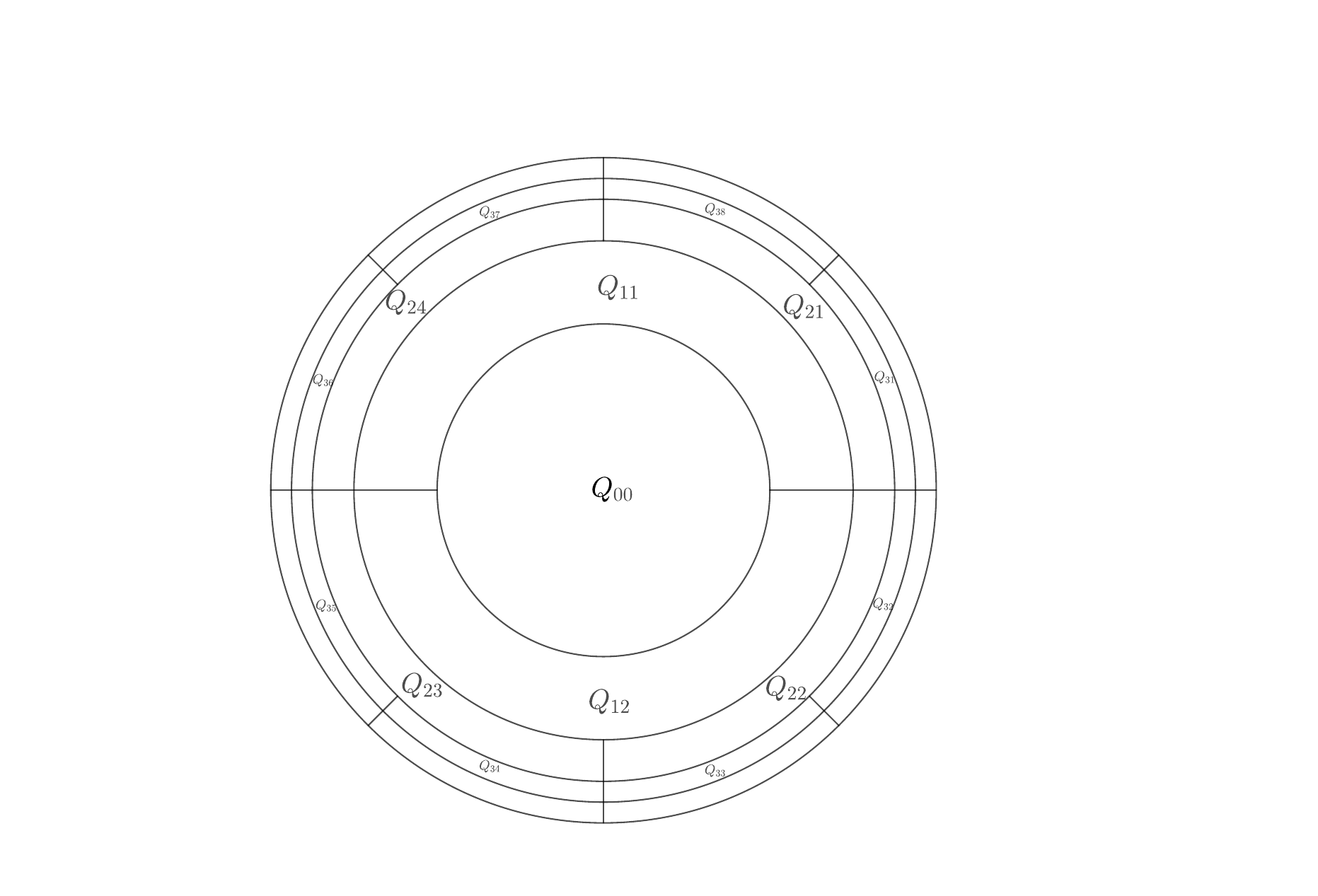}
\caption{Discretized disc}
\label{fig:17}
\end{figure}
We start with making a decomposition of the unit disc into dyadic Carleson boxes. For integer $j\geq0$ and $1\leq l\leq 2^j$ let $z_{jl}= (1-2^{-j})e^{\frac{2\pi i(2l-1)}{2^{j+1}}}$, and for $z=re^{it}\in\mathbb{D}$ let $J(z) = \{e^{is}:\; t- (1-r)\pi \leq s \leq t+ (1-r)\pi\}$, $S(z) = \{\rho e^{is}:\;e^{is}\in J(z);\, r\leq\rho\leq1\}$, and let $Q(z) = \{ \rho e^{is}\in S(z):\;\frac{1-r}{2}\leq 1-\rho\leq 1-r\}$ be the 'upper half' of $S(z)$. We write $Q_{jl} := Q(z_{jl})$. Now we see that there is one-to-one map between points (vertices) of $T$ and dyadic Carleson half-cubes $Q_{jl}$; $Q_{00}$ corresponds to the root $o$, $Q_{11}$ and $Q_{12}$ to its two children etc. (see Fig. \ref{fig:17}). In other words, for every $\alpha\in T$ there exists a unique half-cube $Q_{\alpha}$, and vice versa, for every half-cube $Q_{jl}$ there is exactly one point $\alpha^{jl}\in T$. The collection $\{Q_{\alpha}\}_{\alpha\in T}$ forms a covering of the unit disc. Note also that given a point $z\in \mathbb{D}$ it is possible to pick the half-box $Q_{\alpha}\ni z$ in a unique way. Though it can happen that there are several half-boxes $Q_{\alpha}$ containing $z$ (up to four), we can still pick up one of them (say, whichever is closer to $\partial\mathbb{D}$ and/or with larger $\arg z_{jl}$), and we do this, wherever it is needed, in a consistent fashion throughout the whole paper.

Next we introduce an auxiliary graph $\mathfrak{G}$ in such a way that $\mathcal{V}(\mathfrak{G}) := \mathcal{V}(T)$, and $\{\alpha,\beta\}$ is an edge of $\mathfrak{G}$, if $\textup{cl}Q_{\alpha}\cap \textup{cl}Q_{\beta} \neq\emptyset$. Here $\textup{cl}$ means closure in Euclidean distance. Basically we take $T$ and add extra edges connecting points corresponding to adjacent Carleson half-cubes, see Fig. \ref{fig:18}.
\begin{figure}[h]
\centering
\includegraphics[width=0.6\textwidth]{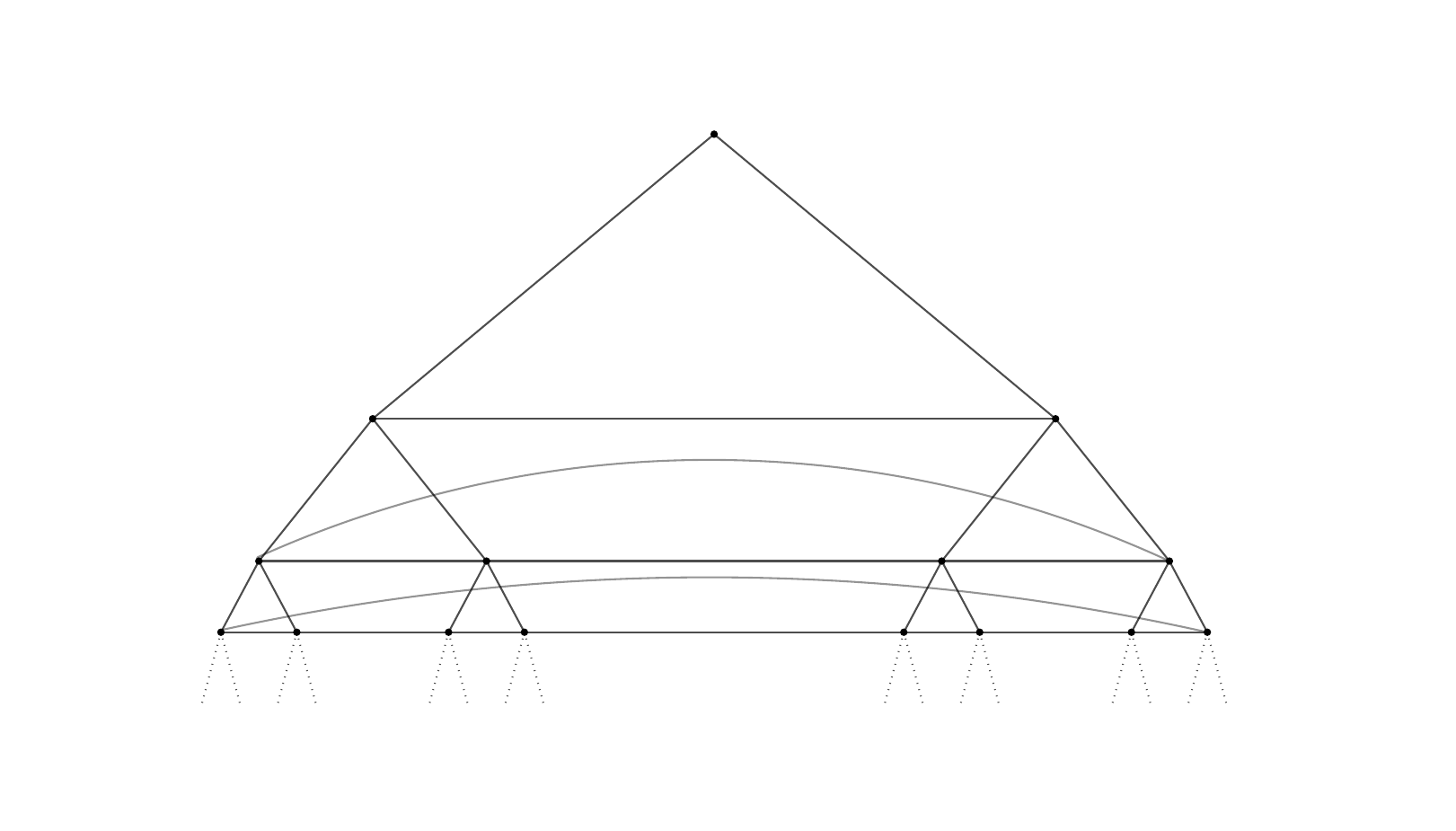}
\caption{Graph $\mathfrak{G}$}
\label{fig:18}
\end{figure}
 Given a vertex or leaf $\alpha\in \mathcal{V}(T)\bigcup\partial T$ we define the \textit{$\mathfrak{G}$-extended predecessor set} to be
$
\mathcal{P}_{\mathfrak{G}}(\alpha) := \{\beta \in \mathcal{V}(T) = \mathcal{V}(\mathfrak{G}):\; dist_{\mathfrak{G}}(\beta,\Gamma(o,\alpha)) \leq 1\},
$
where $\Gamma(o,\alpha)$ is the (unique) geodesic in $\overline{T}$ connecting $\alpha$ and the root $o$. In other words, we take the $\overline{T}$-predecessor set $\mathcal{P}_T(\alpha)$ and add all the adjacent (in $\mathfrak{G}$) vertices (see Fig. \ref{fig:19}). As before, we set 
$
\mathcal{S}_{\mathfrak{G}}(\beta) := \{\alpha\in \overline{T}:\; \beta\in \mathcal{P}_{\mathfrak{G}}(\alpha) \}
$
to be the \textit{$\mathfrak{G}$-extended successor set.}
\begin{figure}[h]
	\centering
	\begin{subfigure}{0.45\textwidth}
	\includegraphics[width=\textwidth]{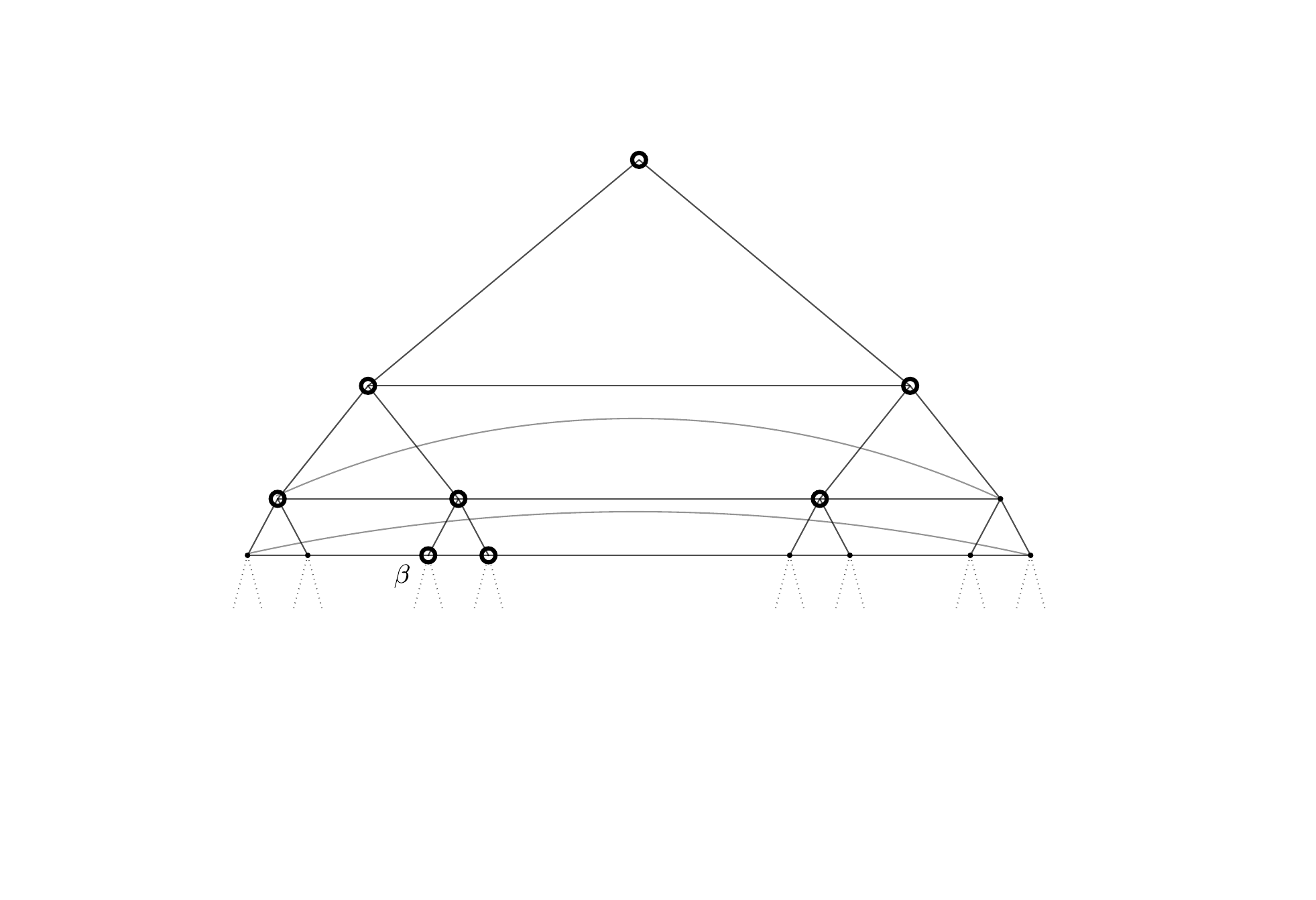}
        \caption{Predecessor set}
        \vspace{0.5mm}
    \end{subfigure}
    \begin{subfigure}{0.45\textwidth}\
     \includegraphics[trim = 0 10 0 10, clip, width=\textwidth]{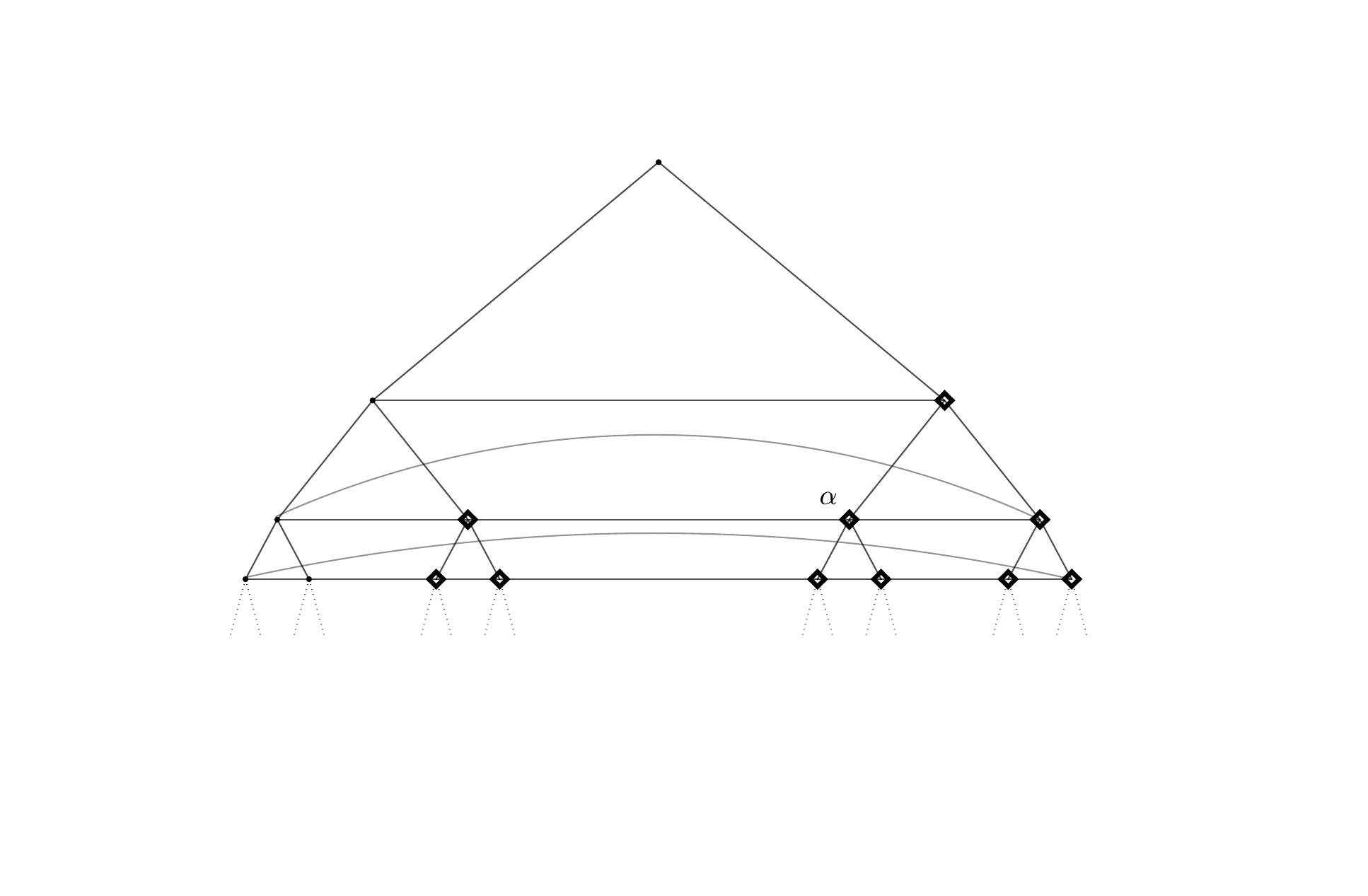}
        \caption{Successor set}
    \end{subfigure}
	\caption{}
	\label{fig:19}
\end{figure}
Since (by definition) $\mathcal{P}(\alpha) \subset \mathcal{P}_{\mathfrak{G}}(\alpha)$ for any $\alpha\in \overline{T}$, we have the same inclusion for the successor sets, $\mathcal{S}(\alpha) \subset \mathcal{S}_{\mathfrak{G}}(\alpha)$, and this inclusion is proper unless the vertex in question is the root $o$. On the other hand, the successor sets are 'comparable on average'. To elaborate, let $N(\alpha)$ be the set of neighbours of $\alpha$ in $\mathfrak{G}$ (so $\alpha$ is connected by an edge in $\mathfrak{G}$ to the points in $N(\alpha)$). Then
\begin{equation} \label{e:331}
\mathcal{S}_{\mathfrak{G}}(\alpha) \subset \bigcup_{\beta\in N(\alpha)}\mathcal{S}(\beta),
\end{equation}
and $|N(\alpha)| \leq 5$ (in particular $\bigcup_{\alpha\in T}N(\alpha)$ covers each point at most $5$ times). Another way to look at $\mathcal{S}_{\mathfrak{G}}(\alpha)$ is to consider the dyadic interval $J_{\alpha}$ and its two immediate neighbours of the same rank $J^{\pm}_{\alpha}$. Then 
\[
\mathcal{S}_{\mathfrak{G}}(\alpha)\cap T = \left\{\beta\in T: J_{\beta}\subset J_{\alpha}\cup J^-_{\alpha}\cup J^{+}_{\alpha}\right\}.
\]
The correspondence between $\partial T$ and $\partial\mathbb{D}$ will be explained later in Section \ref{SS:2.5}.
Finally, we let
\[
d_{\mathfrak{G}}(\alpha\wedge\beta) := \left|\mathcal{P}_{\mathfrak{G}}(\alpha)\cap\mathcal{P}_{\mathfrak{G}}(\beta)\right|,\quad \alpha,\beta\in T.
\]
As before, we keep the same notation for the bitree, namely given two points $\alpha = (\alpha_x,\alpha_y)$ and $\beta = (\beta_x,\beta_y)$ in $T^2$ we set
\begin{equation}\notag
\begin{split}
&\mathcal{P}_{\mathfrak{G}^2}(\alpha) := \mathcal{P}_{\mathfrak{G}}(\alpha_x)\times\mathcal{P}_{\mathfrak{G}}(\alpha_y),\; \mathcal{S}_{\mathfrak{G}^2}(\alpha) := \mathcal{S}_{\mathfrak{G}}(\alpha_x)\times\mathcal{S}_{\mathfrak{G}}(\alpha_y)\\
&d_{\mathfrak{G}^2}(\alpha\wedge\beta) := d_{\mathfrak{G}}(\alpha_x\wedge\beta_x)\cdot d_{\mathfrak{G}}(\alpha_y\wedge\beta_y)
\end{split}
\end{equation}
The main reason to introduce this auxiliary graph $\mathfrak{G}$ is that the geometry of the tree $T$ does not completely agree with the geometry of the unit disc $\mathbb{D}$. For instance, one can easily find a pair of points $z,w\in\mathbb{D}$, very close to each other, while the tree distance between $\alpha$ and $\beta$ corresponding to these points (i.e. $z\in Q_{\alpha},\; w\in Q_{\beta}$) is very large. It is a well-known (if somewhat minor) obstacle, and there are several ways to overcome it. We have chosen what we think is the simplest one, especially since we do not care about precise values of arising constants.

Taking two identical dyadic coordinate trees $T_x,\, T_y$ we see that the collection $\{Q_{\alpha}\}:= \{Q_{\alpha_x}\times Q_{\alpha_y}\},\; \alpha = (\alpha_x,\alpha_y) \in T_x\times T_y = T^2$  gives almost a disjoint decomposition of the bidisc $\mathbb{D}\times\mathbb{D}$ (these Whitney cubes may intersect, but each point of the bidisc is counted at most $16$ times). Assume that $\mu\geq0$ is a Borel measure on $\mathbb{D}^2$ for which $\mu(\partial\mathbb{D}^2)=0$ and that $g\in L^2(\mathbb{D}^2,\,d\mu)$ is a non-negative function. We then let
\begin{equation}\label{e:377}
\begin{split}
& \tilde{\mu}(\alpha) := \mu(Q_{\alpha}),\\
& \tilde{g}(\alpha) := \frac{1}{\mu(Q_{\alpha})}\int_{Q_{\alpha}}g(z)\,d\mu(z),\quad \alpha = (\alpha_x,\alpha_y)\in T^2,
\end{split}
\end{equation}
and we set $\tilde{g}(\alpha) := 0$, if $\mu(Q_{\alpha}) = 0$. Now we recall that $\mu$ is Carleson measure for $\mathcal{D}(\mathbb{D}^2)$ if and only if \eqref{e:25} holds for any $g$ as above, namely
\begin{equation}\label{e:37}
\begin{split}
&\|g\|^2_{L^2(\mathbb{D}^2,\,d\mu)} = \int_{\mathbb{D}^2}g^2(z)\,d\mu(z) = \sum_{\alpha\in T^2}\int_{Q_{\alpha}}g^2(z)\,d\mu(z) \gtrsim \int_{\mathbb{D}^2}\int_{\mathbb{D}^2}g(z)g(w)\Re K_z(w)\,d\mu(z)\,d\mu(w) \approx\\
&\sum_{\alpha\in T^2}\sum_{\beta\in T^2}\int_{Q_{\alpha}}\int_{Q_{\beta}}g(z)g(w)\Re K_z(w)\,d\mu(z)\,d\mu(w).
\end{split}
\end{equation}
In order to proceed we need the following Lemma, the proof of which is given in Section \ref{S:A.0}.
\begin{lemma}\label{l:A.0.s}
For any $\alpha,\beta\in T^2$ we have
\[
\Re K_z(w) \approx d_{\mathfrak{G}^2}(\alpha\wedge\beta) = d_{\mathfrak{G}}(\alpha_x\wedge\beta_x)\cdot d_{\mathfrak{G}}(\alpha_y\wedge\beta_y),\quad  z\in Q_{\alpha},\; w\in Q_{\beta}.
\].
\end{lemma}

Applying Lemma \ref{l:A.0.s} to the right-hand side of \eqref{e:37} we get
\begin{equation}\notag
\begin{split}
&\sum_{\alpha\in T^2}\sum_{\beta\in T^2}\int_{Q_{\alpha}}\int_{Q_{\beta}}g(z)g(w)\Re K_z(w)\,d\mu(z)\,d\mu(w) \approx\\
& \sum_{\alpha\in T^2}\sum_{\beta\in T^2}\int_{Q_{\alpha}}\int_{Q_{\beta}}g(z)g(w)d_{\mathfrak{G}^2}(\alpha\wedge\beta)\,d\mu(z)\,d\mu(w)=\\
&\sum_{\alpha\in T^2}\sum_{\beta\in T^2}\tilde{g}(\alpha)\tilde{g}(\beta)d_{\mathfrak{G}}(\alpha_x\wedge\beta_x)d_{\mathfrak{G}}(\alpha_y\wedge\beta_y)\tilde{\mu}(\alpha)\tilde{\mu}(\beta)
\end{split}
\end{equation}
We attack the calculation from the end, letting $\sigma(\alpha) = \tilde{g}(\alpha)\tilde{\mu}(\alpha),\; \alpha\in T^2$ (recall that measures and functions on $T^2$ are the same):
\begin{equation}\notag
\begin{split}
&\sum_{\gamma\in T^2}\left(\sum_{\alpha\in\mathcal{S}_{\mathfrak{G}^2}(\gamma)}\sigma(\alpha)\right)^2 = \sum_{\gamma\in T^2}\sum_{\alpha\in\mathcal{S}_{\mathfrak{G}^2}(\gamma)}\sum_{\beta\in\mathcal{S}_{\mathfrak{G}^2}(\gamma)}\sigma(\alpha)\sigma(\beta) = \sum_{\alpha\in T^2}\sum_{\beta\in T^2}\sum_{\gamma\in \mathcal{P}_{\mathfrak{G}^2}(\alpha)\cap \mathcal{P}_{\mathfrak{G}^2}(\beta)}\sigma(\alpha)\sigma(\beta)=\\
&\sum_{\alpha\in T^2}\sum_{\beta\in T^2}|\mathcal{P}_{\mathfrak{G}^2}(\alpha)\cap \mathcal{P}_{\mathfrak{G}^2}(\beta)|\sigma(\alpha)\sigma(\beta) = \sum_{\alpha\in T^2}\sum_{\beta\in T^2}\tilde{g}(\alpha)\tilde{g}(\beta)d_{\mathfrak{G}}(\alpha_x\wedge\beta_x)d_{\mathfrak{G}}(\alpha_y\wedge\beta_y)\tilde{\mu}(\alpha)\tilde{\mu}(\beta).
\end{split}
\end{equation}
 Repeating the calculation with $\mathcal{P}$ instead of $\mathcal{P}_{\mathfrak{G}^2}$ we obtain
\begin{equation}\notag
\sum_{\gamma\in T^2}\left(\sum_{\alpha\in\mathcal{S}(\gamma)}\sigma(\alpha)\right)^2 = \sum_{\alpha\in T^2}\sum_{\beta\in T^2}\tilde{g}(\alpha)\tilde{g}(\beta)d_{T^2}(\alpha\wedge\beta)\tilde{\mu}(\alpha)\tilde{\mu}(\beta).
\end{equation}
The successor set formula \eqref{e:331} implies that
\begin{equation}\notag
\begin{split}
&\sum_{\gamma\in T^2}\left(\sum_{\alpha\in\mathcal{S}_{\mathfrak{G}^2}(\gamma)}\sigma(\alpha)\right)^2 = \sum_{\gamma\in T^2}\sigma\left(\mathcal{S}_{\mathfrak{G}}(\gamma_x)\times\mathcal{S}_{\mathfrak{G}}(\gamma_y)\right)^2 \approx \sum_{\gamma\in T^2}\sigma\left(\mathcal{S}(\gamma_x)\times\mathcal{S}(\gamma_y)\right)^2=\\
&\sum_{\gamma\in T^2}\left(\sum_{\alpha\in\mathcal{S}(\gamma)}\sigma(\alpha)\right)^2.
\end{split}
\end{equation}
Combining the estimates above we see that \eqref{e:37} is equivalent to
\begin{equation}\notag
\sum_{\alpha\in T^2}\int_{Q_{\alpha}}g^2(z)\,d\mu(z) \gtrsim \sum_{\alpha\in T^2}\sum_{\beta\in T^2}\tilde{g}(\alpha)\tilde{g}(\beta)d_{T^2}(\alpha\wedge\beta)\tilde{\mu}(\alpha)\tilde{\mu}(\beta),
\end{equation}
where $\tilde{g}$ and $\tilde{\mu}$ are defined in \eqref{e:377}. We see that if $g$ is constant on the boxes $Q_{\alpha}$ and $\mu$ is Carleson measure for $\mathcal{D}(\mathbb{D}^2)$, then
\begin{equation}\label{e:38}
\|\tilde{g}\|^2_{L^2(T^2, d\tilde{\mu})} \gtrsim \sum_{\alpha\in T^2}\sum_{\beta\in T^2}\tilde{g}(\alpha)\tilde{g}(\beta)d_T(\alpha\wedge\beta)\tilde{\mu}(\alpha)\tilde{\mu}(\beta) = \sum_{\alpha\in T^2}\left(\sum_{\beta\leq\alpha}\tilde{g}(\beta)\tilde{\mu}(\beta)\right)^2.
\end{equation}
On the other hand, by Jensen's inequality, 
$$\sum_{\alpha\in T^2}\int_{Q_{\alpha}}g^2(z)\,d\mu(z) \geq \sum_{\alpha\in T^2}\tilde{g}^2(\alpha)\tilde{\mu}(\alpha),$$
 so if \eqref{e:38} holds for any non-negative $\tilde{g}$ in $L^2(T^2,\,d\tilde{\mu})$, then $\mu$ is Carleson.
 
Let $\nu\geq0$ be a Borel measure on $\overline{T}^2$. Given a $\nu$-measurable function $\varphi$ defined on $\overline{T}^2$, we let
\begin{equation}\label{e:40}
(\mathbb{I}_{\nu}^*\varphi)(\beta) := \int_{\mathcal{S}(\beta)}\varphi(\alpha)\,d\nu(\alpha),\quad \beta\in T^2.
\end{equation}
Since by our temporary assumption $\tilde{\mu}$ is supported on $T^2$, inequality \eqref{e:38} can be rewritten as
\begin{equation}\label{e:41}
\|\mathbb{I}^*_{\tilde{\mu}}\tilde{g}\|^2_{\ell^2(T^2)} \lesssim \|\tilde{g}\|^2_{L^2(T^2,\,d\tilde{\mu})},
\end{equation}
which means that the operator $\mathbb{I}^*_{\tilde{\mu}}:  L^2(T^2,\,d\tilde{\mu})\rightarrow \ell^2(T^2)$ is bounded (and its operator norm is comparable to the Carleson constant $[\mu]$). Given a pair of real functions $\varphi\in \ell^2(T^2)$ and $f\in L^2(T^2,\,d\nu)$ we see that
\begin{equation}\notag
\begin{split}
&\langle \mathbb{I}^*_{\nu}\varphi,f\rangle_{\ell^2(T^2)} = \sum_{\beta\in T^2}\int_{\mathcal{S}(\beta)}\varphi(\alpha)\,d\nu(\alpha)f(\beta) = \int_{\overline{T}^2}\varphi(\alpha)\sum_{\beta\in T^2}\chi_{\mathcal{S}(\beta)}(\alpha)f(\beta)\,d\nu(\alpha)=\\
& \int_{\overline{T}^2}\varphi(\alpha)\sum_{\beta\geq\alpha}f(\beta)\,d\nu(\alpha) = \langle\varphi, \mathbb{I}f\rangle_{L^2(\overline{T}^2,\,d\nu)}.
\end{split}
\end{equation}
Hence $\mathbb{I}$, as an operator acting from $\ell^2(T^2)$ to $L^2(\overline{T}^2,\,d\nu)$, is adjoint to $\mathbb{I}^*_{\nu}$ and $\|\mathbb{I}\| = \|\mathbb{I}^*_{\nu}\|$. We arrive at the following statement.
\begin{proposition}\label{p:71}
Let $\mu\geq0$ be a Borel measure on $\mathbb{D}^2$ and define $\tilde{\mu}$ as in \eqref{e:377}. Then $\tilde{\mu}$ is a trace measure for discrete bi-parameter Hardy inequality,
\begin{equation}\label{e:42}
\int_{\overline{T}^2}(\mathbb{I}f)^2\,d\tilde{\mu} \leq \tilde{C}_{\mu}\sum_{\alpha\in T^2}f^2(\alpha)
\end{equation}
if and only if $\mu$ is Carleson for $\mathcal{D}(\mathbb{D}^2)$. The best possible constant in \eqref{e:42} is comparable to the Carleson constant of $\mu$.
\end{proposition}

\subsection{Proof of Theorem \ref{intromain}: Maz'ya approach}\label{SS:2.6}
Here we show that trace measures for the bi-parameter Hardy operator admit a characterization via a discrete subcapacitary condition, as in Theorem \ref{traceinequality}. Then we translate this condition back to the continuous world, obtaining \eqref{introcap}.

Let $\nu\geq0$ be a Borel measure on $\overline{T}^2$ (note that now it might have non-zero mass on $\partial T^2$). We call it subcapacitary, if for any finite collection $\{\alpha^j\}_{j=1}^N\subset T^2$ one has
\[
\nu\left(\bigcup_{j=1}^{N}\mathcal{S}(\alpha^j)\right) \leq C \capp \left(\bigcup_{j=1}^{N}\mathcal{S}(\alpha^j)\right)
\]
for some constant $C>0$ the depends only on $\nu$ -- the smallest such constant we denote by $C_{\nu}$).

Assume now that $\nu$ is a trace measure for the Hardy operator,
\[
\int_{\overline{T}^2}(\mathbb{I}f)^2\,d\nu \lesssim \sum_{\alpha\in T^2}f^2(\alpha)
\]
for any $f:T^2\mapsto\mathbb{R}_+$. Given a Borel set $E\subset \overline{T}^2$ consider the family $$\Omega_E = \{f\in \ell^2(T^2): f\geq0,\; \mathbb{I}f\geq1\;\textup{on}\;E\}$$ of $E$-admissible functions. Then for any $f\in \Omega_E$ one has
\[
\nu(E) = \int_{\overline{T}^2}\chi_E\,d\nu \leq \int_{\overline{T}^2}(\mathbb{I}f)^2\,d\nu \lesssim \|f\|^2_{\ell^2(T^2)}.
\]
Taking infimum over $\Omega_E$ we immediately get $\nu(E) \lesssim \capp E$.\par
The other direction is more involved, and the argument follows the route pioneered by Maz'ya. Assume that $\nu\geq0$ is a subcapacitary measure on $\overline{T}^2$ and that $f\in\ell^2(T^2),\,f\geq0$. By a distribution function argument 
\[
\int_{\overline{T}^2}f^2\,d\nu \approx \sum_{k\in\mathbb{Z}}2^{2k}\nu\left\{\alpha\in\overline{T}^2: \mathbb{I}f > 2^k\right\}.
\]
Since $f\geq0$, we have that if $(\mathbb{I}f)(\alpha) > 2^k$, then $(\mathbb{I}f)(\beta)> 2^k$ for any $\beta\leq\alpha$. Therefore for any $k\in\mathbb{Z}$ there exists a countable family $\{\alpha^j_k\}_{j=0}^{\infty}\subset T^2$ such that
\[
\left\{\alpha\in\overline{T}^2: \mathbb{I}f > 2^k\right\} = \bigcup_{j=0}^{\infty}\mathcal{S}(\alpha^j_k),
\]
with $(\mathbb{I}f)(\alpha^j_k) > 2^k$. It follows that
\[
\nu\left\{\alpha\in\overline{T}^2: \mathbb{I}f > 2^k\right\} = \lim_{N\rightarrow\infty}\nu\left(\bigcup_{j=0}^{N}\mathcal{S}(\alpha^j_k)\right) \lesssim\lim_{N\rightarrow\infty}\capp\left(\bigcup_{j=0}^{N}\mathcal{S}(\alpha^j_k)\right) = \capp\left(\bigcup_{j=0}^{\infty}\mathcal{S}(\alpha^j_k)\right).
\]
Now assume for a moment that the following inequality holds (see Theorem \ref{capacitarystrong} and its proof in Section \ref{S:5}),
\[
\sum_{k\in\mathbb{Z}}2^{2k}\capp\{\mathbb{I}f \geq 2^k\} \leq C\|f\|^2_{\ell^2(T^2)}
\]
for some absolute constant $C>0$. Then we immediately have
\[
\sum_{k\in\mathbb{Z}}2^{2k}\nu\left\{\alpha\in\overline{T}^2: \mathbb{I}f > 2^k\right\} \leq  \sum_{k\in\mathbb{Z}}2^{2k}\capp\left\{\alpha\in\overline{T}^2: \mathbb{I}f \geq 2^k\right\} \leq C\|f\|^2_{\ell(T^2)},
\]
for any $f\geq0$ on $T^2$. Therefore $\nu$ is a trace measure for Hardy inequality. Theorem \ref{traceinequality} is proven.

All that remains to finish the proof of Theorem \ref{intromain} (for measures with zero mass on $\partial\mathbb{D}^2$, and still assuming the Strong Capacitary Inequality) is to go back to the bidisc. We start by defining a continuous version of capacity that is convenient for our purposes.
 The Riesz-Bessel kernel of order $(\frac12,\frac12)$ on the torus $(\partial \mathbb{D})^2 $ is
\[
b_{(\frac12,\frac12)}(z,\zeta) = |\theta_1-\eta_1|^{-\frac12}|\theta_2-\eta_2|^{-\frac12},\quad z = (e^{i\theta_1},e^{i\theta_2}),\,\zeta = (e^{i\eta_1},e^{i\eta_2})\in (\partial\mathbb{D})^2,
\]
where the difference $\theta_i-\eta_i\in [-\pi,\pi)$ is taken modulo $2\pi$. The kernel extends to a convolution operator on $(\partial\mathbb{D})^2$ acting on Borel measures supported there,
\[
(B_{(\frac12,\frac12)}\mu)(z) = \int_{(\partial\mathbb{D})^2}b_{(\frac12,\frac12)}(z,\zeta)\,d\mu(\zeta).
\]
Let $E\subset (\partial\mathbb{D})^2$ be a closed set. The $(\frac12,\frac12)$-Bessel capacity of $E$ is
\[
\capp_{(\frac12,\frac12)}(E) = \inf\{\|h\|^2_{L^2((\partial\mathbb{D})^2,dm)}: \;h\geq0\;\textup{and}\; B_{(\frac12,\frac12)}h \geq 1\;\textup{on}\; E\},
\]
and it is realized by an equilibrium measure $\mu_E$:
\[
\capp_{(\frac12,\frac12)}(E) = \mathcal{E}_{(\frac12,\frac12)}[\mu_E] := \int_{(\partial\mathbb{D})^2}\left((B_{(\frac12,\frac12)}\mu_E)(z)\right)^2\,dm(z),
\]
where $m$ is normalized area measure on the torus $(\partial\mathbb{D})^2$.

Let $\{J_k\}_{k=0}^N$ be a finite collection of dyadic rectangles on $(\partial\mathbb{D})^2$, i.e. $J_k = J^1_k\times J^2_k$, where $J^i_k$ is a dyadic interval in $\partial\mathbb{D}$. For any such collection there exists a unique sequence $\{\alpha_k\}_{k=0}^N\subset T^2$ such that $J_k = S_{\alpha_k}\cap(\partial\mathbb{D})^2$ (here $S_{\alpha_k}$ is the Carleson box corresponding to $\alpha_k$), and vice versa, any finite sequence $\{\alpha_k\}_{k=0}^{N}$ produces a family $\{J_k\}_{k=0}^N$ of dyadic rectangles. A standard argument shows that $(\frac12,\frac12)$-Bessel capacity and discrete bilogarithmic capacity are comparable. For the proof, see Section \ref{S:A.0}.
\begin{lemma}
For any finite collection $\{\alpha_k\}_{k=0}^N$ one has
\begin{equation}\notag
\capp_{(\frac12,\frac12)}\left(\bigcup_{k=0}^N J_k\right) \approx \capp\left(\bigcup_{k=0}^N\mathcal{S}(\alpha_k)\right),
\end{equation}
where $\alpha_k$ and $J_k$ are related as above.
\end{lemma}

Theorem \ref{intromain} follows immediately. Indeed, we have shown that $\mu\geq0$ on $\mathbb{D}^2$ is Carleson if and only if its discrete image $\tilde{\mu}$ is a trace measure for Hardy operator, and that this happens if and only if $\tilde{\mu}$ is subcapacitary in $\overline{T}^2$. Since $\tilde{\mu}(\cup_{k=0}^N\mathcal{S}(\alpha_k)) = \mu(\cup_{k=0}^NS(J_k))$, \eqref{introcap} follows, and we are done.

\subsection{From the bidisc to the bitree: general case}\label{SS:2.5}
Up until now we assumed the measure $\mu$ to be supported inside the bidisc. Here we get rid of this restriction and prove Theorem \ref{intromain} in full generality, still assuming the Strong Capacitary Inequality. We also show that $C[\mu]$ is comparable to $\sup_{f}\|\vvar f\|^2_{L^2(\overline{\mathbb{D}}^2,\,d\mu)}$, as promised in Theorem \ref{dyadization}. To do so we first need to define the discrete image of a measure with non-zero mass on the boundary $\partial\mathbb{D}^2$. We consider the case of the
distinguished boundary $(\partial\mathbb{D})^2$ first, which is more interesting and contains the ingredients for the remaining part as well. The problem is that the boundaries of the complex disc and of the tree, and the measures supported on them, can not be identified without some care.

We introduce a map $\Lambda : (\partial T)^2\rightarrow (\partial\mathbb{D})^2$, $\Lambda(\alpha) = (\Lambda_0(\alpha_x),\Lambda_0(\alpha_y))$, where $\Lambda_0:\partial T\rightarrow\partial\mathbb{D}$ maps a geodesic $\omega = \{o=\omega^0,\omega^1,\dots\}$ to the point $\Lambda_0(\omega) = \bigcap_{n=0}^{\infty}S_{\omega^n} \in\partial\mathbb{D}$.
We will use $\Lambda$
to move measures back and forth from $(\partial T)^2$ to $(\partial
\mathbb{D})^2$, in such a way corresponding measures have comparable mass and
energy.

Consider on $\overline{T}$ the distance
\[
\delta_0 (\zeta, \xi) \assign 2^{-d_T(\zeta\wedge\xi)} - \frac{1}{2}\left(2^{-d_T(\zeta)} + 2^{-d_T(\xi)}\right)
\]
It is clear that  $\Lambda_0$ is a Lipschitz map with respect to the distance $\delta_0$ on $\partial T$ and Euclidean distance on $\partial \mathbb{D}$, and that $\Lambda_0$ is injective but for the set of the dyadic values $\frac{2 \pi j}{2^n}$ with $1 \leqslant j \leqslant 2^n$, which have two preimages.

Then $\Lambda$ is Lipschitz with respect to the distance $\delta$ defined in \eqref{e:31} on $(\partial T)^2$ and the usual distance of the torus $(\partial \mathbb{D})^2$. Given a positive, Borel measure $\nu$ on $(\partial T)^2$ , let $\Lambda_{\ast} \nu (F) \assign \nu (\Lambda^{- 1} (F))$ be its natural push-forward. We need to define an (unnatural) pull-back.
Given a positive, Borel measure $\mu$ on $(\partial \mathbb{D})^2$, define
$\Lambda^{\ast} \mu$ be the measure assigning to a Borel subset $E \subseteq
(\partial T)^2$ the number
\[ \Lambda^{\ast} \mu (E) = \int_{(\partial \mathbb{D})^2} \frac{\sharp
   (\Lambda^{- 1} (\{ z \} ) \cap E)}{\sharp (\Lambda^{- 1} (\{ z \} ))} d \mu
   (z) \nocomma, \]
that is,
\[ \int_{(\partial T)^2} \varphi (x) d \Lambda^{\ast} \mu (x) =
   \int_{(\partial \mathbb{D})^2} \frac{\sum_{x \in \Lambda^{- 1} (\{ z \})}
   \varphi (x)}{\sharp (\Lambda^{- 1} (\{ z \} ))} d \mu (z) . \]
\[ \  \]
If it is well-defined, then $\Lambda^{\ast} \mu$ defines a countably additive,
positive set function. But we have to show, first, that the function $z
\mapsto \frac{\sharp (\Lambda^{- 1} (\{ z \} ) \cap E)}{\sharp (\Lambda^{- 1}
(\{ z \} ))} = \varphi_E (z)$ is measurable on $(\partial \mathbb{D})^2$ (this is a simpler but slightly more technical version of the argument in \cite{arsw2014}).

For each point $\alpha$ in $T$ we denote its children by $\alpha_+$ and $\alpha_-$. We
can split $\partial T = A_+ \cup A_- \cup A$ into the disjoint union of three
Borel measurable sets: $A_+$ is the countable set of the geodesics $\omega =
(\omega^n)_{n = 0}^{\infty}$ such that $\omega^{n + 1} = \omega^n_+$
definitely; $A_-$ is defined similarly; $A = \partial T \setminus (A_+ \cup
A_{- .})$. The map $\Lambda_0$ is injective on each set. Correspondingly, we
split $(\partial T)^2$ into nine disjoint measurable sets $B_1, \ldots, B_9$ ,
on each of which $\Lambda$ is injective.

The map $z \mapsto \sharp (\Lambda^{- 1} (\{ z \} ))$ takes on the value $1$
on $\Lambda (A \times A)$, it takes on the value $2$ on $\Lambda (A_{\pm}
\times A \cup A \times A_{\pm})$, it takes on the value $4$ on $\Lambda
(A_{\pm} \times A_{\pm} \cup A_{\mp} \times A_{\pm})$; hence, it is Borel
measurable on $(\partial \mathbb{D})^2$. Similarly, the map $z \mapsto \sharp
(\Lambda^{- 1} (\{ z \} ) \cap E)$ takes on the value $1$ on $\Lambda (A
\times A \cap E)$, etcetera; hence it is Borel measurable as well. Thus,
$\varphi_E$ is measurable, as desired.\par

Next we consider the measures on the rest of the bidisc. First we extend the map $\Lambda$ on $\overline{T}^2$ by letting $\Lambda(\alpha_x,\alpha_y) = \Lambda_0(\alpha_x)\times\Lambda_0(\alpha_y)$, where  $\Lambda_0(\alpha_x) := Q_{\alpha_x},\; \Lambda_0(\alpha_y) := Q_{\alpha_y}$, if $\alpha = (\alpha_x,\alpha_y)\in T^2$. For a point $(\alpha_x,\omega_y)$ on the mixed boundary $T\times\partial T$ we set $\Lambda((\alpha_x,\omega_y)) := Q_{\alpha_x}\times\{\Lambda_0(\omega_y)\}$, and we do the same for the other part of the boundary.

Assume $\mu$ to be a positive Borel measure on $\mathbb{D}\times\partial\mathbb{D}$ and let $\alpha_x\in T_x,$ and $E_y$ be a Borel subset of $T_y$. We define the pull-back to be 
\[
(\Lambda^*\mu)(\{\alpha_x\}\times E_y) := \int_{Q_{\alpha_x}}\int_{E_y}\frac{\sharp\{\Lambda_0^{-1}(\{z_2\})\cap E_y\}}{\sharp\{\Lambda_0^{-1}(\{z_2\})\}}\,d\mu(z_1,z_2),
\]
the integrand being measurable for the same reasons as above.\\
Any set $E \subset T\times\partial T$ is a disjoint countable union of the product sets, i.e. there exist families $\{\alpha^j_x\}, E_y^j,\; j=0,\dots$, such that 
\[
E = \bigvee_{j=0}^{\infty}\left(\{\alpha_x^j\}\times E^j_y\right).
\]
Hence $\Lambda^*\mu$ admits a unique extension to Borel sets on $T\times\partial T$. The measures on $\partial\mathbb{D}\times\mathbb{D}$ are dealt with the same way.

Finally, for $\alpha\in T^2$ we put
\[
(\Lambda^*\mu)(\alpha) := \mu(\Lambda(\alpha)) =  \mu(Q_{\alpha}).
\]
We also need the one-dimensional version of the pull-back. Consider a Borel measure $\mu\geq0$ on the closed unit disc $\overline{\mathbb{D}}$, we define its pull-back to the tree $T$ to be
\[
(\Lambda^*_0\mu)(\alpha) = \mu(Q_{\alpha}),\; \Lambda^{\ast}_0 \mu (E) = \int_{\partial\mathbb{D}} \frac{\sharp
   (\Lambda_0^{- 1} (\{ z \} ) \cap E)}{\sharp (\Lambda_0^{- 1} (\{ z \} ))} d \mu
   (z) \nocomma,
\]
for a set $E\subset \partial T$ (it is much simpler in one dimension, since we do not need to take care of the mixed parts of the boundary).

There is no natural way to define a push-forward $\Lambda_*$ of a measure on $T^2$ (or on $T\times\partial T$ for that matter), since a point mass on $T^2$ (a positive number attached to a point $\alpha\in T^2$) can be moved to $Q_{\alpha}$ in several different manners (for instance it could be spread uniformly over $Q_{\alpha}$, or considered as a point mass, concentrated at the centerpoint of $Q_{\alpha}$). On the other hand, in what follows we do not need to use a push-forward of such a measure anyway.

Now we can prove the following Theorem, which contains one half of Theorem \ref{dyadization}.
\begin{theorem}\label{t:equivthingies}
  Let $\mu$ be a Borel measure on $\overline{\mathbb{D}}^2$, then
  \[ | | \mathbb{I} | |_{\mathcal{B} (\ell^2 (T^2), L^2 (\overline{T}^2,\,d\Lambda^{\ast} \mu))}^2
     \approx \| \tmop{Id} \|^2_{\mathcal{B} (\mathcal{D} (\mathbb{D}^2), L^2
     (\overline{\mathbb{D}}^2,\,d\mu))} . \]
  Moreover, $\| \tmop{Id} \|^2_{\mathcal{B} (\mathcal{D} (\mathbb{D}^2), L^2
  (\overline{\mathbb{D}}^2,\,d\mu))}$ is also comparable with the best constant $K^2$ in the stronger
  inequality
  \[ \int_{\overline{\mathbb{D}}^2} (\vvar f)^2 d \mu \leqslant K^2  \| f
     \|^2_{\mathcal{D} (\mathbb{D}^2)}, \]
  where $\vvar$ is the {\tmem{radial variation}} of $f$.
\end{theorem}



Let $f$ be holomorphic in $\mathbb{D}^2$. The {\tmem{radial variation}} \ of $f$ at
$(\zeta_1,\zeta_2) \in \overline{\mathbb{D}}^2$ is
\begin{equation}\label{e:575}
\begin{split}
&\vvar(f)(\zeta_1,\zeta_2) = \vvar_{12}(f)(\zeta_1,\zeta_2) + \vvar_1(f)(\zeta_1) + \vvar_2(f)(\zeta_2) + |f(0,0)| =\\
&\int_0^{\zeta_1}\int_0^{\zeta_2}|\partial_{z_1,z_2}f(z_1,z_2)||dz_1||dz_2| + \int_0^{\zeta_1}|\partial_{z_1}f(z_1,0)||dz_1| + \int_0^{\zeta_2}|\partial_{z_2}f(0,z_2)||dz_2| + |f(0,0)|.
\end{split}
\end{equation}

\begin{proof}
We will prove the chain of implications $(A) \Rightarrow (B) \Rightarrow (C)
\Rightarrow (A)$:
\begin{itemizedot}
  \item (A) $\int_{\overline{\mathbb{D}}^2} \vvar (f)^2 d \mu \leqslant K_0^2  \|
  f \|^2_{\mathcal{D} (\mathbb{D}^2)}$.
  
  \item (B) $\int_{\overline{\mathbb{D}}^2} | f |^2 d \mu \leqslant K_1^2  \|
  f \|^2_{\mathcal{D} (\mathbb{D}^2)}$.
  
  \item (C) $\int_{\overline{T}^2} | \mathbb{I} \varphi |^2 d \Lambda^{\ast} \mu
  \leqslant K_2^2 \| \varphi \|_{\ell^2 (T^2)}^2$.
\end{itemizedot}
We note that (B) means 
\[ \sup_{R < 1} \int_{\overline{\mathbb{D}}^2} | f (R \zeta_1, R \zeta_2) |^2 d
   \mu (\zeta_1, \zeta_2) \leqslant K_1^2  \| f \|^2_{\mathcal{D} (\mathbb{D}^2)} .
\]

The implication $(A) \Rightarrow (B)$ is elementary: 
 the
inequality for the variation is a priori stronger than the inequality for $| f
|$. For $(\zeta_1,\zeta_2) \in \mathbb{D}^2$:
\begin{eqnarray*}
  | f (\zeta_1,\zeta_2) | & = & \left| \int_0^{\zeta_1} \partial_{z_1} f (z_1, \zeta_2) \,dz_1 +
  \int_0^{\zeta_2} \partial_{z_2} f (0, z_2) d z_2 + f (0, 0) \right|\\
  & = & \left| \int_0^{\zeta_1} \left( \int_0^{\zeta_2} \partial_{z_1z_2} f (z_1, z_2) \,dz_2
  + \partial_{z_1} f (z_1, 0) \right) d z_1 + \int_0^{\zeta_2} \partial_{z_2} f (0, z_2)\, dz_2 + f
  (0, 0) \right|\\
  & \leqslant & \vvar_{12} (f) (\zeta_1, \zeta_2) + \vvar_1 (f) (\zeta_1) + \vvar_2 (f) (\zeta_2) + |
  f (0, 0) | .
\end{eqnarray*}

For $0 < r < 1$ let $f_r (z_1, z_2) = f (r z_1, r z_2)$. If $\mu$
satisfies $(A)$, then it satisfies Carleson inequality for $f_r$ with constant independent of $r$, and we are done.\\

The proof of the implications $(B)\Rightarrow (C)$ and $(C)\Rightarrow (A)$ are more involved, and before proceeding we need an additional smoothness property of Carleson measures: if a measure $\mu\geq0$ on the closed bidisc is Carleson, then for any set $E\subset \overline{T}^2$ one has
\begin{equation}\label{e:A.42}
(\Lambda^*\mu)(E) = \mu(\Lambda(E)).
\end{equation}
It means essentially that Carleson measures have no singularities on coordinate slices of the torus (the dyadic grid $\{\partial\mathbb{D}\times\{2\pi j2^{-n}\}\},\; j,n\geq0$ has no mass), see Lemma \ref{l:A.8} for details.

We now prove $(B) \Rightarrow (C)$. Suppose $\mu$ satisfies

\[ \int_{\overline{\mathbb{D}}^2} | f (Rz_1,Rz_2) |^2 d \mu (z_1,z_2)
   = : \int_{\overline{\mathbb{D}}^2} | f (z_1, z_2) |^2 d \mu_R (z_1, z_2) \leqslant K_1^2  \| f \|^2_{\mathcal{D} (\mathbb{D}^2)} \]
with $K_1$ independent of $R < 1$. Here $\mu_R (F) \assign \mu \left(
\frac{1}{R} F \right),\; F\subset \mathbb{C}^2$, where we consider the measure $\mu$ as a Borel measure
on $\mathbb{C}^2$, supported on $\overline{\mathbb{D}}^2$. The measure
$\mu_R$ has support in $\mathbb{D}^2$ and $\mu_R (\partial \mathbb{D}^2) =
0$. By Theorems \ref{dyadization} (already proven above for measures inside the bidisc) and \ref{traceinequality} the measure $\nu_R := \Lambda^*\mu_R$ is subcapacitary on $\overline{T}^2$ for any $R\in (0,1)$. It is enough to show that this implies the subcapacitary property of $\nu := \Lambda^*\mu$ as well.

Consider an arbitrary point $\alpha = (\alpha_1,\alpha_2)\in T^2$. We recall that it uniquely corresponds to the Carleson box $S_{\alpha} = S_{\alpha_1}\times S_{\alpha_2}$ with $S_{\alpha_k} = \{\rho_ke^{is_k}:\; e^{is_k} \in J_k;\; r_k\leq\rho_k<1\},\;k=1,2$. Here $r_k = 1-2^{-d_T(\alpha_k)+1}$ and $J_k$ is a dyadic interval of generation $d_T(\alpha_k)$ on $\partial\mathbb{D}$ such that $J_k = \Lambda_0(\partial\mathcal{S}(\alpha_k))$. Denote by $p(\alpha)$ the 'grandparent' of $\alpha\in T^2$, $p(\alpha) = (p_0(\alpha_1),p_0(\alpha_2))$, where $p_0(\alpha_k)$ is the immediate parent of $\alpha_k$ in respective coordinate tree. We claim that for $R\geq \max(r_1,r_2)$ one has
\begin{equation}\label{e:140}
\nu(\mathcal{S}(\alpha)) \leq \nu_R(\mathcal{S}(p(\alpha))).
\end{equation}
Indeed, for those values of $R$ we immediately have $RS_{\alpha} \subset S_{p(\alpha)}$, since $R r_k \geq 1-2(1-r_k) = 1-2^{-d_T(p_0(\alpha_k))+1},\; k=1,2$. The smoothness property \eqref{e:A.42} implies
\begin{equation}\notag
\begin{split}
&\nu(\mathcal{S}(\alpha)) = (\Lambda^*\mu)(\mathcal{S}(\alpha)) = \mu(\Lambda(\mathcal{S}(\alpha))) = \mu(S_{\alpha}) = \mu_R(RS_{\alpha}) \leq \mu_R(S_{p(\alpha)}) =\\
&\mu_R(\Lambda(\mathcal{S}(p(\alpha)))) = \nu_R(\mathcal{S}(p(\alpha))),
\end{split}
\end{equation}
recalling $\Lambda(\mathcal{S}(\alpha)) = S_{\alpha}$.

Consider now any finite collection $\{\alpha^j\}_{j=1}^N$ of points in $T^2$. Taking $R$ to be greater than $\max(r^j_k),\; k=1,2,\; j=1,\dots, N$, we obtain
\[
\nu\left(\bigcup_{j=1}^N\mathcal{S}(\alpha^j)\right) \leq \nu_R\left(\bigcup_{j=1}^N\mathcal{S}(p(\alpha^j))\right).
\]
Since 
\[
\capp \left(\bigcup_{j=1}^N\mathcal{S}(\alpha^j)\right) \approx \capp \left(\bigcup_{j=1}^N\mathcal{S}(p(\alpha^j))\right),
\]
see Lemma \ref{l:A.50.5}, it follows that $\nu = \Lambda^*\mu$ is a subcapacitary measure with constant comparable to $K_1^2$, hence we have $(C)$.\par

Finally we show that $(C)\Rightarrow (A)$.
We start with a local estimate for pieces of the (main term of) radial variation. Given a point $\zeta =(\zeta_1,\zeta_2) \in \overline{\mathbb{D}}^2$ let $P(\zeta) := \{z = (r_1\zeta_1,r_2\zeta_2),\; 0\leq r_1, r_2\leq1\}$,
so that $\vvar_{12}(\zeta) = \int_{P(\zeta)}|\partial_{z_1z_2}f(z_1,z_2)|\,d|z_1|d|z_2|$. For $\zeta\in\overline{\mathbb{D}}^2$ and $\alpha\in T^2$ define
\begin{eqnarray*}
  W (\zeta,\alpha) & \assign & \int_0^{\zeta_1} \int_0^{\zeta_2}
  \chi_{Q_{\alpha} \cap P(\zeta)} (z_1, z_2)_{} | \partial_{z_1 z_2} f
  (z_1, z_2) |\,  | d z_1 |  | d z_2 |\\
  & = & \int_0^1\int_0^1 \chi_{Q_{\alpha}} (p \zeta_1, q \zeta_2) |
  \partial_{z_1z_2} f (p \zeta_1, q \zeta_2) |\cdot|\zeta_1||\zeta_2|\,dp\,dq\\
  & \leqslant & \max \{| \partial_{z_1z_2} f (z_1, z_2) | : (z_1, z_2) \in Q_{\alpha}\} \cdot | Q_{\alpha} |^{1 / 2}\\
  & = & | \partial_{z_1,z_2} f (z_1(\alpha), z_2(\alpha)) | \cdot
  | Q_{\alpha} |^{1 / 2}\\
  &  & \tmop{for} \tmop{some\; point}\; z(\alpha) = (z_1(\alpha), z_2(\alpha))\;
  \tmop{in}\; Q_{\alpha}\\
  & = & \frac{| Q_{\alpha} |^{1 / 2}}{| D (z_1(\alpha), r)
  \times D (z_2(\alpha), s) |} \left| \int_{D (z_1(\alpha), r)
  \times D (z_2(\alpha), s)} \partial_{z_1 z_2} f (z_1, z_2) \tmop{dA} (z_1)
  \tmop{dA} (z_2) \right|\\
  &  & \tmop{by} \tmop{the} \tmop{mean\; value\; principle\; } \nocomma \tmop{with} r = \frac{1 - |
  z_1(\alpha) |}{2} \tmop{and} s = \frac{1 - | z_2(\alpha)|}{2}  \\
  & \lesssim & \left( \int_{D (z_1(\alpha), r) \times D (z_2(\alpha), s)} | \partial_{z_1 z_2} f (z_1, z_2) |^2 \nobracket d A (z_1) d A (z_2)
  \nobracket \right)^{1 / 2}\\
  &  & \tmop{by} \tmop{Jensen's\; inequality} .\\
  & = : & (H ((\alpha) \nobracket)^{1 / 2} .
\end{eqnarray*}
In other words, every piece of radial variation that passes through  $Q_{\alpha}$ can be estimated by  a single quantity $H(\alpha)$.
Summing over $\alpha$ along the 'route' $P(\zeta)$ we get
\begin{eqnarray*}
  \vvar_{12} (f) (\zeta)  = \sum_{Q_{\alpha}\cap P(\zeta)\neq\emptyset}W(\zeta,\alpha) &\lesssim & \sum_{Q_{\alpha} \cap P
  (\zeta) \neq \emptyset} (H (\alpha))^{1 / 2}\\
  & \leqslant & \sum_{\Lambda^{-1}(\zeta) = \beta} (\mathbb{I} H^{1 / 2}) (\beta) .
\end{eqnarray*}
To elaborate, if $\zeta\in \mathbb{D}^2$, then we can identify a Whitney box $Q_{\beta}\ni \zeta$ in a unique way, uniquely defining $\Lambda^{-1}(\zeta) = \beta$. If $\zeta$ lies on the boundary of the bidisc, then $\zeta$ has several (but boundedly many) $\Lambda$-preimages, and we just sum over all of them.
Integrating:
\begin{eqnarray*}
  \int_{\overline{\mathbb{D}}^2} \vvar_{12} (f) (\zeta)^2 d \mu (\zeta)
  & \lesssim & \int_{\overline{\mathbb{D}}^2} \sum_{\Lambda^{-1} (\zeta) = \beta} (\mathbb{I} H^{1 / 2} (\beta))^2 d \mu (\zeta)\\
  & \approx & \int_{\overline{\mathbb{D}}^2} \frac{\sum_{\Lambda^{-1} (\zeta) =
  \beta} (\mathbb I H^{1 / 2} (\beta))^2}{\sharp \{ \beta : \Lambda^{-1} (\zeta) =
  \beta \}} d \mu (\zeta)\\
  & = & \int_{\overline{T}^2} (\mathbb{I} H^{1 / 2} (\beta))^2 d \Lambda^{\ast} \mu (\beta)\\
  & \lesssim & \sum_{\alpha\in T^2} H (\alpha)\\
  &  & \tmop{because} \Lambda^{\ast} \mu \tmop{is} \tmop{a} \tmop{trace}
  \tmop{measure} \tmop{for} \;\textup{the Hardy operator on the bitree}\; \nocomma\\
  & = & \sum_{\alpha\in T^2} \int_{D (z_1(\alpha), r) \times D
  (z_2(\alpha), s)} | \partial_{z_1 z_2} f (z_1, z_2) |^2 \nobracket d A (z_1) d
  A (z_2) \nobracket\\
  & = & \int_{\mathbb{D}^2} \sharp \{ \alpha : z \in D
  (z_1(\alpha), r) \times D (z_2(\alpha), s) \} | \partial_{z_1
  z_2} f (z_1, z_2) |^2 \nobracket \tmop{dA} (z_1) d A (z_2) \nobracket\\
  & \approx & \int_{\mathbb{D}^2} | \partial_{z_1 z_2} f (z_1, z_2) |^2 \nobracket d
  A (z_1) d A (z_2) \nobracket .
\end{eqnarray*}
Consider now $\vvar_1 (f) (\zeta_1) = \int_0^{\zeta_1} | \partial_{z_1} f (z_1, 0) |  |dz_1|$,
i.e. the radial variation of the function $z_1 \mapsto f (z_1, 0)$. By the one
variable result (see \cite{arsw2014}), we have that if $\mu_1$ is a Carleson measure for
$\mathcal{D} (\mathbb{D})$ then
\begin{equation}\label{e:422}
  \int_{\overline{\mathbb{D}}} \vvar_1 (f) (\zeta_1)^2 d \mu_1 (\zeta_1) \lesssim
   \frac{1}{\pi} \int_{\mathbb{D}} | \partial_{z_1} f (z_1, 0) |^2 d A (z_1). 
\end{equation}
Using the mean value property and Jensen's inequality we get
\[ | \partial_{z_1} f (z_1, 0) |^2 = \left| \frac{1}{2 \pi}  \int_{\partial\mathbb{D}}
   \partial_{z_1} f (z_1, e^{is}) ds \right|^2 \le \frac{1}{2 \pi} 
   \int_{\partial\mathbb{D}} | \partial_{z_1} f (z_1, e^{is}) |^2 ds. \]
Therefore
\[ \int_{\overline{\mathbb{D}}} \vvar_1 (f) (\zeta_1)^2 d \mu_1 (\zeta_1) \lesssim \frac{1}{2
   \pi^2}  \int_{\mathbb{D}} \int_{\partial\mathbb{D}} | \partial_{z_1} f (z_1, e^{is}) |^2
   dsdxdy \lesssim \|f\|^2_{\mathcal{D} (\mathbb{D}^2)} . \]
We are left to show that if $\mu$ is a Carleson measure for
$\mathcal{D} (\mathbb{D}^2)$, then the measure $\mu_1$ defined on every subset $A \subseteq \overline{\mathbb{D}}$ by
\[ \int_A d \mu_1 (\zeta_1) \assign \int_A \int_{\overline{\mathbb{D}}} d \mu
   (\zeta_1,\zeta_2) \]
is a Carleson measure for $\mathcal{D} (\mathbb{D})$. Let us prove the
implication in the discrete setting. By previous results, see \cite{arsw2014}, it is enough to show that $\Lambda_0^{\ast}
\mu_1$ is a trace measure for the Hardy operator on the tree. In the one
dimensional case we know that this is equivalent to requiring that
\begin{equation}
  \label{iff} \sum_{\beta_1 \le \alpha_1} \left(\Lambda_0^{\ast} \mu_1\right) (\mathcal{S}_T (\beta_1))^2 \lesssim
  \left(\Lambda_0^{\ast} \mu_1\right) (\mathcal{S}_T(\alpha_1))
\end{equation}
for any $\alpha_1 \in T$. Now
\[ \left(\Lambda_0^{\ast} \mu_1\right) (\mathcal{S}_T(\alpha_1)) = \int_{\mathcal{S}_T(\alpha_1) \times \overline{T}} d
   \Lambda^{\ast} \mu = \left(\Lambda^{\ast} \mu\right) (\mathcal{S}_T(\alpha_1)\times \overline{T}) . \]
Let $g \assign \chi_{\mathcal{S}_T(\alpha_1)\times \overline{T} }$. Then, for any $\beta = (\beta_1,\beta_2) \in T^2$
\[ \mathbb{I}^{\ast}_{\Lambda^{\ast} \mu} g (\beta) = \int \int_{\mathcal{S}_{T^2}(\beta) \cap
   (\mathcal{S}_{T}(\alpha_1)\times \overline{T})} d \Lambda^{\ast} \mu = \left(\Lambda^{\ast} \mu\right) (\mathcal{S}_{T^2} (\beta)
   \cap (\mathcal{S}_T(\alpha_1)\times \overline{T})) . \]
Hence, by the inequality for the adjoint operator we obtain
\begin{align*}
  \Lambda^{\ast} \mu (\mathcal{S}_{T}(\alpha_1)\times \overline{T}) & = \|g\|^2_{L^2  (\Lambda^{\ast} \mu)}
  \gtrsim \|\mathbb{I}^{\ast}_{\Lambda^{\ast} \mu} g\|^2_{\ell^2(T^2)} = \sum_{\beta\in T^2} \Lambda^{\ast} \mu (\mathcal{S}_{T^2}(\beta) \cap \left(\mathcal{S}_{T}(\alpha_1)\times \overline{T}\right))^2\\
 &\textup{by restricting to}\,\beta_1\geq\alpha_1,\,\beta_2=o \\ & \ge \sum_{\beta_1\geq\alpha_1} \Lambda^{\ast} \mu (\mathcal{S}_{T}(\beta_1)\times \overline{T})^2 = \sum_{\beta_1 \geq \alpha_1}
  \Lambda_0^{\ast} \mu_1 (\mathcal{S}_T(\beta_1))^2,
\end{align*}
thus proving inequality \eqref{iff}. Since $\Lambda_0^{\ast} \mu_1$ is a trace
measure for the Hardy operator on $T$ if and only if $\mu_1$ is Carleson for
$\mathcal{D} (\mathbb{D})$, \eqref{e:422} is proved.\\

The last term $| f (0, 0) |$ is elementary to treat. By the subcapacitary property, $(\Lambda^*\mu)(\overline{T}^2) \leq \capp(\overline{T}^2)C_{\mu}$, and therefore
\[
\int_{\overline{\mathbb{D}}^2}|f(0,0)|^2\,d\mu = |f(0,0)|^2\mu(\overline{\mathbb{D}}^2) = |f(0,0)|^2(\Lambda^*\mu)(\overline{T}^2) \lesssim C_{\mu}|f(0,0)|^2\leq C_{\mu}\|f\|^2_{\mathcal{D}(\mathbb{D}^2)}.
\]
We are done.\end{proof} 

It follows immediately that
\begin{proposition}
  Suppose $\mu$ on $\partial\mathbb{D}^2$ is a Carleson measure for the Dirichlet space on the bidisc. If $\| f
  \|^2_{\mathcal{D} (\mathbb{D}^2)} < \infty$, then
  \[ \lim_{R \rightarrow 1} f (R \zeta_1, R \zeta_2) = : f (\zeta_1, \zeta_2) \;
     \tmop{exists} \tmop{for} \mu\tmop{-a.e.} \; (\zeta_1, \zeta_2) \in  (\partial
     \mathbb{D})^2 \]
  and
  \[\lim_{R \rightarrow 1} \int_{(\partial \mathbb{D})^2} | f (R
     \zeta_1, R \zeta_2) |^2 d \mu (\zeta_1, \zeta_2) = \int_{(\partial \mathbb{D})^2} |
     f (\zeta) |^2 d \mu (\zeta) \leqslant \int_{(\partial
     \mathbb{D})^2} \vvar (f)^2 d \mu \leqslant K_1^2  \| f \|^2_{\mathcal{D}
     (\mathbb{D}^2)} . \]
\end{proposition}

\subsection{Characterization of multipliers: Theorem \ref{intromultiplier}}\label{SS:2.7}
We start by showing that the left-hand side of \eqref{introstegenga} dominates the right-hand side.
A standard argument with reproducing kernels shows that $\|b\|_{H^\infty} \leq \|M_b\|_{\BBB(\DDD(\DD^2))}$. Namely,
$$|b(z,w)| \|K_{(z,w)}\|_{\DDD(\DD^2)}^2 = |b(z,w)||K_{(z,w)}(z,w)| = \left|\langle bK_{(z,w)}, K_{(z,w)} \rangle_{\DDD(\DD^2)}\right| \leq \|M_b\|_{\BBB(\DDD(\DD^2))}\|K_{(z,w)}\|_{\DDD(\DD^2)}^2.$$
On the other hand, we may view $M_b$ as a vector-valued multiplier on the vector-valued Dirichlet space $\DDD(\DD) \otimes \DDD(\DD) \approxeq \DDD(\DDD(\DD))$.
That is, we identify the function $f(z,w)=\sum_{m,n}a_{m,n}z^mw^n$ with the function 
$F:\DD\to\DDD(\DD)$, $F:z\mapsto F(z)=\left(w\mapsto f(z,w)\right)$. Here we equip $\DDD(\DD)$ with the norm given in \eqref{introdef}, and $\DDD(\DDD(\DD))$ with the norm 
$$
\left\|F\right\|_{\DDD(\DDD(\DD))}^2:=\frac{1}{\pi}\int_\DD\left\|\frac{dF}{dz}(z)\right\|_{\DDD(\DD)}^2dA(z)
+\sup_{0<r<1}\frac{1}{2\pi}\int_{0}^{2\pi}\|F(re^{it})\|_{\DDD(\DD)}^2dt=\|f\|_{\DDD(\DD^2)}^2.
$$
The multiplicator with symbol $b=b(z,w)$ is then identified with
$$M_b \colon f(z, \cdot) \mapsto b(z, \cdot) f(z, \cdot).$$
For each fixed $z \in \DD$, let $M_{b(z, \cdot)}$ denote multiplication by $b(z, \cdot)$ on $\DDD(\DD)$, and $K_z$ the reproducing kernel of 
$\DDD(\DD)$ at $z$. 
Then, see \cite[\S 2.5]{AM02},
$$ M_b^\ast(K_z \otimes h) = K_z \otimes M_{b(z, \cdot)}^\ast h, \quad h \in \DDD(\DD),$$
and thus
\begin{equation} \label{eq:multarg1}
\sup_{z \in \DD}  \|M_{b(z, \cdot)}\|_{\BBB(\DDD(\DD))} \leq \|M_b\|_{\BBB(\DDD(\DD^2))}.
\end{equation}
We can apply Stegenga's characterization \cite{stegenga1980} of the multipliers of $\DDD(\DD)$ to the left hand side of (\ref{eq:multarg1}), yielding that
$$\sup_{z \in \DD} [\left|\partial_w b(z, \cdot)\right|^2dA(\cdot)]_{CM(\mathcal{D}(\mathbb{D}))} \lesssim \|M_b\|_{\BBB(\DDD(\DD^2))}^2.$$
Hence,
\begin{equation} \label{eq:multarg2}
 \int_{\DD} |\partial_w b(z,w)|^2 |f(z,w)|^2 \, dA(w) \lesssim \|M_b\|_{\BBB(\DDD(\DD^2))}^2 \left(\int_{\partial\mathbb{D}} | f(z,e^{it})|^2 \, dt + \int_{\DD} |\partial_{w} f(z,w)|^2 dA(w)\right),
\end{equation}
which by integration, standard properties of $H^2(\DD)$, Fatou's Lemma, and dominated convergence, yields that
\begin{eqnarray}\label{eq:multarg3}
 &&\int_{\partial\mathbb{D}} \int_{\DD} |\partial_w b(e^{is},w)|^2 |f(e^{is},w)|^2  dA(w)\, \,ds =
 \int_{\DD}\lim_{r\to1}\int_{\partial\mathbb{D}}  |\partial_w b(re^{is},w)|^2 |f(re^{is},w)|^2 ds\,  dA(w) \crcr
 &\le&\varliminf_{r\to1}\int_{\DD}\int_{\partial\mathbb{D}}  |\partial_w b(re^{is},w)|^2 |f(re^{is},w)|^2 ds\,  dA(w) \crcr
 &\lesssim&\|M_b\|_{\BBB(\DDD(\DD^2))}^2 \varliminf_{r\to1}\int_{\partial\mathbb{D}} \left(\int_{\partial\mathbb{D}} | f(re^{is},e^{it})|^2 \, dt + \int_{\DD} |\partial_{w} f(re^{is},w)|^2 dA(w)\right)\,ds  \crcr
 &\le& \|M_b\|_{\BBB(\DDD(\DD^2))}^2 \|f\|_{\DDD(\DD^2)}^2.
\end{eqnarray}
Applying \eqref{eq:multarg2} with $\partial_z f(z, \cdot)$ in place of $f$ and integrating also yields that
\begin{equation} \label{eq:multarg4}
\int_{\DD} \int_{\DD} |\partial_w b(z,w)|^2 |\partial_z f(z,w)|^2 \, dA(z) \, dA(w) \lesssim \|M_b\|_{\BBB(\DDD(\DD^2))}^2 \|f\|_{\DDD(\DD^2)}^2.
\end{equation}
Similarly, the inequalities \eqref{eq:multarg1}-\eqref{eq:multarg4} also hold with the roles of $z$ and $w$ reversed. 

Suppose $M_b$ is bounded. Writing out the norm of $bf$, $f \in \DDD(\DD^2)$, applying the triangle inequality, the fact that $b \in H^\infty(\DD^2)$, and 
inequality \eqref{eq:multarg4} yields that
\begin{multline*}
\int_{\DD^2}|\partial_{zw}(bf)|^2dA(z)\,dA(w) \\ 
\gtrsim \int_{\DD^2}|(\partial_{zw}b)f|^2dA(z)\,dA(w)-
\int_{\DD^2}\left[\left|(\partial_zb)\partial_wf\right|^2+\left|(\partial_wb)\partial_zf\right|^2+\left|b\partial_{zw}f\right|^2\right]dA(z)\,dA(w) \\
\gtrsim \int_{\DD^2}|(\partial_{zw}b)f|^2dA(z)\,dA(w)-\|M_b\|_{\BBB(\DDD(\DD^2))}^2 \|f\|_{\DDD(\DD^2)}^2-\|b\|_{H^\infty}^2\|f\|_{\DDD(\DD^2)}^2.
\end{multline*}
Hence,
$$\|M_b f\|_{\DDD(\DD^2)}^2 \gtrsim \int_{\DD} \int_{\DD} |f|^2 \, d\mu_b - \|M_b\|_{\BBB(\DDD(\DD^2))}^2\|f\|_{\DDD(\DD^2)}^2,$$
and thus 
$$[\mu_b]_{CM} \lesssim \|M_b\|_{\BBB(\DDD(\DD^2))}^2.$$
The computations thus far have shown that the left-hand side of \eqref{introstegenga} dominates the right-hand side. 
The converse inequality is also clear, using the triangle inequality, from the estimates we have made.

\section{Strong Capacitary Inequality on the bitree}\label{S:5}
Here we prove Theorem \ref{capacitarystrong}. First we establish some extra notation. Similarly to the definitions in Section \ref{SS:2.35} we define the one-dimensional Hardy operator, its adjoint, and the logarithmic potential on the tree $T$:
\begin{equation}\notag
\begin{split}
&(I\varphi)(\alpha) := \sum_{\gamma\geq\alpha}\varphi(\gamma);\\
&(I^*\mu)(\beta) := \int_{\mathcal{S}_T(\beta)}\,d\mu(\alpha);\\
&V^{\mu} := (II^*)(\mu),
\end{split}
\end{equation}
where $\varphi\geq0$ is a function on $T$ and $\mu\geq0$ is a Borel measure on $\overline{T}$. The (one-dimensional) logarithmic capacity is defined in the same way,
\[
\capp E = \inf\{\|\varphi\|^2_{\ell^2(T)}: \; I\varphi \geq 1\;\textup{on}\; E\}, 
\]
for a Borel set $E\subset \overline{T}$.
We aim to prove the following result:
\begin{theorem}\label{t:th5.1}
For any $f:T^2\mapsto \mathbb{R}_+$ in $\ell^2(T^2)$ we have
\begin{equation}\label{e:th5.1}
\sum_{k\in\mathbb{Z}}2^{2k}\capp\{\alpha\in \overline{T}^2: \mathbb{I}f(\alpha) \geq 2^{k}\} \lesssim \sum_{\alpha\in T^2}f^2(\alpha) := \|f\|_{\ell^2(T^2)}^2.
\end{equation}
\end{theorem}
Similar results were obtained by Adams \cite{adams1976}, Maz'ya \cite{mazya1972}, and others. However, they were based on a certain property of the respective potential-theoretic kernels --- that they were of 'radial nature'. In our context this roughly translates to the uniqueness of the geodesic between two different vertices of the underlying graph. While this property is elementary for a uniform dyadic tree $T$ (as well for a $p$-adic tree), the bitree $T^2$ does not enjoy it any more, and this is one of the main problems we have to overcome when we increase the dimension.

To highlight this difference we give a rough sketch of the proof for $d=1$ (i.e. for a dyadic tree). Given $k\in\mathbb{Z}$ assume that $If(\alpha) \approx 2^{k+1}$ and $If(\beta)\approx 2^{k}$ for $\beta>\alpha$ in the tree. Then, since there exists a unique geodesic connecting $\alpha$ and $\beta$, we have that $\sum_{\beta>\tau\geq\alpha}f(\tau)\approx 2^k$. Therefore one can expect, if we set $f_k:=\chi_{2^k<If\leq2^{k+1}}f$, that $If_k\approx 2^k\approx If$ on $\{If\geq 2^{k+1}\}$, so that $2^{-k+1}f_k$ is admissible for this set, and $\sum_{k\in\mathbb{Z}}2^{2k}\capp\{If\geq 2^{k+1}\} \lesssim \sum_{k\in\mathbb{Z}}\|f_k\|^2 \approx \|f\|^2$, since the functions $f_k$ have disjoint supports.

However we see that already for $d=2$ there are many geodesics in $T^2$ with endpoints at $\alpha$ and $\beta$, and the above argument fails, since one can construct a function $f_0$ such that $\mathbb{I}f_0(\alpha) \leq 1$ for every $\alpha\in\supp f_0$, but for every $\lambda > 1$ there exists a point $\beta\in T^2$ with $\mathbb{I}f_0(\beta) > \lambda$ (see Proposition \ref{p:A.2.1}). In other words, the maximum principle does not hold for $T^2$. However, while it fails pointwise, a quantitative version of the maximum principle is still true --- the set of 'bad'  points has asymptotically small capacity, see Corollary \ref{c:c5.1}. Therefore we can salvage enough of the argument to obtain Theorem \ref{t:th5.1}.

The proof is based on the following rearrangement lemma, Lemma \ref{l:l5.1}. We explain how it implies Theorem \ref{t:th5.1} in Section \ref{ss:5.ii}. Lemma \ref{l:l5.1} is proved in Section \ref{ss:5.iii} by reduction to a one-dimensional statement, Lemma \ref{l:l5.2}, which in turn is proved in Section \ref{ss:5.iv}.

\begin{lemma}[Rearrangement Lemma]\label{l:l5.1}
Let $\mu\geq0$ be any Borel measure with finite energy on $(\partial T)^2$. Given $\delta>0$ we define the $\delta$-level set of $\mathbb{V}^{\mu}$ by
\[
E^{\delta} := \{\alpha\in T^2:\; \mathbb{V}^{\mu}(\alpha) \leq \delta\}.
\]
We also define the $\delta$-restricted potential and energy by setting
\[
\mathbb{V}_{\delta}^{\mu}(\alpha) := \sum_{\beta\in E^{\delta}:\;\beta\geq\alpha}(\mathbb{I}^*\mu)(\beta);\;\;\;\; \mathcal{E}_{\delta}[\mu] := \sum_{\alpha\in E^{\delta}}(\mathbb{I}^*\mu)^2(\alpha) = \int_{(\partial{T})^2}\mathbb{V}^{\mu}_{\delta}\,d\mu.
\]
For $\lambda\geq\delta$ let 
\[
E_{\delta,\lambda} := \{\alpha\in (\partial T)^2:\; \mathbb{V}_{\delta}^{\mu}(\alpha)> \lambda\}.
\]
Then there exists a function $\varphi: \overline{T}^2\rightarrow\mathbb{R}_+$ supported on $T^2$ such that
\begin{subequations}\label{e:l5.1s}
\begin{eqnarray}
\label{e:l5.1s.1}& \mathbb{I}\varphi(\alpha) = \sum_{\beta \geq\alpha} \varphi(\beta) > \lambda, \quad \alpha \in E_{\delta,\lambda};\\
\label{e:l5.1s.2}& \|\varphi\|_{\ell^2(T^2)}^2 = \sum_{\alpha\in T^2} \left(\varphi(\alpha)\right)^2 \lesssim \frac{\delta}{\lambda}\mathcal{E}_{\delta}[\mu] = \frac{\delta}{\lambda}\sum_{\alpha\in E^{\delta}}\left(\mathbb{I}^*\mu(\alpha)\right)^2.
\end{eqnarray}
\end{subequations}
\end{lemma}
\textbf{Remark.} Observe that, by the maximum principle, $E_{\delta,\lambda} = \emptyset$ in the one-dimensional setting of a tree $T$.
\begin{corollary}\label{c:c5.1}
Let $E\subset (\partial T)^2$ be a Borel set, and $\mu = \mu_E$ be the equilibrium measure for $E$. Given $\lambda>1$ define 
\begin{equation}\label{e:c5.1c}
E_{\lambda} := \{\omega\in(\partial T)^2: \mathbb{V}^{\mu}(\omega) = \sum_{\beta\in \mathcal{P}(\omega)}\mathbb{I}^*\mu(\beta) > \lambda\}.
\end{equation}
Then
\begin{equation}\label{e:l5.1s2}
\capp E_{\lambda} \lesssim \frac{1}{\lambda^3}\capp E.
\end{equation}
\end{corollary}
\begin{proof}
Put $\delta = 1$. Since $\mu$ is equilibrium for $E$, we have $\{\alpha\in T^2:\; (\mathbb{I}^*\mu)(\alpha)>0\}\subset E^1$ and $E_{1,\lambda} = E_{\lambda}$. It remains to apply Lemma \ref{l:l5.1} with data $1,\lambda$.
\end{proof}

\subsection{Deducing Theorem \ref{t:th5.1} from Corollary \ref{c:c5.1}}\label{ss:5.ii}
Here we mostly follow the argument from Adams and Hedberg \cite[Chapter 7]{ah1996}. First we separate the $\ell^2$-norm of $f$, reducing \eqref{e:th5.1} to estimates of the level sets of $\mathbb{I}f$. We then prove that the energy scalar product of two equilibrium measures can be estimated by the capacities of the respective sets, Lemma \ref{l:l5.2.1}. This is the key point of the argument, and it is here that we use the Rearrangement Lemma \ref{l:l5.1} (or, more precisely, Corollary \ref{c:c5.1}). We finish the proof by showing that the mixed energy of the level sets (energy scalar product of their equilibrium measures) is concentrated on the diagonal (inequality \eqref{e:35}).

\noindent \textit{Removing $\|f\|_{\ell^2(T^2)}$}.

Given $k\in\mathbb{Z}$ let $\tilde{E}_k$ be the $k$-th level set of $\mathbb{I}f$,
\[
\tilde{E}_k = \{\alpha\in \overline{T}^2: (\mathbb{I}f)(\alpha) > 2^k\}.
\]
We then define $E_k$ to be the boundary projection of $\tilde{E}_k$
\[
E_k = \mathcal{S}(\tilde{E}_k)\cap (\partial T)^2,
\]
where $\mathcal{S}(\tilde{E}_k) = \bigcup_{\beta\in \tilde{E}_k}\mathcal{S}(\beta)$.
Corollary \ref{c:A.3} (see Section \ref{S:A.3}) implies that
\[
\capp \tilde{E}_k \approx \capp E_k,\quad k\in \mathbb{Z}.
\]
Hence \eqref{e:th5.1} is equivalent to
\[
\sum_{k\in\mathbb{Z}}2^{2k}\capp E_k \lesssim  \|f\|_{\ell^2(T^2)}^2.
\]
Since $E_k\subset \mathcal{S}(\tilde{E}_k)$, we see that $\mathbb{I}f \geq 2^k$ on $E_k$ as well. By its nature, $E_k$ is a countable union of clopen rectangles on $(\partial T)^2$, hence $E_k = \bigcup_{j=1}^{\infty}E_k^j$ with $\{E_k^j\}$ is an increasing sequence of compact sets. Define $\mu_k$ and $\mu_k^j$ to be equilibrium measures for $E_k$ and $E_k^j$ respectively. Clearly $\lim_{j\rightarrow\infty}\capp E_k^j = \capp E_k$. For $k\in \mathbb{Z}$ and $j\geq1$ we have
\begin{equation}\notag
2^{2k}\capp E_k^j = 2^k\int_{\overline{T}^2}2^k\,d\mu_k^j \leq 2^k\int_{\overline{T}^2}\mathbb{I}f\,d\mu_k^j = 2^k\int_{T^2}f\,d(\mathbb{I}^*\mu_k^j) = 2^k\sum_{\alpha\in T^2}f(\alpha)(\mathbb{I}^*\mu_k^j)(\alpha)
\end{equation}
by Tonelli's theorem. Since $\mathbb{I}^*\mu^j_k \rightarrow \mathbb{I}^*\mu_k$ in $\ell^2(T^2)$ \cite[Proposition 2.3.12]{ah1996}, we may pass to the limit as $j\rightarrow\infty$, obtaining
\begin{equation}\notag
2^{2k}\capp E_k \leq 2^k\int_{T^2}f\,d(\mathbb{I}^*\mu_k).
\end{equation}
Summing this estimate over $k\in\mathbb{Z}$ and applying Cauchy-Schwartz  we arrive at
\[
\sum_{k\in\mathbb{Z}}2^{2k}\capp E_k \leq \int_{{T}^2}f\sum_{k\in\mathbb{Z}}2^{k}\,d(\mathbb{I}^*\mu_k) \leq \|f\|_{\ell^2(T^2)}\|\sum_{k\in\mathbb{Z}}2^{k}\,\mathbb{I}^*\mu_k\|_{\ell^2(T^2)}.
\]
 We conclude that \eqref{e:th5.1} follows from
\begin{equation}\label{e:34}
\|\sum_{k\in\mathbb{Z}}2^{k}\,\mathbb{I}^*\mu_k\|^2_{\ell^2({T}^2)} \lesssim \sum_{k\in\mathbb{Z}}2^{2k}\capp E_k.
\end{equation}

\noindent \textit{Equilibrium potential on the subset.}

Expanding the left-hand side of \eqref{e:34} we obtain
\[
\sum_{k\in\mathbb{Z}}\sum_{j\in\mathbb{Z}}2^{j+k}\int_{\overline{T}^2}\mathbb{V}^{\mu_k}\,d\mu_j,
\]
and this expression is symmetric over $j$ and $k$. Therefore \eqref{e:34} is equivalent to
\begin{equation}\label{e:35}
\sum_{k\in\mathbb{Z}}\sum_{j\leq k}2^{j+k}\int_{\overline{T}^2}\mathbb{V}^{\mu_k}\,d\mu_j \lesssim \sum_{k\in\mathbb{Z}}2^{2k}\int_{\overline{T}^2}\mathbb{V}^{\mu_k}\,d\mu_k,
\end{equation}
since $\capp E_k = |\mu_k| = \int_{\overline{T}^2}\mathbb{V}^{\mu_k}\,d\mu_k$. We see that in order to prove \eqref{e:th5.1} we need to show that the sum on the left-hand side of \eqref{e:35} is dominated by its diagonal term. First we state the following lemma.
\begin{lemma}\label{l:l5.2.1}
Let $F, E \subset (\partial T)^2$ be a pair of sets on the distinguished boundary of $T^2$, such that $\capp F\leq\capp E$, and let $\mu_F,\mu_E$ be their equilibrium measures. Then there exists an absolute constant $C>1$ such that
\begin{equation}\label{l:l5.2.1s}
\int_{\overline{T}^2}\mathbb{V}^{\mu_E}\,d\mu_F \leq C|\mu_E|^{\frac13}|\mu_F|^{\frac23} = C (\capp E)^{\frac13}(\capp F)^{\frac23}.
\end{equation}
\end{lemma}
\begin{proof}
If $\capp F=0$, then \eqref{l:l5.2.1s} is trivial. Therefore we let
\[
\lambda = \frac{\int_{\overline{T}^2}\mathbb{V}^{\mu_E}\,d\mu_F}{|\mu_F|},
\]
and we assume that $\lambda \geq 16$ (otherwise we just set $C$ to be large enough). Define $F_k,\; k\in\mathbb{Z}_+,$ to be the level sets of $\mathbb{V}^{\mu_E}$ on $\textup{cl}F$
\begin{equation}\label{notag}
\begin{split}
&F_k := \{\alpha\in \textup{cl}F:\; 2^{k-1} \leq \mathbb{V}^{\mu_E}(\alpha)< 2^{k}\},\quad k\geq1\\
&F_0 := \{\alpha\in \textup{cl}F:\; \mathbb{V}^{\mu_E} < 1\}.
\end{split}
\end{equation}
Corollary \ref{c:c5.1} implies that
\[
\capp F_k \leq C 2^{-3k}\capp E
\]
for some $C>0$. Clearly $F_k,\; k\geq0$, are disjoint, and $\capp\left(\textup{cl}F\setminus\bigcup_{k=0}^{\infty}F_k\right) = 0$ (since the potential of $\mu_E$ can be infinite only on a polar set). Hence we have
\begin{equation}\notag
\begin{split}
&\lambda|\mu_F| = \int_{\overline{T}^2}\mathbb{V}^{\mu_E}\,d\mu_F = \sum_{k\geq 0}\int_{F_k}\mathbb{V}^{\mu_E}\,d\mu_F \leq\\
& \sum_{k\geq 0}\int_{F_k}2^{k}\,d\mu_F=
\sum_{k\geq 0}2^{k}\mu_F(F_k).
\end{split}
\end{equation}
Fix $j$ such that $2^{j-1}\leq \lambda < 2^{j}$. Since $\mu_F(F_k) = \mu_F(F_k\cap\supp\mu_F) \leq \capp F_k$ (see Lemma \ref{l:A.1.1}), we get
\begin{equation}\notag
\begin{split}
&\lambda|\mu_F| \leq \sum_{k=0}^{j-3}2^{k}\mu_F(F_k) + \sum_{k\geq j-2}2^{k}\capp F_k\leq\\
&2^{j-2}\mu_F(\textup{cl}F) + C\sum_{k\geq j-2}2^{-2k}\capp E \leq \frac{\lambda}{2}\mu_F (\textup{cl}F) + C\lambda^{-2}\capp E.
\end{split}
\end{equation}
Hence 
\[
|\mu_F| \leq C\lambda^{-3}|\mu_E|,
\]
which immediately implies \eqref{l:l5.2.1s}.
\end{proof}

As we will see below, in order to prove the Strong Capacitary Inequality it actually suffices to show that 
\[
\int_{\overline{T}^2}\mathbb{V}^{\mu_E}\,d\mu_F \leq |\mu_E|^{\frac12-\varepsilon}|\mu_F|^{\frac12+\varepsilon},
\]
for some $\varepsilon>0$. H\"{o}lder's inequality gives $|\mu_E|^{\frac12}|\mu_F|^{\frac12}$ on the right-hand side (which is not good enough). On the other hand, in the tree setting, or, more generally, in any setting where the Maximum Principle holds, one has much better estimate
\[
\int_{\overline{T}^2}\mathbb{V}^{\mu_E}\,d\mu_F \lesssim |\mu_F|.
\]

\noindent \textit{Estimates of the potentials of the level sets.}

Applying Lemma \ref{l:l5.2.1} to the sets $E_k$ and $E_j,\; k\geq j$, yields
\[
\int_{\overline{T}^2}\mathbb{V}^{\mu_j}\,d\mu_k \lesssim |\mu_k|^{\frac23}|\mu_j|^{\frac13},
\]
since $E_k\subset E_j$. By H\"{o}lder's inequality
\begin{equation}\notag
\begin{split}
&\sum_{k\in\mathbb{Z}}\sum_{j\leq k}2^{k+j}\int_{\overline{T}^2}\mathbb{V}^{\mu_j}\,d\mu_k \lesssim \sum_{k\in\mathbb{Z}}\sum_{j\leq k}2^{k+j}|\mu_k|^{\frac23}|\mu_j|^{\frac13} = \\
&\sum_{k\in\mathbb{Z}}2^{\frac43k}|\mu_k|^{\frac23}\cdot 2^{-\frac13k}\sum_{j\leq k}2^j|\mu_j|^{\frac13} \leq \left(\sum_{k\in\mathbb{Z}}2^{2k}|\mu_k|\right)^{\frac23}\left(\sum_{k\in\mathbb{Z}}2^{-k}\left(\sum_{j\leq k}2^j|\mu_j|^{\frac13}\right)^3\right)^{\frac13}.
\end{split}
\end{equation}
Another application of H\"{o}lder's inequality to the second term on the right-hand side gives
\begin{equation}\notag
\begin{split}
&\sum_{k\in\mathbb{Z}}2^{-k}\left(\sum_{j\leq k}2^j|\mu_j|^{\frac13}\right)^3 = \sum_{k\in\mathbb{Z}}2^{-k}\left(\sum_{j\leq k}2^{\frac16j}2^{\frac{5}{6}j}|\mu_j|^{\frac13}\right)^3 \lesssim\\
&\sum_{k\in\mathbb{Z}}2^{-k}\sum_{j\leq k}2^{\frac12j}\sum_{j\leq k}2^{\frac52j}|\mu_j| \lesssim \sum_{k\in\mathbb{Z}}2^{2k}|\mu_k|.
\end{split}
\end{equation}
Gathering the estimates we obtain
\[
\sum_{k\in\mathbb{Z}}\sum_{j\leq k}2^{k+j}\int_{\overline{T}^2}\mathbb{V}^{\mu_j}\,d\mu_k \lesssim  \sum_{k\in\mathbb{Z}}2^{2k}|\mu_k|,
\]
which is \eqref{e:35}.

We note that in the last part of the proof we did not use the fact that the sets $E_k$ are generated by the function $f$. Indeed, \eqref{e:35} holds for any nested sequence $\{E_k\}$ of sets on the distinguished boundary.
\subsection{Rearrangement and the energy decay}\label{ss:5.iii.5}
Before proceeding to the proof of Rearrangement Lemma we show how one can also deduce the energy decay rate of a measure outside its support. 
See e.g.\cite{AHMV18bis} for some other applications.
\begin{proposition}\label{l:l5.11}
Assume $\mu\geq0$ is a Borel measure on $(\partial T)^2$ such that $\mathbb{V}^{\mu} \geq 1$ on $\supp\mu$. Then for any $0<\delta \leq 1$ one has 
\begin{equation}\notag
\mathcal{E}_{\delta}[\mu] \lesssim \delta^{\frac13}\mathcal{E}[\mu].
\end{equation}
In particular, there exists $\delta>0$ such that
\[
\mathcal{E}[\mu] - \mathcal{E}_{\delta}[\mu] = \sum_{\alpha:\;\mathbb{V}^{\mu}(\alpha)\geq\delta}(\mathbb{I}^*\mu)^2(\alpha) \geq \frac{9}{10}\mathcal{E}[\mu].
\]
\end{proposition}
\begin{proof}
Fix $\delta\in (0,1]$, and let $\eta_{\delta}$ be such that
\[
\eta_{\delta} := \frac{\mathcal{E}_{\delta}[\mu]}{\mathcal{E}[\mu]}.
\]
Given $k\in\mathbb{Z}$ let $F_k := \{\alpha\in\supp\mu:\; 2^{-k}\geq \mathbb{V}_{\delta}(\alpha) > 2^{-k-1}\}$.  Applying the rearrangement procedure to $\lambda = 2^{-k}\geq \delta$ we obtain a function $\varphi_k$ that satisfies \eqref{e:l5.1s}. In particular, $\mathbb{I}\varphi_k \geq 2^{-k}$ on $F_k$, therefore by Tonelli's theorem and \eqref{e:l5.1s.2} one has
\begin{equation}\notag
\begin{split}
&2^{-k}\mu(F_{k}) = \int_{F_k}2^{-k}\,d\mu \leq \int_{F_k}\mathbb{I}\varphi_k\,d\mu = \int_{(\partial T)^2}\chi_{F_k}(\omega)\left(\sum_{\beta\geq\omega}\varphi_k(\beta)\right)\,d\mu(\omega) \leq\\ &\int_{(\partial T)^2}\left(\sum_{\beta\geq\omega}\varphi_k(\beta)\right)\,d\mu(\omega) =
\sum_{\beta\in T^2}\varphi_k(\beta)(\mathbb{I^*}\mu)(\beta) \leq \left(\sum_{\beta\in T^2}\varphi_k^2(\beta)\right)^{\frac{1}{2}}\left(\sum_{\beta\in T^2}(\mathbb{I}^*\mu)^2(\beta)\right)^{\frac{1}{2}} \leq \\ &C_0\left(2^k\delta\mathcal{E}_{\delta}[\mu]\right)^{\frac12}\mathcal{E}^{\frac{1}{2}}[\mu] \leq C_02^{\frac{k}{2}}\delta^{\frac{1}{2}}\mathcal{E}[\mu]
\end{split}
\end{equation}
for some absolute constant $C_0>1$. Set $N = N(\delta)$ to be such that $2^{-N+1} \geq \frac{1}{(10C_0)^2}\eta_{\delta} \geq 2^{-N}$. If $\eta_{\delta} \leq 2(10C_0^2)^2\delta$, then the result follows immediately, hence we may assume that $2^{-N} \geq \delta$.
Summing up over $k$ we realize that
\begin{multline*}
\mathcal{E}_{\delta}[\mu] = \int_{(\partial T)^2}\mathbb{V}_{\delta}^{\mu}\,d\mu \leq2\sum_{k=-\infty}^{+\infty}\int_{F_k}\mathbb{V}_{\delta}^{\mu}\,d\mu \\ \leq 2\sum_{k=-\infty}^N2^{-k}\mu(F_k) + 2^{-N+1}\mu \left( \left\{\alpha\in \supp\mu:\; \mathbb{V}_{\delta}^{\mu}(\alpha)\leq 2^{-N}\right\}\right) \\ \leq
2\sum_{k=-\infty}^N2^{-k}\mu(F_k) + 2^{-N+1}|\mu| \leq 2\sum_{k=-\infty}^N2^{-k}\mu(F_k) + 2^{-N+1}\mathcal{E}[\mu],
\end{multline*}
since $\mathcal{E}[\mu] = \int_{(\partial T)^2}\mathbb{V}^\mu\,d\mu$ and $\mathbb{V}^{\mu} \geq1$ on the support of $\mu$. Therefore
\[
\mathcal{E}_{\delta}[\mu] = \eta_{\delta}\mathcal{E}[\mu] \leq 2^{-N+1}\mathcal{E}[\mu] + 2C_0\sum_{k\leq N}2^{\frac{k}{2}}\delta^{\frac12}\mathcal{E}[\mu],
\]
hence
\[
\eta_{\delta} \leq \frac{1}{5}\eta_{\delta} + 4C_02^{\frac{N}{2}}\delta^{\frac{1}{2}} \leq\frac{1}{5}\eta_{\delta} + \frac45\eta_{\delta}^{-\frac12}\delta^{\frac12},
\]
and we are done.
\end{proof}
\subsection{Proof of Lemma \ref{l:l5.1}. Reducing dimension.}\label{ss:5.iii}

We assume that $\delta = 1$ (otherwise we just rescale by replacing $\lambda$ by $\lambda/\delta$ in the $\delta=1$ statement), and from now on we write $E_{\lambda}$ instead of $E_{1,\lambda}$.

We construct the function $\varphi$ that satisfies \eqref{e:l5.1s}. It is done separately on each layer of the form $T_x\times\{\alpha_y\},\;\alpha_y\in T_y$ (although the statement we wish to prove is symmetric in the two components $x$ and $y$, our argument is not, and it can be of course developed exchanging their roles).
More precisely, for every $\alpha_y\in T_y$ we produce a function $\varphi_{\alpha_y}: \overline{T}^2\rightarrow\mathbb{R}_+$ such that \textit{(a)} it is supported on the layer $R = T_x\times\{\alpha_y\}$; \textit{(b)} on a certain subset $A_R$ of the set $E_{\lambda}\cap(\partial T_x\times \mathcal{S}(\alpha_y))$ it gives at least as much potential as $\mu$ restricted to this layer 
\begin{equation}\label{e:41.a}
(\mathbb{I}\varphi_{\alpha_y})(\omega) = \sum_{\alpha\geq\omega}\varphi_{\alpha_y}(\alpha) = \sum_{\alpha_x\geq\omega_x}\varphi_{\alpha_y}(\alpha_x,\alpha_y) \geq \sum_{\alpha_x\geq\omega_x:\;\alpha\in E^{1}}(\mathbb{I}^*\mu)(\alpha_x,\alpha_y),\quad \omega\in E_{R};
\end{equation}
\textit{(c)} its $\ell^2$-norm is much smaller than the energy of $\mu\vert_{R}$
\begin{equation}\label{e:41.b}
\|\varphi_{\alpha_y}\|^2_{\ell^2(T^2)} = \sum_{\alpha\in T^2}\varphi_{\alpha_y}^2(\alpha_x,\alpha_y) =\sum_{\alpha_x\in T_x}\varphi_{\alpha_y}^2(\alpha_x,\alpha_y) \lesssim \frac{1}{\lambda}\sum_{\alpha_x\in T_x:\;\alpha\in E^1}(\mathbb{I}^*\mu)^2(\alpha_x,\alpha_y),
\end{equation}
where we have recalled that $\varphi_{\alpha_y}$ is supported only on $T_x\times \{\alpha_y\}$.
Each layer $T_x\times\{\alpha_y\}$ is essentially a dyadic tree, and the (restricted) potential of $\mu$ exhibits one-dimensional behaviour there, so we can consider the problem in the dyadic tree setting and use one-dimensional arguments.

Finally we set $\varphi = \frac32\sum_{\alpha_y\in T_y}\varphi_{\alpha_y}$, and show that $\varphi$ satisfies \eqref{e:l5.1s.1} and \eqref{e:l5.1s.2}; the second inequality immediately follows from \eqref{e:41.b}, since $T^2 = \bigcup_{\alpha_y\in T_y}T_x\times\{\alpha_y\}$ and the supports of $\varphi_{\alpha_y}$ are disjoint.

\noindent \textit{Construction of $\varphi_{\alpha_y}$}\\
Given $\gamma\in \overline{T}$ we define $\partial\mathcal{S}(\gamma) := \mathcal{S}(\gamma)\cap\partial T$ to be the boundary successor set of $\gamma$.
Fix a point $\alpha_y\in T_y$, and let
\[
A_R = \{\omega\in E_{\lambda}\cap (\partial T_x\times \partial\mathcal{S}(\alpha_y)): \mathbb{V}^{\mu}_1(\omega_x,\alpha_y) > \frac{\lambda}{3}\}.
\]
In other words, $\omega\in E_{\lambda}$ is in $A_R$, if $\omega_y\leq \alpha_y$, and the (restricted) potential of $\mu$ at the 'fiber' $(\omega_x,\alpha_y)$ is large enough. Define $F_R$ to be the projection of $A_R$ on the coordinate tree $T_x$,
\[
F_{R} = \{\omega_x\in\partial T_x:\; \textup{there exists}\; \omega_y\leq\alpha_y\;\textup{such that}\; (\omega_x,\omega_y)\in A_R\}.
\]
Observe that $F_R$ is an open set in $T_x$.
\begin{SCfigure}[][h]
\centering
\caption{\;\\$\times$: points in $E_{\lambda}$\\
$\otimes$: points in $A_R$\\
$\bigcirc$: points in $F_R$}

\includegraphics[height=8cm]{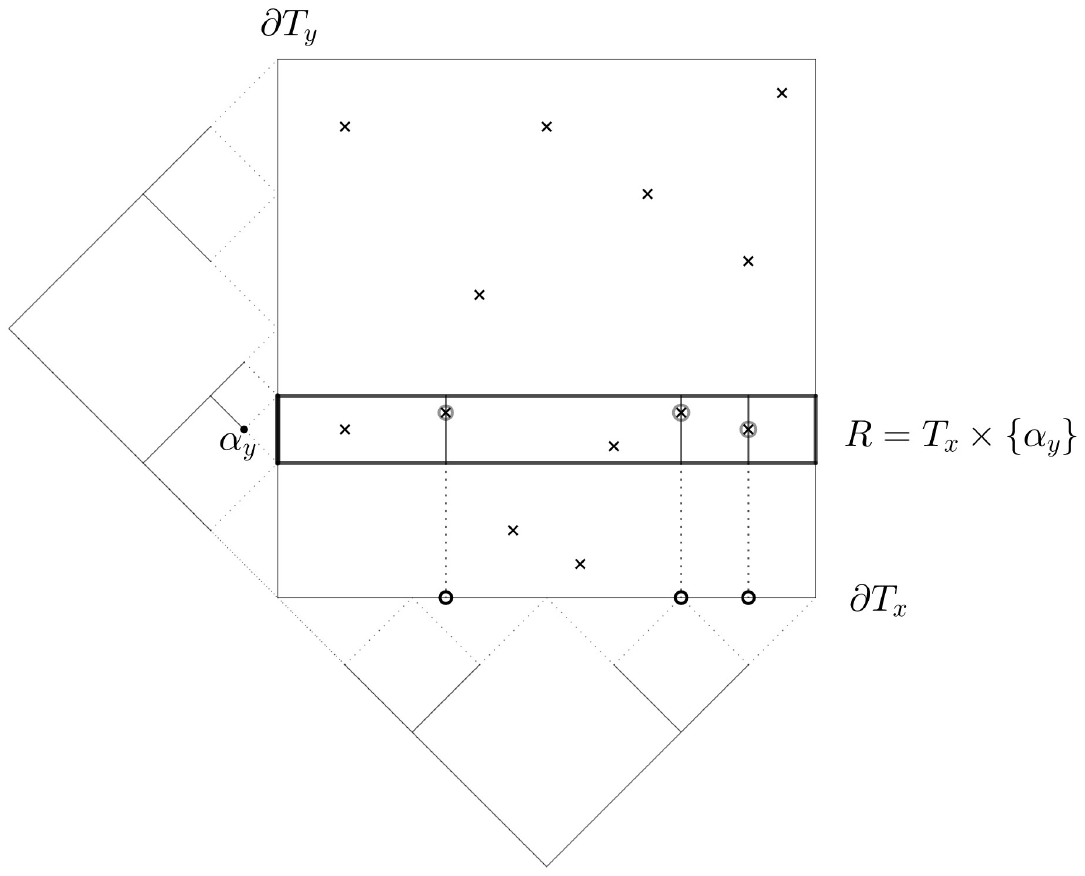}
\end{SCfigure}

We proceed by performing the dimension reduction argument --- we restate our problem on the dyadic tree $T_x$. To do so we introduce auxiliary functions $f_R$ and $g_R$ supported on $T_x$. Let
\[
f_R(\beta_x) :=  \int_{\partial\mathcal{S}(\beta_x)}\int_{\partial\mathcal{S}(\alpha_y)}d\mu(\omega_x,\omega_y) = (\mathbb{I}^*\mu)(\beta_x,\alpha_y),
\]
if $(\beta_x,\alpha_y) \in E^1$ (i.e. if $\mathbb{V}^{\mu}(\beta_x,\alpha_y) \leq 1$),
and $f_R(\beta_x) := 0$ otherwise. Next,
\[
g_R(\beta_x) := \sum_{\beta_y\geq\alpha_y:\; (\beta_x,\beta_y)\in E^1}\int_{\partial\mathcal{S}(\beta_x)}\int_{\partial\mathcal{S}(\beta_y)}d\mu(\omega_x,\omega_y) = \sum_{\beta_{y}\geq\alpha_y:\;(\beta_x,\beta_y)\in E^1}(\mathbb{I}^*\mu)(\beta_x,\beta_y).
\]
Therefore 
\[
(Ig_R)(\alpha_x) = \sum_{\beta_x\geq\alpha_x}g_R(\beta_x) = \sum_{\beta\geq\alpha:\; \beta\in E^1}(\mathbb{I}^*\mu)(\beta) = \mathbb{V}^{\mu}_1(\alpha_x,\alpha_y),\quad\alpha_x\in \overline{T}_x,\; \alpha = (\alpha_x,\alpha_y),
\]
and, in particular,
\begin{equation}\label{e:421}
(Ig_R)(\omega_x) = \mathbb{V}^{\mu}_1(\omega_x,\alpha_y)>\frac{\lambda}{3},
\end{equation}
if $\omega_x\in F_R$. On the other hand, if $\alpha_x \in \supp f_R$, then by definition of $E^1$ one has
\begin{equation}\label{e:43}
(Ig_R)(\alpha_x) \leq 1.
\end{equation}
Next we present some crucial properties of $f_R$ and $g_R$.
\begin{lemma}\label{l:auxl}
Let $f_R$ and $g_R$ be as above. Given $\alpha_x\in T_x$ one always has $f_R(\alpha_x) \geq f_R(\alpha_x^-) + f_R(\alpha_x^+)$, where $\alpha_x^{\pm}$ are two children of $\alpha_x$. In particular, the function $If_R$ is positive superharmonic on $T_x$ (i.e. $(If_R)(\alpha_x)\geq \frac{1}{3}\left((If_R)(\alpha_x^+) + (If_R)(\alpha_x^-) + (If_R)(p(\alpha_x))\right)$), and for any $\alpha_x\in T_x$ either $f_R(\alpha_x) =0$, or $f(\beta_x) > 0$ for any $\beta_x\geq\alpha_x$. The same is true for $g_R$.
\end{lemma}
\begin{proof}
All of these properties immediately follow from the definition of $f_R$ an $g_R$, and the fact that if $(\beta_x,\beta_y)\in E^1$, then $(\gamma_x,\gamma_y)$  is in $E^1$ as well for any $\gamma_x\geq\beta_x,\; \gamma_y\geq\beta_y$.
\end{proof}

The inequalities \eqref{e:421} and \eqref{e:43} show that $F_R$ is, in a sense, far away from the support of $f_R$, and we can express this property only in terms of $F_R$ and $\supp f_R$. Since $F_R$ is open in $T_x$, we can exhaust it by compacts, i.e. there exists an increasing sequence of compact sets $F_k$ such that $F = \bigcup_{k=1}^{\infty}F_k$. Define $\rho_k$ and $\rho_R$ to be equilibrium measures for $F_k$ and $F_R$. We have $\frac{\lambda}{3}V^{\rho_k} \leq Ig_R$ on $F_k$. By the Domination Principle, given in Lemma \ref{l:A.1.1}, it follows that
\[
\frac{\lambda}{3}V^{\rho_k} \leq Ig_R
\]
everywhere on $T_x$. In particular,
\[
V^{\rho_k}(\tau_x) \leq \frac{3}{\lambda},\quad \tau_x\in\supp f_R.
\]
Since $V^{\rho_k}\rightarrow V^{\rho_R}$ pointwise on $T_x$, we have
\[
V^{\rho_R}(\tau_x) \leq \frac{3}{\lambda},\quad \tau_x\in\supp f_R.
\]

We have moved all pieces of our problem, constructing $\varphi_{\alpha_y}$, to the dyadic tree $T_x$, and its solution is given by the following lemma, the proof of which will be given in the next subsection.
\begin{lemma}[One-dimensional statement]\label{l:l5.2}
Let $F$ be an open set on the boundary $\partial T$ of the dyadic tree $T$ and let $\rho$ be its equilibrium measure. Assume that a function $f:T\rightarrow\mathbb{R}_+$ satisfies 
\begin{equation}\label{e:l5.2c}
V^{\rho}(\alpha) \leq \delta,\quad \alpha \in \supp f
\end{equation}
with some $\delta\leq \frac13$. Then there exists a non-negative measure $\sigma$ such that
\begin{subequations}\label{e:l5.2s}
\begin{eqnarray}
\label{e:l5.2s.1}& V^{\sigma}(\omega) \geq (If)(\omega),\quad \omega\in F;\\
\label{e:l5.2s.2}& \mathcal{E}[\sigma] \lesssim \delta\|f\|^2_{\ell^2(T)}.
\end{eqnarray}
\end{subequations}
\end{lemma}
Suppose for the moment that Lemma \ref{l:l5.2} holds. We apply it with $T=T_x$, $F=F_R$, $\delta = \frac{3}{\lambda}$ and $f = f_R$ to obtain a measure $\sigma = \sigma_R$ supported on $F_R\subset\partial T_x$ that satisfies \eqref{e:l5.2s}. Now we define
\[
\varphi_{\alpha_y}(\alpha_x,\alpha_y) := (I^*\sigma_R)(\alpha_x),\quad \alpha_x\in T_x,
\]
and we set $\varphi_{\alpha_y}\equiv 0$ outside of $T_x\times\{\alpha_y\}$. We see that \eqref{e:l5.2s} implies \eqref{e:41.a} and \eqref{e:41.b}. Finally we let
\[
\varphi = \frac32\sum_{\alpha_y\in T_y}\varphi_{\alpha_y}.
\]
We are left to show that $\varphi$ is the desired function, that is, it satisfies \eqref{e:l5.1s}. The inequality \eqref{e:l5.1s.2} follows immediately from \eqref{e:41.b}
\begin{equation}\notag
\begin{split}
&\|\varphi\|^2_{\ell^2(T^2)} = \frac94\sum_{\alpha_y\in T_y}\|\varphi_{\alpha_y}\|^2_{\ell^2(T^2)} = \frac94\sum_{\alpha_y\in T_y}\|(I^*\sigma_R)\|^2_{\ell^2(T_x)}\lesssim
\sum_{\alpha_y\in T_y}\frac{1}{\lambda}\|f_R\|^2_{\ell^2(T_x)} =\\
& \frac{1}{\lambda}\sum_{\alpha = (\alpha_x,\alpha_y):\alpha\in E^1}f_R^2(\alpha_x) = \frac{1}{\lambda}\sum_{\alpha\in E^1}(\mathbb{I}^*\mu)^2(\alpha).
\end{split}
\end{equation}
To prove \eqref{e:l5.1s.1} we use a stopping time argument. Fix a point $\omega\in E_{\lambda}$. We define $\alpha_y(\omega)$ to be the first (with respect to the natural order on $T_y$) point such that the (restricted) potential of $\mu$ on the fiber $(\omega_x,\alpha_y(\omega))$ exceeds $\frac{\lambda}{3}$. In other words, 
\[
\mathbb{V}_1^{\mu}(\omega_x,\alpha_y) > \frac{\lambda}{3},\quad\alpha_y\leq\alpha_y(\omega),
\]
and
\[
\mathbb{V}_1^{\mu}(\omega_x,\alpha_y) \leq \frac{\lambda}{3},\quad\alpha_y>\alpha_y(\omega)
\]
(if $\mathbb{V}_1^{\mu}(\omega_x,o_y)\geq\frac{\lambda}{3}$, we set $\alpha_y(\omega) = o_y$, where $o_y$ is the root of $T_y$). Clearly $\omega\in A_R$ with $R = T_x\times\{\alpha_y\}$ for $\alpha_y\leq\alpha_y(\omega)$ -- remember that $\omega\in A_R$, if $\mathbb{V}_1^{\mu}(\omega_x,\alpha_y)>\frac{\lambda}{3}$. Therefore
\begin{equation}\notag
\begin{split}
&V_1^{\mu}(\omega) = \sum_{\alpha\geq\omega:\;\alpha\in E^1}(\mathbb{I}^*\mu)(\alpha) = \sum_{\alpha_x\geq\omega_x}\sum_{\alpha_y\geq\omega_y}(\mathbb{I}_1^*\mu)(\alpha_x,\alpha_y) = \\
&\sum_{\alpha_x\geq\omega_x}\left(\sum_{\alpha_y(\omega)\geq\alpha_y\geq\omega_y}(\mathbb{I}_1^*\mu)(\alpha_x,\alpha_y) +\sum_{\alpha_y>\alpha_y(\omega)}(\mathbb{I}_1^*\mu)(\alpha_x,\alpha_y)\right),
\end{split}
\end{equation}
where $\mathbb{I}^*_1\mu := \mathbb{I}^*\mu\cdot\chi_{E^1}$.
By the definition of $\alpha_y(\omega)$,
\[
\frac{\lambda}{3} > \sum_{\alpha_x\geq\omega_x}\sum_{\alpha_y>\alpha_y(\omega)}(\mathbb{I}_1^*\mu)(\alpha_x,\alpha_y),
\]
and therefore
\[
\lambda \leq \mathbb{V}_1^{\mu}(\omega_x,\omega_y) \leq \sum_{\alpha_x\geq\omega_x}\sum_{\alpha_y(\omega)\geq\alpha_y\geq\omega_y}(\mathbb{I}_1^*\mu)(\alpha_x,\alpha_y)+\frac{\lambda}{3}.
\]
By \eqref{e:41.a}
\begin{equation}\notag
\sum_{\alpha_x\geq\omega_x}\varphi_{\alpha_y}(\alpha_x,\alpha_y) \geq \sum_{\alpha_x\geq\omega_x}(\mathbb{I}_1^*\mu)(\alpha_x,\alpha_y),
\end{equation}
for $\alpha_y\leq\alpha_{y}(\omega)$. Therefore
\begin{equation}\notag
\begin{split}
&\mathbb{I}\varphi(\omega) = \frac32\sum_{\alpha_y\in T_y}\mathbb{I}\varphi_{\alpha_y}(\omega)= \frac32\sum_{\alpha_y\geq \omega_y}\sum_{\alpha_x\geq\omega_x}\varphi_{\alpha_y}(\alpha_x,\alpha_y)\geq\\
&\frac32\sum_{\alpha_y(\omega)\geq\alpha_y\geq \omega_y}\sum_{\alpha_x\geq\omega_x}\varphi_{\alpha_y}(\alpha_x,\alpha_y) \geq \frac32\sum_{\alpha_y(\omega)\geq\alpha_y\geq\omega_y}\sum_{\alpha_x\geq\omega_x}(\mathbb{I}_1^*\mu)(\alpha_x,\alpha_y)\geq\\
&\frac{3}{2}\frac23\lambda = \lambda,
\end{split}
\end{equation}
proving \eqref{e:l5.1s.1}.
\subsection{Proof of Lemma \ref{l:l5.1}. One-dimensional argument: Lemma \ref{l:l5.2}}\label{ss:5.iv}
As mentioned earlier, the condition \eqref{e:l5.2c} can be interpreted as a statement about the distance between $F$ and $\supp f$, in the sense that these two sets are far from each other. More precisely, if we want to find a function $\varphi$ such that $I\varphi \geq If$ on $F$, there is a more effective, in terms of energy, solution than simply letting $\varphi = f$. A natural approach is to modify the equilibrium measure $\rho$ of $F$, since $\varphi = I^*\rho$ provides the best way of acquiring unit potential on $F$.

The argument below goes as follows: first we split the set $F$ into several parts in such a way that $If$ is constant on each part. Then we modify the equilibrium measure $\rho$ on each part according to the value of $If$ there. Finally we show that the resulting measure satisfies \eqref{e:l5.2s}.\\

\noindent \textit{Partition of $F$}

We start with observing that $V^{\rho}(o) \leq \delta$. Indeed, since $V^{\rho}$ is monotone on $T$ (with respect to the natural order), we see that for any $\omega\in\supp f$
\[
\delta \geq V^{\rho}(\omega) \geq V^{\rho}(o).
\]
This allows us to define the $\delta$-level sets of $V^{\rho}$,
\[
F_{\delta} := \{\beta\in T: V^{\rho}(\beta) >\delta \textup{\,and}\, V^{\rho}(\alpha) \leq \delta,\, \textup{if}\,\alpha>\beta\}.
\]
$F_{\delta}$ is essentially a stopping-time set for $V^{\rho}$. Define $\tilde{F}:= \{\omega\in \textup{cl}F:\; V^{\rho} = 1\}$. Since $F$ is open, we have $V^{\rho} \equiv 1$ on $F$ (see Lemma \ref{l:A.1.1}), hence $F\subset \tilde{F}$. Therefore $\tilde{F}\subset \mathcal{S}(F_{\delta})$. Also we note that, if $\beta\in F_{\delta}$, then
\begin{equation}\notag
\begin{split}
&V^{\rho}(\beta) = V^{\rho}(p(\beta)) + (I^*\rho)(\beta) \leq V^{\rho}(p(\beta)) + (I^*\rho)(o) =\\
&V^{\rho}(p(\beta)) + V^{\rho}(o)\leq \delta + \delta = 2\delta,
\end{split}
\end{equation}
where $p(\beta)$ denotes the immediate parent of $\beta$ in $T$. In particular we see that $\supp f$ is outside $\mathcal{S}(F_{\delta})$,
\[
\supp f\cap \mathcal{S}(F_{\delta}) = \emptyset,
\]
where $\mathcal{S}(F_{\delta}) = \bigcup_{\beta\in F_{\delta}}\mathcal{S}(\beta)$. Also $\mathcal{S}(\beta_1)\cap\mathcal{S}(\beta_2) = \emptyset$ for any pair of (different) points $\beta_1,\beta_2\in F_{\delta}$, so that the sets $\{\partial\mathcal{S}(\beta)\}_{\beta\in F_{\delta}}$ form a disjoint covering of the set $\tilde{F}$.  
Now we define the partition of $\rho$ as follows
\[
\rho_{\beta} = \rho\vert_{\partial\mathcal{S}(\beta)},\quad \beta \in F_{\delta}.
\]
Recall that $\supp\rho = \textup{cl}F\subset\partial T$ and that $\int_{\partial T}V^{\rho}\,d\rho = \capp F = |\rho|$, therefore $\rho(\textup{cl}F\setminus\tilde{F}) = 0$. It follows immediately that
\[
\rho = \sum_{\beta\in F_{\delta}}\rho_{\beta}.
\] 
We are ready to define the measure $\sigma$. Given $\beta\in F_{\delta}$ we set
\[
\tilde{\sigma}_{\beta} := (If)(\beta)\rho_{\beta},
\]
and 
\[
\tilde{\sigma} = \sum_{\beta\in F_{\delta}}\tilde{\sigma}_{\beta}.
\]
Finally we let
\[
\sigma = (1-2\delta)^{-1}\tilde{\sigma}.
\]
We are left to show that $\sigma$ satisfies \eqref{e:l5.2s}.\\

\noindent \textit{Inequality \eqref{e:l5.2s.1}}\\
Fix any $\omega\in F$. There exists exactly one point $\beta_{\omega}\in F_{\delta}$ such that $\beta_{\omega} > \omega$. By definition of $\rho_{\beta}$ we have
\[
(I^*\rho_{\beta{\omega}})(\alpha) = (I^*\rho)(\alpha)
\]
for any $\alpha \leq \beta_{\omega}$. It follows that
\begin{equation}\label{e:53}
\begin{split}
&V^{\rho_{\beta_{\omega}}}(\omega) = \sum_{\alpha\geq\omega}(I^*\rho_{\beta_{\omega}})(\alpha) \geq \sum_{\beta_{\omega}>\alpha\geq\omega}(I^*\rho_{\beta_{\omega}})(\alpha) =\\
&\sum_{\beta_{\omega}>\alpha\geq\omega}(I^*\rho)(\alpha) = V^{\rho}(\omega)-V^{\rho}(\beta_{\omega}) \geq 1-2\delta,
\end{split}
\end{equation}
since $V^{\rho} \equiv 1$ on $F$. Hence
\begin{equation}\notag
\begin{split}
&V^{\sigma}(\omega)=(1-2\delta)^{-1}V^{\tilde{\sigma}}(\omega) \geq (1-2\delta)^{-1}V^{\tilde{\sigma}_{\beta_{\omega}}}(\omega) = \\
&(1-2\delta)^{-1}(If)(\beta_{\omega})V^{\rho_{\beta_{\omega}}}(\omega)\geq (If)(\beta_{\omega}) = (If)(\omega),
\end{split}
\end{equation}
since $f$ has no mass on the set $\mathcal{S}(\beta_{\omega})$.

\noindent \textit{Inequality \eqref{e:l5.2s.2}}

To prove this inequality we do a further partition of $\tilde{\sigma}$ and $\rho$ with respect to the distribution of $If$. First we raise $\tilde{\sigma}$ and $\rho$ to the set $F_{\delta}$,
\begin{equation}\notag
\begin{split}
&\tilde{\sigma}_{\delta}(\beta) := \chi_{F_{\delta}}(I^*\tilde{\sigma})(\beta),\\
&\rho_{\delta}(\beta) := \chi_{F_{\delta}}(I^*\rho)(\beta),\quad \beta\in T.
\end{split}
\end{equation}
Clearly $\tilde{\sigma}_{\delta},\rho_{\delta}$ are supported on $F_{\delta}$, and $(I^*\tilde{\sigma_{\delta}})(\alpha) = (I^*\tilde{\sigma})(\alpha)$ for any $\alpha \geq F_{\delta}$ (by this we mean that $\alpha\geq \theta$ for all $\theta\in F_{\delta}$). Also
\begin{equation}\label{e:57}
\tilde{\sigma}_{\delta}(\beta) = (If)(\beta)\rho_{\delta}(\beta),\quad \beta\in F_{\delta},
\end{equation}
by the definition of $\tilde{\sigma}$.

Next we compute the energy of $\tilde{\sigma}$
\begin{equation}\notag
\begin{split}
&\mathcal{E}[\tilde{\sigma}] = \int_{\overline T} V^{\tilde{\sigma}}\,d\tilde{\sigma} = \int_{\overline T}\sum_{\alpha\geq\omega}(I^*\tilde{\sigma})(\alpha)\,d\tilde{\sigma}(\omega)=\\
&\sum_{\beta\in F_{\delta}}\int_{\mathcal{S}(\beta)}\sum_{\alpha\geq\beta}(I^*\tilde{\sigma})(\alpha)\,d\tilde{\sigma}(\omega) + \sum_{\beta\in F_{\delta}}\int_{\mathcal{S}(\beta)}\sum_{\beta> \alpha\geq\omega}(I^*\tilde{\sigma})(\alpha)\,d\tilde{\sigma}(\omega)\leq\\
&\sum_{\beta\in F_{\delta}}\sum_{\alpha\geq\beta}(I^*\tilde{\sigma}_{\delta})(\alpha)\tilde{\sigma_{\delta}}(\beta) + \sum_{\beta\in F_{\delta}}\int_{\mathcal{S}(\beta)}\sum_{\beta\geq \alpha\geq\omega}(I^*\tilde{\sigma})(\alpha)\,d\tilde{\sigma}(\omega) :=\\
& (I) + (II).
\end{split}
\end{equation}
For the first term we have
\[
(I) = \int_{F_{\delta}}V^{\tilde{\sigma}_{\delta}}\,d\tilde{\sigma}_{\delta} = \mathcal{E}[\tilde{\sigma}_{\delta}].
\]
We expand the second term obtaining
\begin{equation}\label{e:55}
\begin{split}
&(II) = \sum_{\beta\in F_{\delta}}\int_{\mathcal{S}(\beta)}\sum_{\beta\geq \alpha\geq\omega}(I^*\tilde{\sigma})(\alpha)\,d\tilde{\sigma}(\omega) = \sum_{\beta\in F_{\delta}}\int_{\mathcal{S}(\beta)}\sum_{\beta\geq \alpha\geq\omega}(I^*\tilde{\sigma}_{\beta})(\alpha)\,d\tilde{\sigma}_{\beta}(\omega)=\\
&\sum_{\beta\in F_{\delta}}(If)(\beta)\int_{\mathcal{S}(\beta)}\sum_{\beta\geq \alpha\geq\omega}(I^*\rho_{\beta})(\alpha)\,d\tilde{\sigma}_{\beta}(\omega)\leq \sum_{\beta\in F_{\delta}}(If)(\beta)(I^*\tilde{\sigma})(\beta),
\end{split}
\end{equation}
since 
\[
\sum_{\beta\geq \alpha\geq\omega}(I^*\rho_{\beta})(\alpha) = \sum_{\beta\geq \alpha\geq\omega}(I^*\rho)(\alpha)\leq \sum_{\alpha\geq\omega}(I^*\rho)(\alpha)=V^{\rho}(\omega) \leq1
\]
for $\omega\in \textup{cl} F$. We see that
\[
(II) \leq  \int_{\overline T} If\,d\tilde{\sigma}_{\delta} = \sum_{\alpha\in T}f(\alpha)(I^*\tilde{\sigma}_{\delta})(\alpha) =: \mathcal{E}[f,\tilde\sigma_{\delta}].
\]
Hence
\[
\mathcal{E}[\tilde{\sigma}] \leq \mathcal{E}[\tilde{\sigma}_{\delta}] + \mathcal{E}[f,\tilde{\sigma}_{\delta}].
\]
We observe that in order to prove \eqref{e:l5.2s.2} it is enough to show that
\begin{equation}\label{e:61}
\mathcal{E}[\tilde{\sigma}_{\delta}] \leq C\delta\mathcal{E}[f,\tilde{\sigma}_{\delta}]
\end{equation}
for some absolute constant $C>0$. Indeed, by positivity of the energy integral we have
\begin{equation}\notag
\begin{split}
&0\leq\sum_{\alpha\in T}\left((I^*\tilde{\sigma}_{\delta})(\alpha) - C\delta f(\alpha)\right)^2 = \mathcal{E}[\tilde{\sigma}_{\delta}]-2C\delta\mathcal{E}[f,\tilde{\sigma}_{\delta}] + (C\delta)^2\|f\|^2_{\ell^2(T)}=\\
&(\mathcal{E}[\tilde{\sigma}_{\delta}] - C\delta\mathcal{E}[f,\tilde{\sigma}_{\delta}]) + C\delta(C\delta\|f\|^2_{\ell^2(T)}-\mathcal{E}[f,\tilde{\sigma}_{\delta}]),
\end{split}
\end{equation}
so if the first term is negative, the second one must be positive. Hence $\mathcal{E}[f,\tilde{\sigma}_{\delta}] \leq C\delta\|f\|^2_{\ell^2(T)}$, and by \eqref{e:61},
\[
\mathcal{E}[\tilde{\sigma}] \leq \mathcal{E}[\tilde{\sigma}_{\delta}] + \mathcal{E}[f,\tilde{\sigma}_{\delta}] \leq (C\delta)^2\|f\|^2_{\ell^2(T)} + C\delta\|f\|^2_{\ell^2(T)},
\]
which is \eqref{e:l5.2s.2}.

Now for any $k\in\mathbb{Z}$ we define
\begin{equation}\notag
F_{\delta,k} := \{\beta\in F_{\delta}: 2^{k}\leq (If)(\beta) < 2^{k+1}\},
\end{equation}
and we set $\tilde{\sigma}_{\delta,k} = \tilde{\sigma}_{\delta}\vert_{F_{\delta,k}}$, $\rho_{\delta,k} = \rho_{\delta}\vert_{F_{\delta,k}}$.\\
Clearly $\tilde{\sigma}_{\delta} = \sum_{k}\tilde{\sigma}_{\delta,k}$, and
\[
\mathcal{E}[\tilde{\sigma}_{\delta}] = \int_{\overline T}V^{\tilde{\sigma}_{\delta}}\,d\tilde{\sigma}_{\delta} = \sum_{k\in\mathbb{Z}}\sum_{j\in\mathbb{Z}}\int_{\overline T}V^{\tilde{\sigma}_{\delta,j}}\,d\tilde{\sigma}_{\delta,k}\leq 2\sum_{k\in\mathbb{Z}}\sum_{j\leq k}\int_{\overline T}V^{\tilde{\sigma}_{\delta,j}}\,d\tilde{\sigma}_{\delta,k}.
\]
For $j\leq k$, by \eqref{e:57}, we have that $V^{\tilde{\sigma}_{\delta,j}} \leq 2^{j+1}V^{\rho_{\delta,j}}$ on $F_{\delta,k}$. Thus
\begin{equation}\notag
\begin{split}
&\sum_{j\leq k}\int_{\overline T}V^{\tilde{\sigma}_{\delta,j}}\,d\tilde{\sigma}_{\delta,k} \leq \sum_{j\leq k}2^{j+1}\int_{\overline T}V^{\rho_{\delta,j}}\,d\tilde{\sigma}_{\delta,k} \leq 2^{k+1}\int_{\overline T} \sum_{j\leq k}V^{\rho_{\delta,j}}\,d\tilde{\sigma}_{\delta,k}=\\
& 2^{k+1}\int_{\overline T} V^{\sum_{j\leq k}\rho_{\delta,j}}\,d\tilde{\sigma}_{\delta,k}\leq  2^{k+1}\int_{\overline T} V^{\rho_{\delta}}\,d\tilde{\sigma}_{\delta,k} \leq 2^{k+2}\delta|\tilde{\sigma}_{\delta,k}|,
\end{split}
\end{equation}
since $V^{\rho_{\delta}} \leq 2\delta$ on $F_{\delta}$. Summing this estimate over $k$ we obtain
\begin{equation}\notag
\begin{split}
&\sum_{k\in\mathbb{Z}}\sum_{j\leq k}\int_{\overline T}V^{\tilde{\sigma}_{\delta,j}}\,d\tilde{\sigma}_{\delta,k} \leq \sum_{k\in\mathbb{Z}}2^{k+2}\delta|\tilde{\sigma}_{\delta,k}| = \\
&\sum_{k\in\mathbb{Z}}\sum_{\beta\in F_{\delta,k}}2^{k+2}\delta\tilde{\sigma}_{\delta,k}(\beta) \leq 4\delta\sum_{k\in\mathbb{Z}}\sum_{\beta\in F_{\delta,k}}(If)(\beta)\tilde{\sigma}_{\delta,k}(\beta)=\\
&4\delta\sum_{\beta\in F_{\delta}}(If)(\beta)\tilde{\sigma}_{\delta}(\beta) = 4\delta\mathcal{E}[f,\tilde{\sigma}_{\delta}],
\end{split}
\end{equation}
so we get \eqref{e:61}, and therefore \eqref{e:l5.2s.2}.

\section{Concluding remarks}\label{conclusions}

We have characterized the Carleson measures for the Dirichlet space using, as in Stegenga's \cite{stegenga1980}, a Strong Capacitary Inequality.
In the one-parameter case, other characterizations can be given. In \cite{ars2002} and \cite{kermansawyer1988}, the Carleson measures for $\DDD(\DD)$ are defined in 
terms of two, seemingly different, \it one-box testing \rm conditions, in which the advantage is that they have to be verified for single Carleson boxes,
and not unions thereof. The disadvantage is that the measure $\mu$, unlike in the capacitary condition, 
appears on both sides of the testing inequality. 

The bi-parameter non-linear case, $1<p<\infty$, could also be considered; the space under scrutiny would be the tensor product of two copies of an analytic Besov space.
The one dimensional case was considered for example in \cite{wu99} and \cite{ars2002}. Here, we think that the needed tool is a bi-parameter version of Wolff's inequality 
\cite{HW83}, which could be considered as one half of the Muckenhoupt-Wheeden inequality. With that at hand, one could extend a sizeable portion of the
potential theory we have developed here to the non-linear case.

The probabilistic theory underlying bi-parameter linear Potential Theory is that of two-parameter martingales \cite{Cairoli1971, Gundy1980}.
It would be interesting to make this relationship explicit, and to find a way to pass results from one theory to the other.

Much of the Potential Theory we have developed on bi-trees can be applied to yield a Potential Theory on product spaces much more general than $\partial\DD\times\partial\DD$, following, for example, the route taken in \cite{arsw2014}.

The Dirichlet space on the bidisc does not come with a Complete Nevanlinna--Pick kernel. In fact, no tensor product Hilbert space does \cite{Young18}. 
If the kernel had the Complete Nevanlinna-Pick property, the characterization of Carleson measures for $\DDD(\DD^2)$ would as consequence yield the characterization of its 
{\it universal interpolating sequences}, by a recent result of 
Aleman, Hartz, McCarthy, and Richter \cite{AHMR17}. We think this is an interesting open problem, for which we have no guess. See \cite{AM02, Seip04} for a deep and broad discussion of interpolating sequences for Hilbert function spaces.

\section{Appendix}\label{S:A}
In this section we collect several results that were used or mentioned in the main text. First, in Section \ref{S:A.0} we provide the proofs of the more technical results from Section \ref{S:2} regarding the discretization procedure. Then we present some basic properties of bi-logarithmic potentials and equilibrium measures, see Lemma \ref{l:A.1.1}. In Section \ref{S:A.2} we give counterexamples to the maximum and domination principles, in Propositions \ref{p:A.2.1} and \ref{p:A.2.2}, respectively. Finally, in Theorem \ref{t:th.A.3} we show that given a measure on $\overline{T}^2$, we can construct a measure supported on the distinguished boundary, equivalent to the original measure in the sense of potentials. From this we deduce that the capacity of a set is equivalent to the capacity of its boundary projection, see Corollary \ref{c:A.3}.
\subsection{}\label{S:A.0}
We start with  providing some results justifying the discretization of the unit disc (bidisc) via the graphs $\mathfrak{G}$ and $T$. The graph $\mathfrak{G}$ serves as an intermediate point in the discretization scheme between the unit disc and the dyadic tree -- see Section \ref{SS:2.4} for precise definitions. While it is more complicated and rather inconvenient to work with when compared to the tree $T$, its geometry is better suited for representing the unit disc, which is why we use it to justify passing from $\mathbb{D}$ to $T$ (and therefore from $\mathbb{D}^2$ to $T^2$).

First we show that $\mathfrak{G}$ provides a model for the hyperbolic metric on $\overline{\mathbb{D}}$.
\begin{lemma}\label{l:A.0.1}
Given two points $z,w\in\overline{\mathbb{D}}^2$ one has
\begin{equation}\label{e:A.51}
\left|10+\log\frac{1}{1-\bar{z}_1w_1}\right|\left|10+\log\frac{1}{1-\bar{z}_2w_2}\right|\approx d_{\mathfrak{G}^2}(\alpha(z)\wedge\beta(w)) = d_{\mathfrak{G}}(\alpha_x(z)\wedge\beta_x(w))d_{\mathfrak{G}}(\alpha_y(z)\wedge\beta_y(w)),
\end{equation}
where $z = (z_1,z_2)$, $w= (w_1,w_2)$, and $\alpha,\beta\in \overline{T}$ are any of the preimages $\alpha(z) = (\alpha_x(z),\alpha_y(z)) \in \Lambda^{-1}(z),\; \beta(w) = (\beta_x(w),\beta_y(w)) \in \Lambda^{-1}(w)$. We recall that the natural map $\Lambda$ from $\overline{T}^2$ to $\overline{\mathbb{D}}^2$ was defined in Section \ref{SS:2.5}.
\end{lemma}
\begin{proof}
Clearly it is enough to show \eqref{e:A.51} separately for each coordinate; we show that
\begin{equation}\label{e:A.52}
\left|10 + \log\frac{1}{1-\bar{z_1}w_1}\right| \approx d_{\mathfrak{G}}(\alpha_x\wedge\beta_x).
\end{equation}
Note that
\[
\left|10 + \log\frac{1}{1-\bar{z}_1w_1}\right| \approx 10 + \log\frac{1}{|1-\bar{z}_1w_1|},
\] 
since $1+\log\frac{1}{|1-\bar{z}_1w_1|} >0$ for any pair of points $z_1,w_1\in\overline{\mathbb{D}}$.

We start by assuming  $z_1,w_1 \in\mathbb{D}$. Recall that there exist uniquely defined $\alpha_x,\beta_x\in T$ such that
$z_1 \in Q_{\alpha_x},\; w_1\in Q_{\beta_x}$.  Let $J$ be the smallest interval, not necessarily dyadic, containing both $J_{\alpha_x}$ and $J_{\beta_x}$. We claim that $d_{\mathfrak{G}}(\alpha_x\wedge\beta_x) \approx \log |J|^{-1}$. Indeed, in order for $\gamma$ to be a common point of the sets $\mathcal{P}_{\mathfrak{G}}(\alpha_x)$ and $\mathcal{P}_{\mathfrak{G}}(\beta_x)$, the dyadic interval $J_{\gamma}$ has to be large, $|J_{\gamma}| \geq \frac12\max(|J_{\alpha_x}|, |J_{\beta_x}|)$, and $3J_{\gamma}$ must have non-empty intersection with both $J_{\alpha_x}$ and $J_{\beta_x}$. The number of such intervals is approximately $\log\frac{1}{|J|}$. On the other hand, an elementary computation yields that
\[
|1-\bar{z}_1w_1|\approx \max(1-|z_1|,1-|w_1|, |\arg(\bar{z}_1w_1)|) \approx |J|.
\]

Now let $z_1,w_1\in \partial\mathbb{D}$. If these two points coincide, then $d_{\mathfrak{G}}=\infty$ regardless of the choice of pre-images $\alpha_x \in \Lambda_0^{-1}(\{z_1\}),\; \beta_x \in \Lambda_0^{-1}(\{w_1\})$. Otherwise we let $J$ to be the smallest interval containing $z_1$ and $w_1$, and repeat the above argument. 

The cases $z_1\in \mathbb{D},\; w_1\in \partial\mathbb{D}$ and $w_1\in \partial\mathbb{D},\; z_1\in \mathbb{D}$ are dealt with similarly.
\end{proof}

Next we investigate the properties of the map $\Lambda$, and the induced pull-backs and push-forwards of measures, as introduced in Section \ref{SS:2.5}.
\begin{lemma}\label{l:A.8}
  The Lipschitz map $\Lambda : \overline{T}^2 \rightarrow \overline{\mathbb{D}}^2$ induces maps $\Lambda_{\ast} : \tmop{Meas}^+ ((\partial
  T)^2) \rightarrow \tmop{Meas}^+ ((\partial\mathbb{D})^2))$ and $\Lambda^{\ast} : \tmop{Meas}^+ (\overline{\mathbb{D}}^2) \rightarrow
  \tmop{Meas}^+ (\overline{T}^2)$, $\tmop{Meas}^+$ denoting the space of non-negative Borel measures on the respective set, with the following properties:
  \begin{itemize}
    \item $\Lambda_{\ast} \Lambda^{\ast} \mu = \mu $, if $\mu$ is supported on $(\partial\mathbb{D})^2$.
    
    \item If $\mu\geq0$ is a Borel measure on $(\partial\mathbb{D})^2$ with finite $(\frac12,\frac12)$-energy,
    then $\Lambda^{\ast} \mu (E) = \mu (\Lambda (E))$ for all measurable sets
    $E$ in $(\partial T)^2$. In particular, $\mu(\Delta) =0$, where $\Delta$ is the dyadic grid on $(\partial\mathbb{D})^2$, the set of points with at least one dyadic coordinate.
    \item If $\nu\geq0$ is a Borel measure on $(\partial T)^2$ with finite energy, then $\nu (\partial T_x\times \{\omega_y\} \bigcup \{\omega_x\}\times\partial T_y) = 0$ for any $\omega \in (\partial T)^2$.
    
    \item For such a measure $\nu$, it holds that $\Lambda^{\ast} \Lambda_{\ast} \nu = \nu$.
  \end{itemize}
 
\end{lemma}
\begin{proof}
The first point is obvious.

\textit{Proof of the second point.} It is enough to show that $\mu(\Delta) = 0$, since $\Delta$ precisely consists of the points where $\Lambda^{-1}$ is not uniquely defined. In turn, one only has to prove that, say, $\mu (\{ 1 \} \times \partial \mathbb{D})= 0$, since $\Delta$ is a countable union of such sets.

Let us recall the dual definition of capacity: for any compact set $E\subset \overline{T}^2$ one has
\begin{equation*}
\capp^{\frac12} E = \sup\left\{\frac{\mu(E)}{\mathcal{E}^{\frac12}[\mu]}:\; \supp \mu\subset E\right\},
\end{equation*}
and the maximizer is exactly the equilibrium measure $\mu_E$. From here, is not difficult to see that the proof of the second point of the statement will follow if we show that $\{ 1 \} \times \partial \mathbb{D}$ is a polar set, meaning that
\[
\capp_{(\frac12,\frac12)}(\{ 1 \} \times \partial \mathbb{D}) = 0.
\]
This is almost a direct corollary of the one-dimensional fact that the Bessel $\frac12$-capacity of a singleton on the unit circle is zero. To elaborate, let 
\[
h_K(e^{i\theta_1},e^{i\theta_2}) := \frac{1}{K}\sum_{j=1}^K 2^{\frac{j}{2}}\chi_{[-2^{-j},2^{-j}]}(\theta_1).
\]
Clearly $\int_{(\partial\mathbb{D})^2}h^2_{K}(z)\,dm(z) \lesssim \frac{1}{K}$, and an elementary computation shows that
\[
(b_{(\frac12,\frac12)}h_K)(e^{i0},e^{i\theta_2}) \approx \frac{1}{K}\sum_{j=1}^K2^{\frac{j}{2}}\int_{0}^{2^{-j}}\frac{d\theta_1}{\theta_1^{\frac12}}\approx 1.
\]
Hence $Ch_K$ is an admissible function for some large $C>1$. Letting $K$ to infinity we immediately obtain the desired result.

\textit{Proof of the third point}. Assume that $\nu(\{\omega_x\}\times\partial T_y) =\varepsilon >0$ for some $\omega_x\in \partial T_x$. Then we immediately have $(\mathbb{I}^*\nu)(\alpha_x,o) \geq \varepsilon$ for any $\alpha_x>\omega_x$, and
\[
\mathcal{E}[\nu] = \sum_{\alpha\in T^2}(\mathbb{I}^*\nu)^2(\alpha) \geq \sum_{\alpha_x>\omega_x}(\mathbb{I}^*\nu)^2(\alpha_x,o) = \infty.
\]
The same argument shows that $\partial T_x\times\{\omega_y\}$ has measure zero.

\textit{Proof of the fourth point.}
\[
  \Lambda^{\ast} \Lambda_{\ast} \nu (E)  =  \Lambda_{\ast} \nu (\Lambda (E))
  \nocomma = \nu (\Lambda^{- 1} (\Lambda (E)))   =  \nu (E),
\]
since $\Lambda$ fails to be injective only on a set of vanishing $\nu$-measure.
\end{proof}

Next we show that $\mathfrak{G}$ and $T$ are similar in capacitary sense. In the one-parameter setting, much stronger results are available, see for example Lemma 2.14 in \cite{Bishop94}.
\begin{lemma}\label{l:A.50.5}
Given a finite family of points $\{\alpha^j\}_{j=1}^n\subset T^2$, one has
\begin{equation}\label{e:A.50.5}
\capp \left(\bigcup_{j=1}^n\mathcal{S}(\alpha^j)\right) \approx \capp\left(\bigcup_{j=1}^n\partial\mathcal{S}(\alpha^j)\right) \approx \capp\left(\bigcup_{j=1}^n\mathcal{S}_{\mathfrak{G}^2}(\alpha^j)\right) \approx \capp\left(\bigcup_{j=1}^n\partial\mathcal{S}_{\mathfrak{G}^2}(\alpha^j)\right) .
\end{equation}
In particular, 
\[
\capp \left(\bigcup_{j=1}^n\mathcal{S}(\alpha^j)\right) \approx \capp \left(\bigcup_{j=1}^n\mathcal{S}(p(\alpha^j))\right),
\]
where $p(\alpha^j) = (p(\alpha^j_x), p(\alpha^j_y))$ is the 'grandparent' of $\alpha^j$ in $T^2$, see the proof of Theorem \ref{t:equivthingies}.
\end{lemma}
\begin{proof}
The first and last equivalences of \eqref{e:A.50.5} come from the fact that the capacities of a set and its boundary projection are comparable, see Corollary \ref{c:A.3}. Since $\mathcal{S}_{\mathfrak{G}^2}(\alpha)\supset \mathcal{S}(\alpha)$ for any $\alpha\in T^2$, we have 
\[
\capp \left(\bigcup_{j=1}^n\mathcal{S}(\alpha^j)\right) \leq \capp \left(\bigcup_{j=1}^n\mathcal{S}_{\mathfrak{G}^2}(\alpha^j)\right).
\]
To show the reverse inequality
\[
\capp E:= \capp \left(\bigcup_{j=1}^n\mathcal{S}(\alpha^j)\right) \gtrsim \capp \left(\bigcup_{j=1}^n\mathcal{S}_{\mathfrak{G}^2}(\alpha^j)\right) =: \capp F,
\]
we prove that the energies of $\mu_E$ and $\mu_F$ are comparable. We start by showing that the mixed energy of $\mu_E$ and $\mu_F$ dominates $\mathcal{E}[\mu_F]$, using an argument similar to the one in Section \ref{SS:2.4}. We have
\[
\mathcal{E}[\mu_E,\mu_F] = \sum_{\alpha\in T^2}(\mathbb{I}^*\mu_E)(\alpha)(\mathbb{I^*}\mu_F)(\alpha) = \sum_{\alpha\in T^2}\mu_E(\mathcal{S}(\alpha))\mu_F(\mathcal{S}(\alpha)).
\]
The successor set formula \eqref{e:331} implies that for any $\alpha\in T^2$ there exists a finite collection $G_{\alpha} = \{\beta^j_{\alpha}\}_{j=1}^N$, $N$ independent of $\alpha$, such that $\mathcal{S}_{\mathfrak{G}^2}(\alpha)\subset \bigcup_{j=1}^N\mathcal{S}_{T^2}(\beta^j_{\alpha})$, and moreover, for any $\beta\in T^2$ there exist at most $N$ points $\alpha$ such that $\beta\in G_{\alpha}$. It follows that
\begin{equation}\notag
\begin{split}
&\sum_{\alpha\in T^2}\mu_E(\mathcal{S}(\alpha))\mu_F(\mathcal{S}(\alpha)) \gtrsim  \sum_{\alpha\in T^2}\mu_E(\mathcal{S}(\alpha))\sum_{\beta\in G_{\alpha}}\mu_F(\mathcal{S}_{T^2}(\beta))\gtrsim\sum_{\alpha\in T^2}\mu_E(\mathcal{S}(\alpha))\mu_F(\mathcal{S}_{\mathfrak{G}^2}(\alpha)) =\\ &\sum_{\alpha\in T^2}(\mathbb{I}^*\mu_E)(\alpha)\int_{\mathcal{S}_{\mathfrak{G}^2}(\alpha)}\,d\mu_F =
\int_{\overline{T}^2}\sum_{\alpha\in\mathcal{P}_{\mathfrak{G}^2}(\beta)}(\mathbb{I}^*\mu_E)(\alpha)\,d\mu_F(\beta).
\end{split}
\end{equation}
Given $\beta\in F$ there must exist at least one $\alpha^j$ such that $\alpha^j \in \mathcal{P}_{\mathfrak{G}}^2(\beta)$. Since $\mu_E$ is the equilibrium measure of $E$, $\alpha^j \in E$, and $\alpha^j\in T^2$ (so that $\capp\{\alpha^j\}>0$) we have 
$1\leq\mathbb{V}^{\mu_E}(\alpha^j) \leq \sum_{\alpha\in\mathcal{P}_{\mathfrak{G}^2}(\alpha^j)}(\mathbb{I}^*\mu_E)(\alpha) \leq \sum_{\alpha\in\mathcal{P}_{\mathfrak{G}^2}(\beta)}(\mathbb{I}^*\mu_E)(\alpha)$. It follows immediately that $\int_{\overline{T}^2}\sum_{\alpha\in\mathcal{P}_{\mathfrak{G}^2}(\beta)}(\mathbb{I}^*\mu_E)(\alpha)\,d\mu_F(\beta) \geq \mu_F(F) = \mathcal{E}[\mu_F]$, therefore
\[
\mathcal{E}[\mu_E,\mu_F]\geq \varepsilon\mathcal{E}[\mu_F]
\]
for some $\varepsilon>0$ that does not depend on $E$ or $F$. By positivity of the energy integral
\[
0\leq \mathcal{E}[\mu_E - \varepsilon\mu_F] = \mathcal{E}[\mu_E] - 2\varepsilon\mathcal{E}[\mu_E,\mu_F] + \varepsilon^2\mathcal{E}[\mu_F] = \left(\mathcal{E}[\mu_E] - \varepsilon\mathcal{E}[\mu_E,\mu_F]\right) + \varepsilon\left(\varepsilon\mathcal{E}[\mu_F] - \mathcal{E}[\mu_E,\mu_F]\right).
\]
We have shown that the first term must be positive, which in turn implies that $\mathcal{E}[\mu_E]\geq \varepsilon^2\mathcal{E}[\mu_F]$. We are done.
\end{proof}

The next result compares the capacities of sets in $\overline{T}^2$ and $(\partial\mathbb{D})^2$. We refer to Sections \ref{SS:2.35} and \ref{SS:2.6} for the relevant definitions. By arguments in Section \ref{S:A.3}, we can always estimate the capacity of a set in $\overline{T}^2$ by the capacity of its boundary projection. Therefore we only consider sets on the distinguished boundaries of the bitree and bidisc. Moreover, it is sufficient to consider finite unions of 'rectangles'. The proof mostly consists of arguments from \cite[Chapter 4]{arsw2014}, adapted to the two-parameter setting.

\begin{lemma}\label{l:A.0.3}
Let $\{\alpha^k\}_{j=1}^N$ be a finite collection of points in $T^2$. One then has
\begin{equation}\label{e:A.71}
\capp\left(\bigcup_{k=1}^N\mathcal{S}(\alpha^k)\right) \approx \capp\left(\bigcup_{k=1}^N\partial\mathcal{S}(\alpha^k)\right) \approx \capp_{(\frac12,\frac12)}\left(\bigcup_{k=1}^N \Lambda(\partial\mathcal{S}(\alpha^k))\right) = \capp_{(\frac12,\frac12)}\left(\bigcup_{k=1}^N J_{\alpha^k}\right),
\end{equation}
where $J_{\alpha} = S_{\alpha}\cap (\partial\mathbb{D})^2$ is the intersection of the Carleson box $S_{\alpha}$ with the torus.
\end{lemma}
\begin{proof}
As before, the first equivalence follows from Corollary \ref{c:A.3}. Set $E := \bigcup_{k=1}^n\partial\mathcal{S}(\alpha^k)$ and $F:= \bigcup_{k=1}^N J_{\alpha^k}$. Clearly both sets are compact in their respective topologies. Let $\nu_E, \mu_F$ be the equilibrium measures for $E$ and $F$, so that $\supp\nu_E\subset E,\; \supp\mu_F\subset F$. To compare the capacities of $E$ and $F$ we need to know how to move  equilibrium measures between $(\partial T)^2$ and $(\partial\mathbb{D})^2$. Our first step in this direction is to show that
\begin{equation}\label{e:A.81}
\begin{split}
&\supp\Lambda^*\mu_F\subset E,\; \mu_F(F) = (\Lambda^*\mu_F)(E);\\
&\supp\Lambda_*\nu_E\subset F,\; \nu_E(E) = (\Lambda_*\nu_E)(F).
\end{split}
\end{equation}
We start with $\mu_F$.  The mass conservation property can be shown separately on each rectangle. Fix any $J_{\alpha_k}$ and denote $\partial \mathcal{S}(\alpha^k)$ by $R_{\alpha^k}$. Recalling the definition of $\Lambda^*$ we see that
\[
(\Lambda^*\mu_F)(R_{\alpha^k}) = \int_{(\partial\mathbb{D})^2}\psi(z)\,d\mu_F,
\]
where $\psi(z):= \frac{\sharp\{\Lambda^{-1}(\{z\})\cap R_{\alpha^k}\}}{\sharp\{\Lambda^{-1}(\{z\})\}},\; z\in (\partial\mathbb{D})^2$. For any $z$ in the (torus) interior of the rectangle $J_{\alpha^k}$ we clearly have $\psi(z) =1$. Unfortunately $\Lambda^{-1}(J_{\alpha^k})$ is slightly larger than $R_{\alpha^k}$, so $\psi(z)$ could be $\frac12$ or $\frac14$, depending on whether $z$ is on the side of $J_{\alpha^k}$, or is one of its corners. However, by Lemma \ref{l:A.8} we see that $\mu_F(\partial J_{\alpha^k}) = 0$, since, clearly, $\mathcal{E}_{(\frac12,\frac12)}[\mu_F] = \mu_F(F) < +\infty$. Here $\partial J_{\alpha^k}$ is the boundary of the rectangle $J_{\alpha^k}$ in the torus $(\partial\mathbb{D})^2$. It follows that $(\Lambda^*\mu_F)(R_{\alpha^k}) = \mu_F(J_{\alpha^k})$, hence $\mu_F(F) = (\Lambda^*\mu_F)(E)$. Arguing as above we obtain
\[
(\Lambda^*\mu_F)((\partial T)^2\setminus E) = \int_{(\partial\mathbb{D})^2}\frac{\sharp\{\Lambda^{-1}(\{z\})\cap ((\partial T)^2\setminus E)\}}{\sharp\{\Lambda^{-1}(\{z\})\}}\,d\mu_F = 0,
\]
since $\mu_F$ has zero mass on the boundary of $F$ in $(\partial\mathbb{D})^2$ and the set $\{\Lambda^{-1}(\{z\})\cap ((\partial T)^2\setminus E)\}$ is empty for any $z$ in the interior of $F$. Therefore $\supp\Lambda^*\mu_F\subset E$, and we have the first part of \eqref{e:A.81}.

The argument for $\nu_E$ is similar. Clearly $\Lambda^{-1}((\partial\mathbb{D})^2\setminus F) \cap E = \emptyset$, hence $(\Lambda_*\nu_E)((\partial\mathbb{D})^2\setminus F)=0$, and $\supp\Lambda_*\nu_E\subset F$. Further, $\Lambda(E) = F$, and thus $E \subset \Lambda^{-1}(F)$, and by compactness of $E$
\[
(\Lambda_*\nu_E)(F) = \nu_E(\Lambda^{-1}(F)) = \nu_E(E).
\]

By the dual definition of capacity,
\begin{equation}
\begin{split}
&\capp E = \sup\left\{\frac{(\nu(E))^2}{\mathcal{E}[\nu]}:\;\supp\nu\subset E\right\},\\
&\capp_{(\frac12,\frac12)} F = \sup\left\{\frac{(\mu(F))^2}{\mathcal{E}_{(\frac12,\frac12)}[\mu]}:\;\supp\mu\subset E\right\},
\end{split}
\end{equation}
and $\nu_E, \mu_F$ are the respective maximizers. Therefore 
\[
\capp E = \frac{(\nu_E(E))^2}{\mathcal{E}[\nu_E]} = \frac{((\Lambda_*\nu_E)(F))^2}{\mathcal{E}[\nu_E]},
\]
and
\[
\capp_{(\frac12,\frac12)} F = \frac{(\mu_F(F))^2}{\mathcal{E}_{(\frac12,\frac12)}[\mu_F]} = \frac{((\Lambda^*\mu_F)(E))^2}{\mathcal{E}_{(\frac12,\frac12)}[\mu_F]},
\]
so in order to prove \eqref{e:A.71} it is enough to show that
\begin{subequations}\label{e:A.82}
\begin{eqnarray}
\label{e:A.82.1}& \mathcal{E}[\nu_E] \approx \mathcal{E}_{(\frac12,\frac12)}[\Lambda_*\nu_E],\\
\label{e:A.82.2}& \mathcal{E}_{(\frac12,\frac12)}[\mu_F] \approx \mathcal{E}[\Lambda^*\mu_F].
\end{eqnarray}
\end{subequations}
Both of these equivalences follow from Lemma \ref{l:A.9} below. Indeed, to obtain \eqref{e:A.82.1} we apply this Lemma with $\nu = \nu_E$ and $\mu = \Lambda_*\nu_E$, similarly \eqref{e:A.82.2} follows by assuming $\nu = \Lambda^*\mu_F$ and $\mu = \mu_F$.
\end{proof}
\begin{lemma}\label{l:A.9}
Consider two Borel measures $\mu,\nu\geq0$ on $(\partial\mathbb{D})^2$ and $(\partial T)^2$ respectively such that for any $\alpha\in (\partial T)^2$ they satisfy
\begin{equation}\label{e:A.85}
\mu(J_{\alpha}) = \nu(\Lambda^{-1}(J_{\alpha}));\;\;\; \nu(R_{\alpha}) = \mu(\Lambda(R_{\alpha})),
\end{equation}
and their respective energies are finite.
Then 
\begin{equation}\label{e:A.83}
\mathcal{E}[\nu] \approx \mathcal{E}_{(\frac12,\frac12)}[\mu].
\end{equation}
\end{lemma}
\begin{proof}
We start from the continuous side. Define $M := \Lambda^*m$, where $m$ is the normalized area measure on the torus $(\partial\mathbb{D})^2$. Clearly, $M(\partial{S}(\alpha)) = 2^{-d_{T}(\alpha_x) -d_{T}(\alpha_y)+2} = m(J_{\alpha})$. Given a point $z\in (\partial\mathbb{D})^2$, set
\[
\mathcal{P}_{\mathfrak{G}^2}(z) := \bigcup_{\omega\in\Lambda^{-1}(\{z\})}\mathcal{P}_{\mathfrak{G}^2}(\omega).
\] 
First we discretize the Bessel potential, namely, we show that for any $z\in (\partial\mathbb{D})^2$ one has
\begin{equation}\label{e:A.86}
\int_{(\partial\mathbb{D})^2}b_{(\frac12,\frac12)}(z,\zeta)\,d\mu(\zeta) \approx \sum_{\alpha\in \mathcal{P}_{\mathfrak{G}^2}(z)}\frac{(\mathbb{I}^*\nu)(\alpha)}{M^{\frac12}(\partial\mathcal{S}(\alpha))}.
\end{equation}
Fix a point $z = (e^{i\theta_1},e^{i\theta_2})\in(\partial\mathbb{D})^2$. 
Let $\tilde{J}_{z}(\varepsilon,\delta)$ be a rectangle with centerpoint $z$, and the sidelengths $4\pi\varepsilon,4\pi\delta$, that is,
\[
\tilde{J}_{z}(\varepsilon,\delta) = \left\{\zeta = (e^{i\eta_1},e^{i\eta_2})\in(\partial\mathbb{D})^2:\; \frac{|\theta_1-\eta_1|}{2\pi} \leq \varepsilon;\, \frac{|\theta_2-\eta_2|}{2\pi}\leq\delta\right\}.
\]
 A simple computation gives
\begin{equation}\notag
\begin{split}
&\int_0^{2\pi}\int_0^{2\pi}\frac{d\mu(e^{i\eta_1}, e^{i\eta_2})}{|\theta_1 - \eta_1|^{\frac12}|\theta_2-\eta_2|^{\frac12}} \approx \sum_{n_1\geq0}\sum_{n_2\geq0}\int_{2^{-n_1-1}\leq \frac{|\theta_1-\eta_1|}{2\pi} \leq 2^{-n_1}}\int_{2^{-n_2-1}\leq \frac{|\theta_2-\eta_2|}{2\pi} \leq 2^{-n_2}}\frac{\,d\mu(e^{i\eta_1},e^{i\eta_2})}{2^{-\frac{n_1}{2}}2^{-\frac{n_2}{2}}}\approx\\
&\sum_{n_1\geq0}\sum_{n_2\geq0}\int_{\frac{|\theta_1-n_1|}{2\pi} \leq 2^{-\eta_1}}\int_{\frac{|\theta_2-\eta_2|}{2\pi} \leq 2^{-n_2}}\frac{\,d\mu(e^{i\eta_1},e^{i\eta_2})}{2^{-\frac{n_1}{2}}2^{-\frac{n_2}{2}}} = \sum_{n_1\geq0}\sum_{n_2\geq0}\frac{\mu(\tilde{J}_z(2^{-n_1},2^{-n_2}))}{2^{-\frac{n_1}{2}-\frac{n_2}{2}}}\approx\\
&\sum_{n_1\geq0}\sum_{n_2\geq0}\frac{\mu(\tilde{J}_z(10\cdot2^{-n_1},10\cdot2^{-n_2}))}{2^{-\frac{n_1}{2}-\frac{n_2}{2}}}.
\end{split}
\end{equation}
Fix some $n_1,n_2 \geq 0$ and consider $\alpha\in \mathcal{P}_{\mathfrak{G}^2}(z)$ such that $ d_T(\alpha_x) = n_1+1$, $d_T(\alpha_y) = n_2+1$. Denote the collection of such points by $N_z(n_1,n_2)$. Then
we have
\[
\tilde{J}_z(2^{-n_1},2^{-n_2}) \subset \bigcup_{\alpha \in N_z(n_1,n_2)}J_{\alpha} \subset \tilde{J}_z(10\cdot2^{-n_1},10\cdot2^{-n_2}).
\]
Indeed, if $\alpha$ is such a point, then $J_{\alpha}$ either contains $z$, or is one of the neighbouring rectangles of the same generation, and vice versa, all such rectangles correspond to some point in $\mathcal{P}_{\mathfrak{G}^2}(z)$. It follows that
\[
\int_{(\partial\mathbb{D})^2}b_{(\frac12,\frac12)}(z,\zeta)\,d\mu(\zeta) \approx\sum_{n_1,n_2\geq0}\frac{\mu\left(\bigcup_{\alpha\in N_z(n_1,n_2)}J_{\alpha}\right)}{2^{-\frac{n_1}{2}-\frac{n_2}{2}}}.
\]
Since $\mathcal{E}_{(\frac12,\frac12)}[\mu] < +\infty$, we have $\mu(\partial J_{\alpha}) = 0$ for any $\alpha$ by Lemma \ref{l:A.8}. Combined with \eqref{e:A.85} we obtain
\[
\mu\left(\bigcup_{\alpha\in N_z(n_1,n_2)}J_{\alpha}\right) = \sum_{\alpha\in N_z(n_1,n_2)} \mu(J_{\alpha}) = \sum_{\alpha\in N_Z(n_1,n_2)}\nu(\Lambda^{-1}(J_{\alpha})) = \sum_{\alpha\in N_z(n_1,n_2)}(\mathbb{I}^*\nu)(\alpha),
\]
the last equality following from the fact that $\supp\nu\subset (\partial T)^2$ and $\Lambda(\partial\mathcal{S}(\alpha)) = J_{\alpha}$. Gathering the estimates we arrive at
\begin{equation}\notag
\int_{(\partial\mathbb{D})^2}b_{(\frac12,\frac12)}(z,\zeta)\,d\mu(\zeta) \approx \sum_{n_1,n_2\geq0}\sum_{\alpha\in N_z(n_1,n_2)}\frac{(\mathbb{I}^*\nu)(\alpha)}{M(\partial \mathcal{S}(\alpha))^{\frac12}}= \sum_{\alpha\in\mathcal{P}_{\mathfrak{G}^2}(z)}\frac{(\mathbb{I}^*\nu)(\alpha)}{M(\partial \mathcal{S}(\alpha))^{\frac12}},
\end{equation}
which is \eqref{e:A.86}.

The next part of the argument is actually a very special (linear) case of the well-known Wolff's inequality, that can be proven rather elementarily. We start by expanding the integrand,
\begin{equation}\notag
\begin{split}
&\int_{(\partial\mathbb{D})^2}\left(\sum_{\alpha\in\mathcal{P}_{\mathfrak{G}^2}(z)}\frac{(\mathbb{I}^*\nu)(\alpha)}{M(\partial \mathcal{S}(\alpha))^{\frac12}}\right)^2\,dm(z) = \int_{(\partial\mathbb{D})^2}\sum_{\alpha,\beta\in\mathcal{P}_{\mathfrak{G}^2}(z)}\frac{(\mathbb{I}^*\nu)(\alpha)(\mathbb{I}^*\nu)(\beta)}{M(\partial \mathcal{S}(\alpha))^{\frac12}M(\partial \mathcal{S}(\beta))^{\frac12}}\,dm(z).
\end{split}
\end{equation}
Recall that  $\alpha\in \mathcal{P}_{\mathfrak{G}^2}(\zeta)$ should be relatively close to the point $\zeta$, namely
\[
J_{\alpha} \subset \tilde{J}_{\zeta}(10\cdot2^{-d_T(\alpha_x)+1},10\cdot2^{-d_T(\alpha_y)+1}).
\]
Therefore $\{z: \alpha\in \mathcal{P}_{\mathfrak{G}^2}(z)\} \subset \tilde{J}_{\alpha}$, where $\tilde{J}_{\alpha} = \Lambda(\partial\mathcal{S}_{\mathfrak{G}^2}(p_2(\alpha)))$ and $p_2(\alpha) = (p_2(\alpha_x),p_2(\alpha_y))$ with $p_2(\alpha_x), p_2(\alpha_y)$ being the grandparents of $\alpha_x, \alpha_y$ in the tree geometry (if one of the points is the root $o$ or one of its children we assume the grandparent to be the root as well). It follows that
\begin{equation}\notag
\begin{split}
&\int_{(\partial\mathbb{D})^2}\sum_{\alpha,\beta\in\mathcal{P}_{\mathfrak{G}^2}(z)}\frac{(\mathbb{I}^*\nu)(\alpha)(\mathbb{I}^*\nu)(\beta)}{M(\partial \mathcal{S}(\alpha))^{\frac12}M(\partial \mathcal{S}(\beta))^{\frac12}}\,dm(z)  = \sum_{\alpha,\beta\in T^2}\int_{z: \alpha,\beta\in \mathcal{P}_{\mathfrak{G}^2}(z)}\frac{(\mathbb{I}^*\nu)(\alpha)(\mathbb{I}^*\nu)(\beta)}{M(\partial \mathcal{S}(\alpha))^{\frac12}M(\partial \mathcal{S}(\beta))^{\frac12}}\,dm(z)\leq\\
&\sum_{\alpha,\beta\in T^2}\frac{(\mathbb{I}^*\nu)(\alpha)(\mathbb{I}^*\nu)(\beta)}{M(\partial \mathcal{S}(\alpha))^{\frac12}M(\partial \mathcal{S}(\beta))^{\frac12}}\cdot m\left(\tilde{J}_{\alpha}\cap \tilde{J}_{\beta}\right) \lesssim\\& \sum_{\alpha,\beta\in T^2}\frac{(\mathbb{I}^*\nu)(\alpha)(\mathbb{I}^*\nu)(\beta)}{M(\partial \mathcal{S}(\alpha))^{\frac12}M(\partial \mathcal{S}(\beta))^{\frac12}}\cdot M\left(\partial\mathcal{S}_{\mathfrak{G}^2}(p_2(\alpha))\cap \partial\mathcal{S}_{\mathfrak{G}^2}(p_2(\beta))\right)\lesssim\\
& \sum_{\alpha,\beta\in T^2}\frac{(\mathbb{I}^*\nu)(\alpha)(\mathbb{I}^*\nu)(\beta)}{M(\partial \mathcal{S}_{\mathfrak{G}^2}(\alpha))^{\frac12}M(\partial \mathcal{S}_{\mathfrak{G}^2}(\beta))^{\frac12}}\cdot M\left(\partial\mathcal{S}_{\mathfrak{G}^2}(\alpha)\cap \partial\mathcal{S}_{\mathfrak{G}^2}(\beta)\right),
\end{split}
\end{equation}
since $M(\partial\mathcal{S}_{\mathfrak{G}^2}(p_2(\alpha))) \approx M(\partial\mathcal{S}(\alpha)))$ for any $\alpha\in T^2$.

A point $\gamma\in T^2$ is called a proper $\mathfrak{G}^2$-descendant of $\tau\in T^2$ if $\gamma\in \mathcal{S}_{\mathfrak{G}^2}(\tau)$, and it lies strictly below $\tau$, namely, either $d_{T}(\gamma_x) \geq d_T(\tau_x)$ and $d_{T}(\gamma_y) > d_{T}(\tau_y)$, or $d_{T}(\gamma_x) > d_T(\tau_x)$ and $d_{T}(\gamma_y) \geq d_{T}(\tau_y)$. Observe that for any two points $\alpha,\beta\in T^2$ there exists a (possibly non-unique) $\tau\in T^2$ such that $\alpha,\beta$ are proper $\mathfrak{G}^2$-descendants of $\tau$, and $\tau$ is minimal, that is, for any proper $\mathfrak{G}^2$-descendant $\gamma$ of $\tau$ one of the points $\alpha,\beta$ is not a proper $\mathfrak{G}^2$-descendant of $\gamma$. Clearly,
\begin{equation}\notag
\begin{split}
& \sum_{\alpha,\beta\in T^2}\frac{(\mathbb{I}^*\nu)(\alpha)(\mathbb{I}^*\nu)(\beta)}{M(\partial \mathcal{S}_{\mathfrak{G}^2}(\alpha))^{\frac12}M(\partial \mathcal{S}_{\mathfrak{G}^2}(\beta))^{\frac12}}\cdot M\left(\partial\mathcal{S}_{\mathfrak{G}^2}(\alpha)\cap \partial\mathcal{S}_{\mathfrak{G}^2}(\beta)\right) \leq\\
&\sum_{\tau\in T^2} \sum_{\alpha,\beta:\; \tau \;\textup{is minimal for}\; \alpha,\beta}\frac{(\mathbb{I}^*\nu)(\alpha)(\mathbb{I}^*\nu)(\beta)}{M(\partial \mathcal{S}_{\mathfrak{G}^2}(\alpha))^{\frac12}M(\partial \mathcal{S}_{\mathfrak{G}^2}(\beta))^{\frac12}}\cdot M\left(\partial\mathcal{S}_{\mathfrak{G}^2}(\alpha)\cap \partial\mathcal{S}_{\mathfrak{G}^2}(\beta)\right). 
\end{split}
\end{equation}
We aim to show that for any $\tau\in T^2$ one has
\begin{equation}\label{e:A.87}
\begin{split}
&\sum_{\alpha,\beta:\; \tau \;\textup{is minimal for}\; \alpha,\beta}\frac{(\mathbb{I}^*\nu)(\alpha)(\mathbb{I}^*\nu)(\beta)}{M(\partial \mathcal{S}_{\mathfrak{G}^2}(\alpha))^{\frac12}M(\partial \mathcal{S}_{\mathfrak{G}^2}(\beta))^{\frac12}}\cdot M\left(\partial\mathcal{S}_{\mathfrak{G}^2}(\alpha)\cap \partial\mathcal{S}_{\mathfrak{G}^2}(\beta)\right)\lesssim
(\nu(\mathcal{S}_{\mathfrak{G}^2}(\tau)))^2.
\end{split}
\end{equation}
Given $\alpha,\beta,\tau \in T^2$ let $a_x,a_y,b_x,b_y,t_x,t_y$ be their respective generation numbers, $a_x := d_T(\alpha_x)-1, a_y := d_T(\alpha_y)-1$ etc. First we note that $M(\partial \mathcal{S}_{\mathfrak{G}^2}(\alpha))^{\frac12} \approx 2^{-\frac{a_x+a_y}{2}}$, $M(\partial \mathcal{S}_{\mathfrak{G}^2}(\beta))^{\frac12} \approx 2^{-\frac{b_x+b_y}{2}}$, and\\ $M\left(\partial\mathcal{S}_{\mathfrak{G}^2}(\alpha))\cap \partial\mathcal{S}_{\mathfrak{G}^2}(\beta))\right) \lesssim 2^{-\max(a_x,b_x)-\max(a_y,b_y)}$, where the last equivalence is by a trivial estimate of $m(\tilde{J}_{\alpha}\cap \tilde{J}_{\beta})$. The key observation here is as follows.
\begin{proposition}\label{p:A.11} Assume that $\tau$ is minimal for $\alpha$ and $\beta$ and either $a_x\geq t_x+4$ and $b_x \geq t_x+4$, or $a_y\geq t_y+4$ and $b_y \geq t_y+4$. Then $\mathcal{S}_{\mathfrak{G}^2}(\alpha)\cap \mathcal{S}_{\mathfrak{G}^2}(\beta) = \emptyset$. In other words, if $\alpha$ and $\beta$ lie very 'deep' inside $\tau$ and are not 'perpendicular', then they must be 'far' from each other (see Fig. \ref{e:fig22}).
\end{proposition}
\begin{proof}
 Indeed, assume that $a_x\geq t_x+4$, $b_x \geq t_x+4$, and let $\gamma \in \mathcal{S}_{\mathfrak{G}^2}(\alpha) \cap \mathcal{S}_{\mathfrak{G}^2}(\beta)$, so that, in particular, $\gamma_x \in \mathcal{S}_{\mathfrak{G}}(\alpha_x)\cap\mathcal{S}_{\mathfrak{G}}(\beta_x)$. It follows immediately that both $\alpha_x,\beta_x$ belong either to $ S_{\mathfrak{G}}(p_2(\alpha_x))$ (if $b_x\geq a_x$) or to $\mathcal{S}_{\mathfrak{G}}(p_2(\beta_x))$ (if $b_x\leq a_x$). Suppose we are in the first case. Then, since $a_x \geq t_x+4$, we see that $p_2(\alpha_x) \in S_{\mathfrak{G}}(\tau_x)$ and $d_T(p_2(\alpha_x)) = a_x-1 \geq t_x+1$. Now define $\gamma = (\gamma_x,\gamma_y) := (p_2(\alpha_x),\tau_y)$. Clearly $\gamma$ is a proper $\mathfrak{G}^2$-descendant of $\tau$, and at the same time both $\alpha$ and $\beta$ are proper $\mathfrak{G}^2$-descendants of $\gamma$. We have a contradiction. The case $a_y\geq t_y+4,\; b_y\geq t_y+4$ is done similarly. \end{proof}
\begin{figure}[h]
\centering
\includegraphics[trim= 30mm 0mm 30mm 0mm,clip, width = 0.7\textwidth]{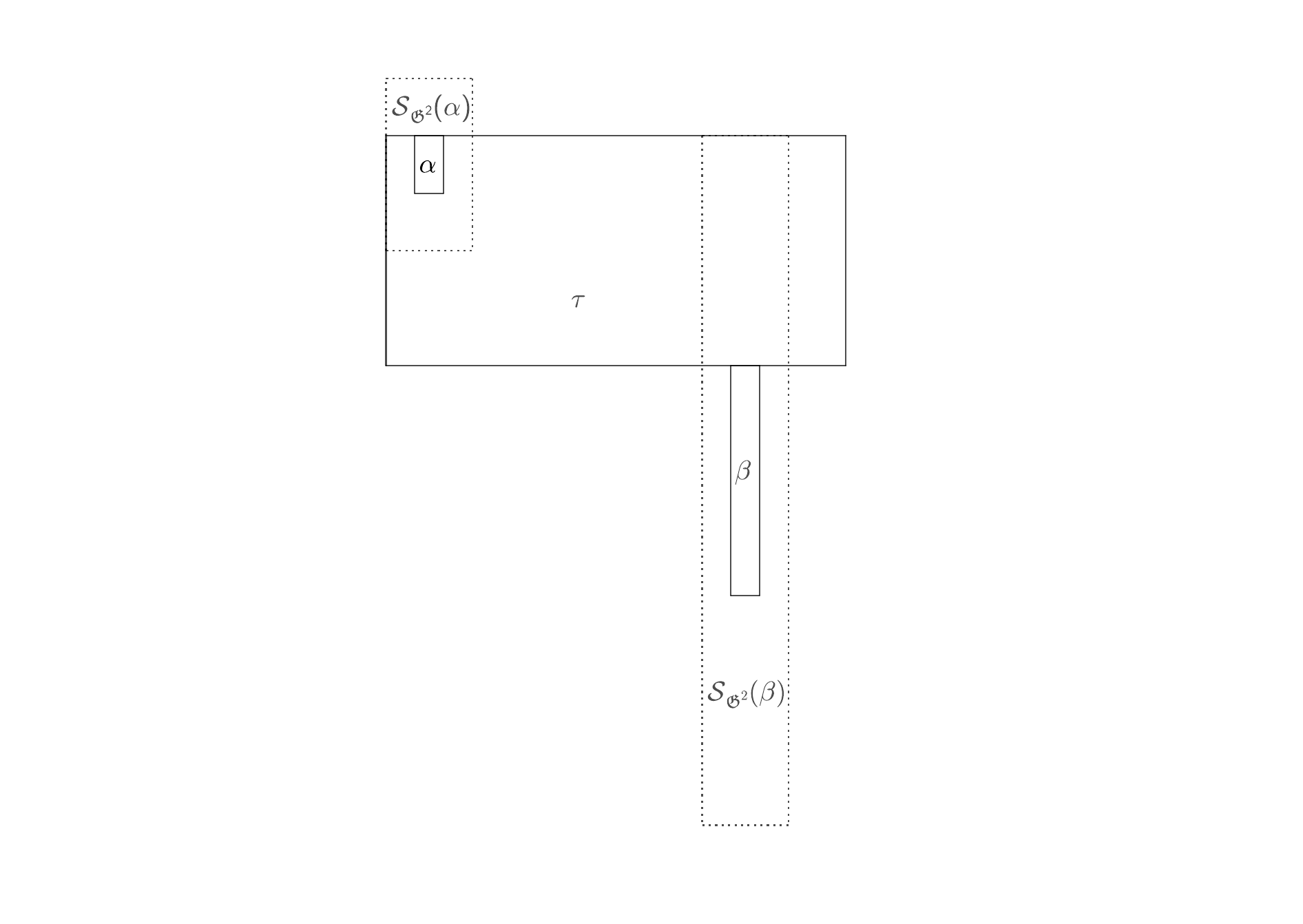}
\caption{}
\label{e:fig22}
\end{figure}

Now we are ready to return to \eqref{e:A.87}. Fix $\tau\in T^2$. By the observation above, if $\tau$ is minimal for $\alpha$ and $\beta$, then 
\begin{multline*}
0 \neq\frac{(\mathbb{I}^*\nu)(\alpha)(\mathbb{I}^*\nu)(\beta)}{M(\partial \mathcal{S}_{\mathfrak{G}^2}(\alpha)^{\frac12}M(\partial \mathcal{S}_{\mathfrak{G}^2}(\beta)^{\frac12}}\cdot M\left(\partial\mathcal{S}_{\mathfrak{G}^2}(\alpha))\cap \partial\mathcal{S}_{\mathfrak{G}^2}(\beta))\right) \\ \lesssim \frac{(\mathbb{I}^*\nu)(\alpha)(\mathbb{I}^*\nu)(\beta)}{2^{-\frac{a_x+a_y+b_x+b_y}{2}}}\cdot 2^{-\max(a_x,b_x) -\max(a_y,b_y)}
\end{multline*}
only if  one of $a_x,b_x$ and one of $a_y,b_y$ are comparable to $t_x$ and $t_y$, respectively. Note that we always have $a_x,b_x\geq t_x$ and $a_y,b_y\geq t_y$, since $\tau$ is minimal for $\alpha,\beta$. Given a point $\alpha\in T^2$, assume that $a_x\leq t_x+4$ and $a_y\leq t_y+4$, and for $b_x\geq t_x,\;b_y\geq t_y$ denote by $A_{\tau}(\alpha,b_x,b_y)$ the set of all points $\beta$ such that $d_T(\beta_x) = b_x+1, d_{T}(\beta_y) = b_y+1$, and $\tau$ is minimal for $\alpha$ and $\beta$. Then
\begin{equation}\notag
\begin{split}
&\sum_{b_x\geq t_x}\sum_{b_y\geq t_y}\sum_{\beta\in A_{\tau}(\alpha,b_x,b_y)}\frac{(\mathbb{I}^*\nu)(\alpha)(\mathbb{I}^*\nu)(\beta)}{2^{-\frac{a_x+a_y+b_x+b_y}{2}}}\cdot 2^{-\max(a_x,b_x) -\max(a_y,b_y)}\approx\\
&(\mathbb{I^*}\nu)(\alpha)\cdot \sum_{b_x\geq t_x}\sum_{b_y\geq t_y}2^{\frac{t_x+t_y+b_x+b_y}{2}-b_x -b_y}\sum_{\beta\in A_{\tau}(\alpha,b_x,b_y)}(\mathbb{I}^*\nu)(\beta)\approx (\mathbb{I^*}\nu)(\alpha)\sum_{\beta\in A_{\tau}(\alpha,b_x,b_y)}(\mathbb{I}^*\nu)(\beta)\lesssim\\
&(\mathbb{I^*}\nu)(\alpha)\nu(\mathcal{S}_{\mathfrak{G}^2}(\tau)),
\end{split}
\end{equation}
since $\mathcal{S}(\beta^1)\cap\mathcal{S}(\beta^2) = \emptyset$ for any non-identical pair  $\beta^1,\beta^2\in A_{\tau}(\alpha,b_x,b_y)$, and $\bigcup_{\beta\in A_{\tau}(\alpha,b_x,b_y)} \subset \mathcal{S}_{\mathfrak{G}^2}(\tau)$. It follows that
\[
\sum_{a_x=t_x}^{t_x+4}\sum_{a_y=t_y}^{t_y+4}\sum_{\alpha:\; \alpha \;\textup{is a proper $\mathfrak{G}^2$-descendant of}\; \tau,\; d_{T}(\alpha_x) = a_x+1,\,d_T(\alpha_y)=a_y+1}(\mathbb{I^*}\nu)(\alpha)\nu(\mathcal{S}_{\mathfrak{G}^2}(\tau)) \approx \left(\nu(\mathcal{S}_{\mathfrak{G}^2}(\tau))\right)^2.
\]

The remaining cases (i.e. $a_x\leq t_x+4$ and $b_y\leq t_y+4$, $b_x\leq t_x+4$ and $a_y\leq t_y+4$, or $b_x\leq t_x+4$ and $b_y\leq t_y+4$) are dealt with similarly. Gathering all the estimates we arrive at \eqref{e:A.87}, hence
\[
\sum_{\tau\in T^2}\sum_{\alpha,\beta:\; \tau \;\textup{is minimal for}\; \alpha,\beta}\frac{(\mathbb{I}^*\nu)(\alpha)(\mathbb{I}^*\nu)(\beta)}{M(\partial \mathcal{S}_{\mathfrak{G}^2}(\alpha))^{\frac12}M(\partial \mathcal{S}_{\mathfrak{G}^2}(\beta))^{\frac12}}\cdot M\left(\partial\mathcal{S}_{\mathfrak{G}^2}(\alpha))\cap \partial\mathcal{S}_{\mathfrak{G}^2}(\beta))\right)\lesssim
\sum_{\tau\in T^2}(\nu(\mathcal{S}_{\mathfrak{G}^2}(\tau)))^2.
\]
By the $\mathfrak{G}$-neighbours argument of Section \ref{SS:2.4} we have
\[
\sum_{\tau\in T^2}(\nu(\mathcal{S}_{\mathfrak{G}^2}(\tau)))^2 \approx \sum_{\tau\in T^2}((\mathbb{I^*}\nu)(\tau))^2 = \mathcal{E}[\nu],
\]
and therefore
\begin{equation}\label{e:A.89}
\int_{(\partial\mathbb{D})^2}\sum_{\alpha,\beta\in\mathcal{P}_{\mathfrak{G}^2}(z)}\frac{(\mathbb{I}^*\nu)(\alpha)(\mathbb{I}^*\nu)(\beta)}{M(\partial \mathcal{S}(\alpha))^{\frac12}M(\partial \mathcal{S}(\beta))^{\frac12}}\,dm(z) \lesssim \mathcal{E}[\nu].
\end{equation}
On the other hand,
\begin{equation}\label{e:A.90}
\begin{split}
&\int_{(\partial\mathbb{D})^2}\sum_{\alpha,\beta\in\mathcal{P}_{\mathfrak{G}^2}(z)}\frac{(\mathbb{I}^*\nu)(\alpha)(\mathbb{I}^*\nu)(\beta)}{M(\partial \mathcal{S}(\alpha))^{\frac12}M(\partial \mathcal{S}(\beta))^{\frac12}}\,dm(z) \geq \int_{(\partial\mathbb{D})^2}\sum_{\alpha\in\mathcal{P}_{\mathfrak{G}^2}(z)}\frac{((\mathbb{I}^*\nu)(\alpha))^2}{M(\partial \mathcal{S}(\alpha))}\,dm(z)\geq\\
&\sum_{\alpha\in T^2}\frac{((\mathbb{I}^*\nu)(\alpha))^2}{M(\partial \mathcal{S}(\alpha))}\cdot m\{z:\; \alpha\in\mathcal{P}_{\mathfrak{G}^2}(z)\} \geq \sum_{\alpha\in T^2}\frac{((\mathbb{I}^*\nu)(\alpha))^2}{M(\partial \mathcal{S}(\alpha))}\cdot m(J_{\alpha}) = \sum_{\alpha\in T^2}((\mathbb{I}^*\nu)(\alpha))^2 = \mathcal{E}[\nu].
\end{split}
\end{equation}
Combined with \eqref{e:A.86} these estimates give us
\begin{multline*}
\mathcal{E}_{(\frac12,\frac12)}[\mu] = \int_{(\partial\mathbb{D})^2}\left(\int_{(\partial\mathbb{D})^2}b_{(\frac12,\frac12)}(z,\zeta)\,d\mu(\zeta)\right)^2\,dm(z) \\ \approx \int_{(\partial\mathbb{D})^2}\left(\sum_{\alpha\in\mathcal{P}_{\mathfrak{G}^2}(z)}\frac{(\mathbb{I}^*\nu)(\alpha)}{M(\partial \mathcal{S}(\alpha))^{\frac12}}\right)^2\,dm(z)\approx \mathcal{E}[\nu],
\end{multline*}
which finishes the proof.
\end{proof}
\subsection{}\label{S:A.1}
\begin{lemma}\label{l:A.1.1}
The following properties hold:

\begin{enumerate}
\item  Let $E$ be a Borel subset of $\overline{T}^2$ and let $\mu\geq0$ be a Borel measure on $\overline{T}^2$ such that $\mathcal{E}[\mu]<\infty$ and $\mathbb{V}^{\mu} \leq 1$ quasi-everywhere on $E$. Then $\capp E \geq \mu(E)$. \label{l:A1.s.2}
\item Let $F$ be an open set in $\partial T$ and $\rho$ be its equilibrium measure. Then $\supp\rho = \textup{cl}F$, and $V^{\rho} \equiv 1$ on $F$. \label{l:A1.71}
\item  \textbf{Maximum principle.} Let $\mu\geq0$ be a measure on $\overline{T}$ with finite energy. Then
\begin{equation}\label{e:A1.s.2}
\sup_{\alpha\in\supp\mu}V^{\mu}(\alpha) = \sup_{\alpha\in \overline{T}}V^{\mu}(\alpha).
\end{equation}\label{l:A1.s.3}
\item \textbf{Domination principle.} Let $f$ be a non-negative function in $\ell^2(T)$ such that $f(\alpha) \geq f(\alpha^+) + f(\alpha^-)$ for any point $\alpha \in T$ and its two children $\alpha^{\pm}$. Suppose $\nu\geq0$ is a Borel measure on $\overline{T}$ with finite energy such that
\begin{equation}\label{e:A1.c.4}
(If)(\alpha) \geq V^{\nu},\quad \alpha\in \supp\nu.
\end{equation}
Then this inequality holds everywhere,
\begin{equation}\label{e:A1.s.4}
(If)(\alpha) \geq V^{\nu}(\alpha),\quad \alpha\in \overline{T}.
\end{equation}
\label{l:A1.s.4}
\end{enumerate}
\end{lemma}
\begin{proof}
\textit{Property \ref{l:A1.s.2}}. Define, as usual, the restricted measure $\mu\vert_E$ by $\mu\vert_E(F) := \mu(E\cap F)$, and let $\mu_E$ be the equilibrium measure of $E$. Clearly, $\mathcal{E}[\mu\vert_E] < \infty$ and $\mathbb{V}^{\mu\vert_E} \leq \mathbb{V}^{\mu}$. We have
\[
\mathcal{E}[\mu\vert_E] = \int_{\overline{T}^2}\mathbb{V}^{\mu\vert_E}\,d\mu\vert_E \leq \int_{\overline{T}^2}\mathbb{V}^{\mu_E}\,d\mu\vert_E = \mathcal{E}[\mu_E,\mu\vert_E],
\]
since $\mathbb{V}^{\mu_E}\geq1$ quasi-everywhere on $E$. Hence, since $\mathcal{E}[\mu_E - \mu\vert_E] \geq 0$,
\[
\mu(E) \leq \mathcal{E}[\mu_E,\mu\vert_E] \leq \mathcal{E}[\mu_E] = \capp E.
\]

\textit{Property \ref{l:A1.71}}.\;\; We first show that $V^{\rho} = 1$ on $F$. Fix a point $\omega\in F$. Since $F$ is open, there is a point $\tau>\omega$ such that $\partial\mathcal{S}(\tau)\subset F$. Since $\rho$ minimizes the energy, an elementary computation shows that for any $\alpha\leq \tau$ one must have $(I^*\rho)(\alpha) = 2(I^*\rho)(\alpha^+) = 2(I^*\rho)(\alpha^-)$, where $\alpha^{\pm}$ are the children of $\alpha$. Therefore $V^{\rho}$ is actually constant on $\partial\mathcal{S}(\tau)$, and thus $V^{\rho}(\omega) = 1$.

Next we show that $\supp\rho = \textup{cl}F$. Let $\omega\in \textup{cl}F$, and consider an open neighbourhood $U_{\omega}$ of $\omega$. If $\rho(U_{\omega}) = 0$ there is a smallest point $\alpha > \omega$ such that $\rho(\mathcal{S}(\alpha)) \neq 0$. Denote its two children by $\alpha^+$ and $\alpha^-$, and assume that $\omega\in\mathcal{S}(\alpha^+)$. Since $F$ is open, $\mathcal{S}(\alpha^+) \cap F \neq \emptyset$. Therefore, $V^{\rho}(\alpha^+) = V^{\rho}(\alpha) \geq 1$, since $V^{\rho} = 1$ on $F$. On the other hand, $(I^*\rho)(\alpha^-) >0$, by the minimality of $\alpha$. Thus $V^{\rho}(\alpha^-) >1$, which contradicts the fact that $V^{\rho} \leq 1$ on $\supp\rho$. Therefore it must have been that $\rho(U_{\omega}) > 0$, and thus that $\textup{cl}F\subset \supp\rho$. The converse is elementary.

\textit{Property \ref{l:A1.s.3}.}\; It is enough to check \eqref{e:A1.s.2} inside the tree (i.e. for $\alpha\in T$), since for any $\beta\in\partial T$ we have $V^{\mu}(\beta) = \sup_{\alpha > \beta}V^{\mu}(\alpha)$. Now assume that there exists a point $\beta\in T\setminus\supp\mu$ such that
\[
V^{\mu}(\beta) > V^{\mu}(\alpha),\quad\alpha\in \supp\mu.
\]
We see immediately that $\mathcal{S}(\beta)\cap \supp\mu = \emptyset$, and hence  there exists a unique point $\tau_{\beta}>\beta$ such that $\mathcal{S}(\tau_{\beta})\cap\supp\mu\neq\emptyset$, but $\mathcal{S}(\tau)\cap\supp\mu=\emptyset$ for every $\tau<\tau_{\beta}$. Then $(I^*\mu)(\tau) = 0$ for such points $\tau$, and
\[
V^{\mu}(\beta) = V^{\mu}(\tau_{\beta}) + \sum_{\tau_{\beta}>\tau\geq\beta}(I^*\mu)(\beta) = V^{\mu}(\tau_{\beta}).
\]
Monotonicity of $V^{\mu}$, with respect to natural order on $T$, implies that $V^{\mu}(\tau_{\beta}) < V^{\mu}(\alpha)$ for any $\alpha \in \mathcal{S}(\tau_{\beta})\cap \supp\mu$, yielding a contradiction.

\textit{Property \ref{l:A1.s.4}.}
As before, it is enough to show \eqref{e:A1.s.4} only for points inside $T$. Now suppose there exists $\alpha^0\in T$ such that
\[
(If)(\alpha^0) < V^{\nu}(\alpha^0)
\]
and
\begin{equation}\notag
(If)(\tau) \geq V^{\nu}(\tau),\quad \tau > \alpha^0.
\end{equation}
It follows immediately that $f(\alpha^0) < (I^*\nu)(\alpha^0)$. Hence one of the children of $\alpha^0$, which we denote by $\alpha_1$, satisfies $f(\alpha^1) < (I^*\nu)(\alpha^1)$. Continuing this argument we obtain a sequence $\{\alpha^k\}_{0}^{\infty}$ of nested points such that $f(\alpha^k) < (I^*\nu)(\alpha^k),\; k=0,1,\dots$. Denote the endpoint of this geodesic by $\omega = \bigcap_{k}\mathcal{S}(\alpha^k)$. Clearly $\omega\in \supp \nu$. It follows that
\[
(If)(\omega) = (If)(\alpha^0) + \sum_{k=1}^{\infty}f(\alpha^k) < V^{\nu}(\alpha^0) + \sum_{k=1}^{\infty}(I^*\nu)(\alpha^k) = V^{\nu}(\omega),
\]
a contradiction.
\end{proof}

\subsection{}\label{S:A.2}
\begin{proposition}\label{p:A.2.1}
For any $\lambda>1$ there exists a point $\omega\in \overline{T}^2$ and a set $E\subset \overline{T}^2$ such that the equilibrium measure $\mu= \mu_E$ of this set satisfies
\begin{equation}\label{e:pA.2.1}
\mathbb{V}^{\mu}(\omega) \geq \lambda.
\end{equation}
\end{proposition}
\begin{proof}
Put $n = 20([\lambda]+1)$ and $k = 20^n$. Now fix any point $\omega\in (\partial T)^2$ and for any $0\leq i \leq n$ consider the unique point $\alpha^i = (\alpha^i_x,\alpha^i_y)$ that satisfies $\alpha^i \geq\omega,\; d_T(\alpha^i_x) = 20^{n-i},\; d_T(\alpha_y^i) = 20^i$, so that $d_{T^2}(\alpha^i) = 20^n=k$. We have $\alpha_x^i\wedge\alpha_x^j = \alpha_x^i$, $\alpha_y^i\wedge\alpha_y^j = \alpha_y^j$ for $i>j$, hence
\begin{equation}\label{e:221}
\sum_{i\neq j}d_{T^2}(\alpha^i\wedge\alpha^j) = \sum_{j<i}20^{n+j-i} + \sum_{j>i}20^{n+i-j} \leq \frac{2}{19}20^n \leq \frac{k}{9}.
\end{equation}
Now let $E := \{\alpha^i\}_{i=0}^n$ and $\mu=\mu_E$ to be the equilibrium measure of $E$. We claim that $\mathbb{V}^{\mu}(\omega) \geq \frac{n}{20} \geq\lambda$.

To show this we first note  that the values of $\mu$ at $\alpha^i,\; i=0,\dots , n$ are more or less the same,
\begin{equation}\label{e:222}
\sup_{0\leq i\leq n}\mu(\alpha^i) \leq 5\inf_{0\leq i \leq n}\mu(\alpha^i).
\end{equation}
Indeed, assume that $i_1,i_2$ are such that $\sup_{0\leq i\leq n}\mu(\alpha^i) = \mu(\alpha^{i_1})$, $\inf_{0\leq i\leq n}\mu(\alpha^i) = \mu(\alpha^{i_2})$, and 
\[
\mu(\alpha^{i_1}) > 5\mu(\alpha^{i_2}).
\]
Since every element of $E$ has non-zero capacity (actually $\capp \{\alpha^i\} \equiv \frac{1}{k} >0$), we have
\begin{equation}\notag
\begin{split}
&1\leq\mathbb{V}^{\mu}(\alpha^{i_2}) = \sum_{i\neq i_2}d_{T^2}(\alpha^{i}\wedge\alpha^{i_2})\mu(\alpha^i) + k\mu(\alpha^{i_2})\leq\\
&\sum_{i\neq i_2}d_{T^2}(\alpha^{i}\wedge\alpha^{i_2})\mu(\alpha^{i_1}) + k\mu(\alpha^{i_2})\leq \frac{k}{9}\mu(\alpha^{i_1}) + \frac{k}{5}\mu(\alpha^{i_1})\leq\\
&\frac12k\mu(\alpha^{i_1}) + \frac12\sum_{i\neq i_1}d_{T^2}(\alpha^{i}\wedge\alpha^{i_1})\mu(\alpha^i) = \frac12\mathbb{V}^{\mu}(\alpha^{i_1}) .
\end{split}
\end{equation}
On the other hand, $\mu(\alpha^{i_1})>0$, hence $\mathbb{V}^{\mu}(\alpha^{i_1}) = 1$, and we have a contradiction.

Furthermore,
\begin{equation}\notag
\begin{split}
&1= \mathbb{V}^{\mu}(\alpha^{i_1}) = k\mu(\alpha^{i_1}) + \sum_{i\neq i_1}d_{T^2}(\alpha^{i}\wedge\alpha^{i_1})\mu(\alpha^i)\leq \\
&k\mu(\alpha^{i_1}) + \sum_{i\neq i_1}d_{T^2}(\alpha^{i}\wedge\alpha^{i_1})\mu(\alpha^{i_1}) \leq \frac{10}{9}k\mu(\alpha^{i_1}).
\end{split}
\end{equation}
Therefore, for any $0\leq i\leq n$,
\[
\mu(\alpha^{i}) \geq\frac15\mu(\alpha^{i_1})\geq \frac{9}{50k}.
\]
It follows immediately that
\[
\mathbb{V}^{\mu}(\omega) = \sum_{i=0}^nd_{T^2}(\omega\wedge\alpha^i)\mu(\alpha^i) = \sum_{i=0}^nd_{T^2}(\alpha^i)\mu(\alpha^i) \geq kn\mu(\alpha^{i_2}) \geq\frac{9n}{50},
\]
and we are done.
\end{proof}
\begin{proposition}\label{p:A.2.2}
For any $\lambda > 0$ there exists a pair of measures $\mu,\nu\geq0$ on $\overline{T}^2$ such that
\begin{equation}\label{e:pA.2.2.c}
\mathbb{V}^{\nu}(\alpha) \geq \mathbb{V}^{\mu}(\alpha),\quad\alpha\in \supp\mu,
\end{equation}
but
\begin{equation}\label{e:pA.2.2.s}
\sup_{\alpha\in \overline{T}^2}\mathbb{V}^{\nu}(\alpha) \leq \sup_{\alpha\in \overline{T}^2}\frac{1}{\lambda}\mathbb{V}^{\mu}(\alpha).
\end{equation}
\end{proposition}
\begin{proof}
This is a direct corollary of Proposition \ref{p:A.2.1} above. Indeed, given $\lambda>0$ let $\mu = \mu_E$ be as in \eqref{e:pA.2.1}, and let $\nu := \chi_{o}$ be the unit point mass at the root. Then, clearly, 
$\mathbb{V}^{\nu}\equiv 1\;\textup{on}\; \overline{T}^2$, and in particular on $\supp\mu$, but $\sup_{\alpha\in \overline{T}^2}\mathbb{V}^{\mu}\geq\lambda$.
\end{proof}
\subsection{}\label{S:A.3}
Let $m_0$ be normalized length on $\partial\mathbb{D}$.  Define $M_0$ to be its natural pull-back on $\partial T$, $M_0 = \Lambda_0^*m_0$. In particular, we have
\[
M_0(\partial\mathcal{S}(\beta)) = 2^{-d_T(\beta)+1},\quad \beta\in T.
\]
Similarly, as in Lemma \ref{l:A.9}, let $M = \Lambda^*m$, where $m$ is normalized area measure on the torus $(\partial\mathbb{D})^2$. Clearly,
\[
M(\partial \mathcal{S}(\beta)) = 2^{-d_{T}(\beta_x)-d_{T}(\beta_y)+2} = M_0(\partial \mathcal{S}(\beta_x))M_0(\partial\mathcal{S}(\beta_y)),\quad \beta = (\beta_x,\beta_y) \in T^2.
\]
Let us show that $d_{T^2}$ is almost a martingale with respect to the measure $M$.
\begin{lemma}\label{l:A.3.1}
Assume that $\alpha,\beta\in T^2$. Then
\begin{equation}\label{e:91}
d_{T^2}(\alpha\wedge\beta) \approx \frac{1}{M(\partial\mathcal{S}(\alpha))M(\partial\mathcal{S}(\beta))}\int_{\partial\mathcal{S}(\alpha)}\int_{\partial\mathcal{S}(\beta)}d_{T^2}(\xi\wedge\omega)dM(\xi)\,dM(\omega).
\end{equation}
\end{lemma}
\begin{proof}
Due to multiplicativity it is enough to prove that, say,
\begin{equation}\notag
d_{T}(\alpha_x\wedge\beta_x) \approx \frac{1}{M_0(\partial\mathcal{S}(\alpha_x))M_0(\partial\mathcal{S}(\beta_x))}\int_{\partial\mathcal{S}(\alpha_x)}\int_{\partial\mathcal{S}(\beta_x)}d_{T}(\xi_x\wedge\omega_x)dM(\xi_x)\,dM(\omega_x).
\end{equation}
If $\xi_x\leq\alpha_x$ and $\omega_x\leq\beta_x$, then $d_{T}(\xi_x\wedge\omega_x) \geq d_{T}(\alpha_x\wedge\beta_x)$, hence
\[
d_{T}(\alpha_x\wedge\beta_x) \leq \frac{1}{M_0(\partial\mathcal{S}(\alpha_x))M_0(\partial\mathcal{S}(\beta_x))}\int_{\partial\mathcal{S}(\alpha_x)}\int_{\partial\mathcal{S}(\beta_x)}d_{T}(\xi_x\wedge\omega_x)dM_0(\xi_x)\,dM_0(\omega_x).
\]
To get the reverse inequality we first show that for any $\beta_x\in T_x$ and $\tau_x\in \overline{T}_x$ we have
\begin{equation}\label{e:92}
\frac{1}{M_0(\partial\mathcal{S}(\beta_x))}\int_{\partial\mathcal{S}(\beta_x)}d_{T}(\tau_x\wedge\omega_x)\,dmM_0(\omega_x) \leq 3d_{T}(\tau_x\wedge\beta_x).
\end{equation}
If $\tau_x\geq\beta_x$ or these two points are not comparable, then, clearly, $d_T(\tau_x\wedge\beta_x) = d_T(\tau_x\wedge\omega_x)$ for $\omega_x\leq\beta_x$, and \eqref{e:92} is trivial. Hence from  now on we assume that $\tau_x < \beta_x$. Let $n:= d_T(\beta_x)$ and $N:= d_T(\tau_x)$. For every $n\leq k \leq N$ there exists exactly one point $\gamma_k\in T_x$ such that $\tau_x \leq\gamma_k \leq\beta_x$, and $d_T(\gamma_k) = k$ (in particular $\gamma_n = \beta_x,\; \gamma_N = \tau_x$). Define 
\[
S_k = \partial\mathcal{S}(\gamma_k)\setminus\partial\mathcal{S}(\gamma_{k+1}),\quad n\leq k \leq N-1,
\]
and
\[
S_N = \partial\mathcal{S}(\tau_x).
\]
If $\omega_x\in S_k$, then, clearly, $d_T(\tau_x\wedge\omega_x) = k$. Moreover, these sets are disjoint and form a covering of $\partial\mathcal{S}(\beta_x)$. Also $M_0(S_k) = 2^{-k}- 2^{-k-1}, n\leq k \leq N-1$ and $M_0(S_N) = 2^{-N}$. We have
\begin{equation}\notag
\begin{split}
&\frac{1}{M_0(\partial\mathcal{S}(\beta_x))}\int_{\partial\mathcal{S}(\beta_x)}d_T(\tau_x\wedge\omega_x)\,dM_0(\omega_x) =\\
& 2^{d_T(\beta_x)}\sum_{k=n}^N\int_{S_k}d_t(\tau_x\wedge\omega_x)\,dM_0(\omega_x) =
2^n\sum_{k=n}^Nk\cdot M_0(S_k) \leq\\
& 2^n\sum_{k=n}^Nk2^{-k}\leq 3n = 3d_T(\tau_x\wedge\beta_x),
\end{split}
\end{equation}
and we arrive at \eqref{e:92}. It follows immediately that
\begin{equation}\notag
\begin{split}
&\frac{1}{M_0(\partial\mathcal{S}(\alpha_x))M_0(\partial\mathcal{S}(\beta_x))}\int_{\partial\mathcal{S}(\alpha_x)}\int_{\partial\mathcal{S}(\beta_x)}d_{T}(\xi_x\wedge\omega_x)\,dM_0(\xi_x)\,dM_0(\omega_x) \leq\\
&3\frac{1}{M_0(\partial\mathcal{S}(\alpha_x))}\int_{\partial\mathcal{S}(\alpha_x)}d_{T}(\xi_x\wedge\beta_x)M_0(\xi_x) \leq 9 d_T(\alpha_x\wedge\beta_x).
\end{split}
\end{equation}
\end{proof}
\begin{theorem}\label{t:th.A.3}
Suppose that $\mu\geq0$ is a Borel measure on $\overline{T}^2$ with finite energy. By the Disintegration Theorem we can define a measure $\mu_b$ supported on the $(\partial T)^2$ by
\begin{equation}\label{e:lA.3}
\begin{split}
&d\mu_b(\omega_x,\omega_y) = \sum_{\beta>\omega}\frac{\mu(\beta)}{M(\partial\mathcal{S}(\beta))}\,dM(\omega) +\\
&\sum_{\beta_y>\omega_y}\frac{d\mu(\omega_x,\beta_y)}{M_0(\partial\mathcal{S}(\beta_y))}\,dM_0(\omega_y) + \sum_{\beta_x>\omega_x}\frac{d\mu(\beta_x,\omega_y)}{M_0(\partial\mathcal{S}(\beta_x))}\,dM_0(\omega_x)+ d\mu(\omega_x,\omega_y).
\end{split}
\end{equation}
Then the potentials of $\mu$ and $\mu_b$ are equivalent,
\begin{equation}\label{e:lA.3s}
\mathbb{V}^{\mu}(\alpha) \approx \mathbb{V}^{\mu_b}(\alpha),\quad \alpha\in \overline{T}^2.
\end{equation}
\end{theorem}
\begin{proof}
Fix any point $\alpha\in \overline{T}^2$. We have
\begin{equation}\notag
\begin{split}
&\mathbb{V}^{\mu_b}(\alpha) = \int_{(\partial T)^2}d_{T^2}(\alpha\wedge\omega)\,d\mu_b(\omega)=
\int_{(\partial T)^2}d_{T^2}(\alpha\wedge\omega)\sum_{\beta>\omega}\frac{\mu(\beta)}{M(\partial\mathcal{S}(\beta))}\,dM(\omega) + \\
&\int_{(\partial T)^2}d_{T^2}(\alpha\wedge\omega)\sum_{\beta_y>\omega_y}\frac{d\mu(\omega_x,\beta_y)}{M_0(\partial\mathcal{S}(\beta_y))}\,dM_0(\omega_y) +\\
&\int_{(\partial T)^2}d_{T^2}(\alpha\wedge\omega)\sum_{\beta_x>\omega_x}\frac{d\mu(\beta_x,\omega_y)}{M_0(\partial\mathcal{S}(\beta_x))}\,dM_0(\omega_x) +\\
&\int_{(\partial T)^2}d_{T^2}(\alpha\wedge\omega)d\mu(\omega) := (I) + (II) + (III) + (IV).
\end{split}
\end{equation}
By Tonelli's theorem and Lemma \ref{l:A.3.1}
\begin{equation}\notag
(I) = \sum_{\beta\in T^2}\frac{\mu(\beta)}{M(\partial\mathcal{S}(\beta))}\int_{\partial\mathcal{S}(\beta)}d_{T^2}(\alpha\wedge\omega)\,dM(\omega)\approx \sum_{\beta\in T^2}\mu(\beta)d_{T^2}(\alpha\wedge\beta).
\end{equation}
Similarly,
\begin{equation}\notag
\begin{split}
&(II) =\\
&\int_{\partial T_x}d_T(\alpha_x\wedge\omega_x)\sum_{\beta_y\in T_y}\left(\frac{1}{M_0(\partial\mathcal{S}(\beta_y))}\int_{\partial\mathcal{S}(\beta_y)} d_T(\alpha_y\wedge\omega_y)dM_0(\omega_y)\right)d\mu(\omega_x,\beta_y)\approx\\
&\sum_{\beta_y\in T_y}d_T(\alpha_y\wedge\beta_y)\int_{\partial T_x}d_T(\alpha_x\wedge\omega_x)d\mu(\omega_x,\beta_y),
\end{split}
\end{equation}
and
\[
(III) \approx \sum_{\beta_x\in T_x}d_T(\alpha_x\wedge\beta_x)\int_{\partial T_y}d_T(\alpha_y\wedge\omega_y)d\mu(\beta_x,\omega_y).
\]
We arrive at
\begin{equation}\notag
\begin{split}
&V^{\mu_b}(\alpha) \approx \sum_{\beta\in T^2}\mu(\beta)d_{T^2}(\alpha\wedge\beta) + 
\sum_{\beta_y\in T_y}d_T(\alpha_y\wedge\beta_y)\int_{\partial T_x}d_T(\alpha_x\wedge\omega_x)d\mu(\omega_x,\beta_y) +\\
&\sum_{\beta_x\in T_x}d_T(\alpha_x\wedge\beta_x)\int_{\partial T_y}d_T(\alpha_y\wedge\omega_y)d\mu(\beta_x,\omega_y) +
\int_{(\partial T)^2}d_{T^2}(\alpha\wedge\omega)\,d\mu(\omega) =\\
&\int_{T^2}d_{T^2}(\alpha\wedge\tau)\,d\mu(\tau) + \int_{\partial T_x\times T_y}d_{T^2}(\alpha\wedge\tau)\,d\mu(\tau) + \int_{T_x\times \partial T_y}d_{T^2}(\alpha\wedge\tau)\,d\mu(\tau)+\\
&\int_{\partial T_x\times\partial T_y}d_{T^2}(\alpha\wedge\tau)\,d\mu(\tau) = \mathbb{V}^{\mu}(\alpha).
\end{split}
\end{equation}
\end{proof}
\begin{corollary}\label{c:A.3}
Given a Borel set $E\subset \overline{T}^2$ define its boundary projection $\mathcal{S}_b(E)\subset(\partial T)^2$ to be
\[
\mathcal{S}_b(E) = \bigcup_{\beta\in E}\partial\mathcal{S}(\beta).
\]
Then there exists a constant $C>1$ such that
\begin{equation}\label{e:cA.3}
\capp \mathcal{S}_b(E) \leq\capp E \leq C\capp \mathcal{S}_b(E).
\end{equation}
\end{corollary}
\begin{proof}
We start by assuming that $E$ is compact.
The left inequality is trivial, since any function admissible for $E$ is also admissible for $\mathcal{S}_b(E)$. Now let $\mu$ and $\nu$ be the equilibrium measures for $E$ and $\mathcal{S}_b(E)$ respectively. By the definition of $\mu_b$,
\begin{equation}\notag
\begin{split}
&|\mu_b| = \int_{(\partial T)^2}\,d\mu_b = 
\int_{(\partial T)^2}\sum_{\beta>\omega}\frac{\mu(\beta)}{M(\partial\mathcal{S}(\beta))}\,dM(\omega) +\\
&\int_{(\partial T)^2}\sum_{\beta_y>\omega_y}\frac{d\mu(\omega_x,\beta_y)}{M_0(\partial\mathcal{S}(\beta_y))}\,dM_0(\omega_y) +\int_{(\partial T)^2}\sum_{\beta_x>\omega_x}\frac{d\mu(\beta_x,\omega_y)}{M_0(\partial\mathcal{S}(\beta_x))}\,dM_0(\omega_x)+\\
&\int_{(\partial T)^2}d\mu(\omega)= \sum_{\beta\in T^2}\mu(\beta) + \sum_{\beta_y\in T_y}\int_{\partial T_x}d\mu(\omega_x,\beta_y) + \sum_{\beta_x\in T_x}\int_{\partial T_y}d\mu(\beta_x,\omega_y)+\\
&\int_{(\partial T)^2}\,d\mu = \int_{\overline{T}^2}\,d\mu = |\mu|.
\end{split}
\end{equation}
By Theorem \ref{t:th.A.3} and the fact that $\mu$ is an equilibrium measure,
\begin{equation} \label{eq:bcomparable}
|\mu_b| = |\mu| = \int_{\overline{T}^2}\mathbb{V}^{\mu}\,d\mu \approx \int_{\overline{T}^2}\mathbb{V}^{\mu_b}\,d\mu \approx \int_{\overline{T}^2}\mathbb{V}^{\mu_b}\,d\mu_b.
\end{equation}
On the other hand, for every $C\in\mathbb{R}$ we have
\begin{equation}\notag
0 \leq \int_{\overline{T}^2}\mathbb{V}^{\mu_b}\,d\mu_b -2C\int_{\overline{T}^2}\mathbb{V}^{\nu}\,d\mu_b + C^2\int_{\overline{T}^2}\mathbb{V}^{\nu}\,d\nu \leq \int_{\overline{T}^2}\mathbb{V}^{\mu_b}\,d\mu_b -2C|\mu_b| + C^2|\nu|,
\end{equation}
since $\nu$ is an equilibrium measure for $\mathcal{S}_b(E)$ and $\mathbb{V}^{\nu} \geq 1$ quasi-everywhere
on $\mathcal{S}_b(E)\supset\supp\mu_b$. By \eqref{eq:bcomparable} there is a $C > 1$, independent of $E$, such that
\begin{equation}\notag
0 \leq \int_{\overline{T}^2}\mathbb{V}^{\mu_b}\,d\mu_b -C|\mu_b| + C\left(C|\nu|-|\mu_b|\right)\leq C\left(C|\nu|-|\mu_b|\right).
\end{equation}
Therefore
\[
C\capp \mathcal{S}_b(E) = C|\nu| \geq |\mu_b| = |\mu| = \capp E,
\]
and we get the second half of \eqref{e:cA.3}.\par
Given a general set $E$ we exhaust it by compact sets $E_k$ from inside. Then $\lim_{k\rightarrow\infty}\capp E_k = \capp E$ and $\lim_{k\rightarrow\infty}\capp\mathcal{S}_b(E_k) = \capp\mathcal{S}_b(E)$, an we still have \eqref{e:cA.3} by the argument above.
\end{proof}


\section*{Acknowledgments}
We thank the anonymous reviewers for their useful comments.

\bibliographystyle{amsplain}

\begin{dajauthors}
\begin{authorinfo}[nicola]
Nicola Arcozzi\\
Dipartimento di Matematica\\
Università di Bologna "Alma Mater"\\
P.zza di P.ta S. Donato 5, 40126 Bologna, Italy\\
nicola.arcozzi\imageat{}unibo.it 
\end{authorinfo}

\begin{authorinfo}[pavel]
  Pavel Mozolyako\\
  Department of Mathematics and Computer Science\\
  Saint-Petersburg University\\
  14th line V.O., 29B, 199178 Saint-Petersburg, Russia\\
  pmzlcroak\imageat{}gmail\imagedot{}com
\end{authorinfo}

\begin{authorinfo}[kalle]
  Karl-Mikael Perfekt \\
  Department of Mathematical Sciences \\
  Norwegian University of Science and Technology \\
  7491 Trondheim, Norway\\
  karl-mikael\imagedot{}perfekt\imageat{}ntnu\imagedot{}no\\
\end{authorinfo}

\begin{authorinfo}[giulia]
  Giulia Sarfatti\\
  Dipartimento di Ingegneria Industriale e Scienze Matematiche \\
  Università Politecnica delle Marche \\
  via Brecce Bianche 12, 60131 Ancona, Italy \\
  g\imagedot{}sarfatti\imageat{}univpm\imagedot{}it\\
\end{authorinfo}
\end{dajauthors}

\end{document}